\documentclass{article}
\usepackage[T1]{fontenc}
\usepackage[numbers]{natbib}
\usepackage{amsmath,amssymb,amsthm}
\usepackage{bm}
\usepackage[bottom=2cm,right=2cm,top=2cm,left=2cm]{geometry}
\usepackage{xfrac}
\usepackage{graphicx}
\usepackage{tabularx}
\usepackage{float}
\usepackage{xcolor}
\usepackage{nicefrac}
\usepackage{tikz}
\usepackage{dsfont}
\usepackage{enumerate}
\usepackage[colorlinks=true, linkcolor=blue, citecolor=blue, urlcolor=blue]{hyperref}
\usepackage{caption}
\usepackage{subcaption}
\usepackage{mathtools}
\usepackage{paracol}
\usepackage{changes}
\usetikzlibrary{arrows.meta}
\usetikzlibrary{angles, quotes, babel}
\allowdisplaybreaks

\theoremstyle{plain}
\newtheorem{lemma}{Lemma}[section]
\newtheorem{corollary}[lemma]{Corollary}
\newtheorem{theorem}[lemma]{Theorem}
\newtheorem{defandthm}[lemma]{Definition and Theorem}

\theoremstyle{definition}
\newtheorem{problem}[lemma]{Problem}
\newtheorem{remark}[lemma]{Remark}
\newtheorem{definition}[lemma]{Definition}

\newcommand{\fatx}{\bm{x}}
\newcommand{\faty}{\bm{y}}
\newcommand{\fatz}{\bm{z}}
\newcommand{\fatv}{\bm{v}}
\newcommand{\fatw}{\bm{w}}
\newcommand{\fatu}{\bm{u}}

\newcommand{\fate}{\bm{e}}
\newcommand{\fatF}{\bm{F}}
\newcommand{\fatn}{\bm{n}}
\newcommand{\fatD}[1]{\bm{D}(#1)}

\newcommand{\divx}{\mathrm{div}_{\hspace{-0.2mm}\fatx}}
\newcommand{\gradx}{\nabla_{\!\hspace{-0.2mm}\fatx}}
\newcommand{\pt}{\partial_t}
\newcommand{\dx}{\mathrm{d}\fatx}
\newcommand{\dt}{\mathrm{d}t}
\newcommand{\deltat}{\Delta t}
\newcommand{\dsx}{\mathrm{d}S_{\fatx}}
\newcommand{\gammadot}{\dot{\gamma}}
\newcommand{\into}{\int_{\Omega}}

\newcommand{\reals}{\mathbb{R}}
\newcommand{\naturals}{\mathbb{N}}
\newcommand{\tr}{\mathrm{tr}}

\newmuskip\uMskip 
\newmuskip\postskip
\DeclareMathOperator{\uM}{%
  \underline{\textit{M}\mkern-\uMskip}
  \mkern\postskip
}

\begin{document}

\title{Hybrid multiscale method for polymer melts: analysis and simulations}

\author{Ranajay Datta$^\dag$ \and Mária Luká\v{c}ová-Medvi\v{d}ová$^\ddag$
\and Andreas Schömer$^\ddag$ \and Peter Virnau$^\dag$
}

\date{\today}

\maketitle

\centerline{\dag\,Institute of Physics, Johannes Gutenberg University Mainz}
\centerline{Staudingerweg 9, 55128 Mainz, Germany}
\centerline{rdatta@uni-mainz.de, virnau@uni-mainz.de}

\medskip

\centerline{\ddag\,Institute of Mathematics, Johannes Gutenberg University Mainz}
\centerline{Staudingerweg 9, 55128 Mainz, Germany}
\centerline{lukacova@uni-mainz.de, anschoem@uni-mainz.de}

\bigskip

\begin{abstract}
    We model the flow behaviour of dense melts of flexible and semiflexible ring polymers in the presence of walls using a hybrid multiscale approach. Specifically, we perform molecular dynamics simulations and apply the Irving-Kirkwood formula to determine an averaged stress tensor for a macroscopic model. For the latter, we choose a Cahn-Hilliard-Navier-Stokes system with dynamic and no-slip boundary conditions. We present numerical simulations of the macroscopic flow that are based on a finite element method. In particular, we present detailed proofs of the solvability and the energy stability of our numerical scheme. Phase segregation under flow between flexible and semiflexible rings, as observed in the microscopic simulations, can be replicated in the macroscopic model by introducing effective attractive forces.
\end{abstract}

\section{Introduction}
In this paper, we present a phenomenological description of the flow behaviour of dense blends of flexible and semiflexible ring polymers using a hybrid multiscale approach. In particular, we present a finite element method for our model and give detailed proofs for its solvability and energy stability. Our considerations are inspired by a molecular dynamics (MD) experiment presented in \citep[Section 4]{Datta_2023}. There, the authors found that a blend consisting of equal proportions of flexible and semiflexible ring polymers in a periodic box remains mixed in equilibrium - even after the introduction of walls at the bottom and the top of the box. However, when subjected to flow (with the walls still being present), the same blend started to phase segregate (see Section \ref{sec_microscopic} for some more details). Therefore, the authors concluded that in microfluidic devices subjection to flow can be a means of separating ring polymers with similar chemical properties according to stiffness. Previous studies involving mixtures of ring and linear polymers \cite{Weiss_Nikoubashman_2019}, star polymers and linear chains \cite{Srivastava_Nikoubashman_2018}, as well as a more recent investigation involving a dense binary blend of small flexible knotted and unknotted ring polymers \cite{Datta_Virnau_2025}, have similarly demonstrated that subjecting a blend to channel flow can induce segregation of the components.

To model the flow on macroscopic scales, we use a continuum model that incorporates rheological information from MD simulations.
More specifically, we apply the Irving-Kirkwood formula to determine an averaged stress tensor from MD. The viscosity in the resulting stress tensor depends on the shear rate $\gammadot$ and on the mixing ratio of the polymers involved. It is subsequently passed to a Cahn-Hilliard-Navier-Stokes system describing the macroscopic flow. A similar strategy has been pursued in \citep{Datta}, where the authors analysed the flow behaviour of melts of chain polymers of different stiffness. The results they present in \citep[Section 5]{Datta} demonstrate that the rheological information derived from MD has a decisive influence on the macroscopic flow. We note that in \citep{Datta} only melts of single polymer types were considered. Consequently, the macroscopic model in \citep{Datta} does not have to be able to capture segregation phenomena, since it consists of the (incompressible) Navier-Stokes equations only.

To capture the segregation phenomena observed in the MD experiment in \citep{Datta_2023} described above, we couple the Navier-Stokes equations to the Cahn-Hilliard equation. The resulting Cahn-Hilliard-Navier-Stokes system has been studied analytically and numerically, see, e.g., \citep{Boyer} for the existence and uniqueness of global weak solutions, \citep{Brunk} for a finite element scheme with order-optimal convergence rates and \cite{GuillenGonzalez} for structure-preserving finite element schemes. In fact, even the more general AGG model proposed by Abels, Garcke and Grün in \cite{AGG} has been studied extensively, see, e.g., \cite{ADG} and \cite{ADG2} for the existence of weak solutions, \cite{Giorgini} for the local-in-time existence of strong solutions in three space dimensions and \cite{AGGio} for a weak-strong uniqueness result as well as global-in-time existence of strong solutions in two space dimensions. Numerical schemes for the AGG model can be found in \cite{GuillenGonzalez2} and \cite{Grün}. We note that in all these papers the viscosity solely depends on the concentration variable $\phi$ (i.e. on the mixing ratio of the two components).

For an analysis of incompressible fluids with shear-dependent viscosity, including existence results for the corresponding governing equations, we refer to \cite{Blechta}. In this work, the reader can find references to other relevant works in this context such as \cite{Bulicek1, Bulicek2, Bulicek3}. Some rigorous numerical results in this direction are also available, see, e.g., \cite{Berselli}, where order-optimal error estimates for a numerical scheme are proven. In all of the five aforementioned works, single-component flows are considered. An existence result for a model describing non-Newtonian two-phase flow with a shear- and composition-dependent stress tensor can be found in \cite{Abels}.

In \cite{Abels}, Neumann boundary conditions are applied to the Cahn-Hilliard part of the model. To obtain the correct segregation behaviour in our situation, we apply dynamic boundary conditions to the Cahn-Hilliard equation. The exact form of our boundary conditions originates from the physics community, see, e.g., \citep{Fischer}. Mathematically, the Cahn-Hilliard equation with boundary conditions similar to ours has been studied, for instance, in \citep{Goldstein} and \citep{Petcu}. Also the AGG model has been studied with dynamic boundary conditions, see, e.g., \cite{Gal, Knopf} for the existence of weak solutions in this case.

The paper is organised as follows. In Section \ref{sec_microscopic}, we present the framework for the MD simulations and their outcomes. Section \ref{sec_macroscopic} is devoted to the numerical scheme for the Cahn-Hilliard-Navier-Stokes system, the proof of its properties and the presentation of numerical results including a convergence study. In Appendix \ref{sec_fitting}, we present the result of a least-squares fitting of the viscosity data obtained from the MD simulations to the Carreau-Yasuda model. Appendices \ref{sec_inequalities}--\ref{sec_schaefer} provide the mathematical tools necessary for the analysis of our numerical scheme in Section \ref{sec_macroscopic}.

\section{Microscopic model and simulation techniques}
\label{sec_microscopic}

In this study, oligomers are depicted as coarse-grained bead-spring chains, consistent with the framework developed by Kremer and Grest \cite{Kremer1990}. Interactions between bead pairs are modelled using the Weeks-Chandler-Andersen (WCA) potential \cite{Weeks_1971} which is repulsive in nature:
\begin{equation}
V_{\rm WCA}(r) = \left\{
\begin{array}{cl}
4\varepsilon \left[ \left( \dfrac{\sigma}{r} \right)^{12} - \left( \dfrac{\sigma}{r} \right)^6 + \dfrac{1}{4} \right] & \quad \text{if} \;\; r \leq 2^{1/6}\sigma, \\[4mm]
0 & \quad \text{if} \;\; r > 2^{1/6}\sigma,
\end{array}
\right.
\end{equation}
where \( \sigma \) denotes the diameter of the bead and \( \varepsilon \) the strength of the WCA interactions. These parameters are adopted as the units of length and energy, respectively, and all subsequent quantities are expressed in dimensionless units.

Along with the repulsive WCA potential, consecutive beads in a chain are additionally bonded via the finitely extensible non-linear elastic (FENE) potential \cite{FENE}:
\begin{equation}
V_{\rm FENE}(r) = -\frac{1}{2} K R^2 \ln \left[ 1 - \left( \frac{r}{R} \right)^2 \right],
\end{equation}
where the spring constant \( K \) is set to 30, and the maximum bond extension \( R \) is 1.5. 

To incorporate semiflexibility into the chains, a bending energy term is introduced as follows:
\begin{equation}
V_\theta = \kappa (1 + \cos \theta),
\end{equation}
where \( \theta \) is the angle formed by three consecutive beads, and \( \kappa \) is the stiffness parameter. This form of bending potential is inspired by the Kratky-Porod model \cite{Kratky_49, DoiEdwards, RubinsteinColby}, which is widely used for semiflexible polymers. For large values of \( \kappa \), the persistence length \( \ell_{\rm p} \) is approximately \( \ell_{\rm p}/\ell_{\rm b} \approx \kappa \), where the bond length \( \ell_{\rm b} \approx \sigma \).

Our dense oligomer melts are characterized by a number density $\rho=0.8$, and system sizes are $15^3$ unless otherwise specified. We perform non-equilibrium MD simulations using the LAMMPS \cite{Plimpton1995} package to shear the oligomer melts. Shear was imposed along the \( x \)-axis, the flow (f) direction, with the \( y \)- and \( z \)-axes being the gradient (g) and the vorticity (v) directions, respectively. In our setup, shear is introduced by overlaying a velocity gradient onto the particles' thermal velocities via the employment of the SLLOD equations \cite{Evans1984, Ladd1984, Tuckerman1997, Evans2008} of motion. These equations alter the equations of motion of the particles by the addition of a velocity component that varies with height, providing a standard microscopic framework for modelling shear. To maintain a constant temperature of \( T = 1 \), LAMMPS couples the SLLOD equations with a Nos\'{e}-Hoover thermostat \cite{Evans1985, Tuckerman1997}. Integration of the equations of motion is carried out via the velocity Verlet algorithm. For the implementation of shear, LAMMPS used a non-orthogonal simulation box characterized by periodic boundary conditions, which is continuously deformed in a manner consistent with the imposed shear rate ($\dot{\gamma}$) \cite{Evans1979, Hansen1994}. This approach is equivalent to the imposition of the Lees-Edwards boundary conditions \cite{Evans2008, Todd2017}. 

The non-diagonal component of the stress tensor \( \sigma_{xy} \) is calculated via Irving-Kirkwood expression \cite{Irving1950, allen-tildesley-87}:
\begin{equation}
\sigma_{xy} = -\frac{1}{V} \left[ \sum_{i=1}^N m_i v_{i,x} v_{i,y} + \sum_{i=1}^N \sum_{j>i}^N \left( r_{ij,x} f_{ij,y} \right) \right],
\end{equation}
where \( m_i \) and \( \mathbf{v}_i \) are the mass and peculiar velocity of the \( i^\text{th} \) particle, \( V \) is the system volume, and \( \mathbf{r}_{ij} \) and \( \mathbf{f}_{ij} \) represent the distance and force vectors between particle pairs. 

Upon the estimation of the stress tensor, the shear viscosity \( \eta(\dot{\gamma}) \) is computed via the following relation:
\begin{equation}
\eta = \frac{\sigma_{xy}}{\dot{\gamma}},
\end{equation}

To establish a reference value for viscosity in the limit of vanishing shear rates ($\dot{\gamma}\rightarrow{0}$), the Green-Kubo relation is employed to calculate the zero-shear viscosity:
\begin{equation}
\eta_{\rm GK} = \frac{V}{k_B T} \int_0^\infty \langle \sigma_{xy}(t) \sigma_{xy}(0) \rangle \;\mathrm{d}t,
\end{equation} Here, \( k_B \) is the Boltzmann constant.

In addition, we examine pressure-induced flow of a dense ($\rho=0.8$) binary blend of flexible  (\( \kappa = 0 \)) and stiffer (\( \kappa = 10 \)) ring polymers within a slit-like channel, with polymers of both components having a degree of polymerization $N=15$. Dimensions of the channel are \( 15 \times 30 \times 15 \). Two layers of atoms, structured as face-centred cubic lattices with a density of 4.0 \cite{Duan_2015}, are positioned below and above the channel (which results in a slight compression of the total volume of the polymers subjected to flow). Wall particles are fixed and mutually non-interacting. A WCA potential with diameter \( \sigma_{\mathrm{wp}} = 0.85 \)~\cite{Duan_2015} and interaction strength \( \varepsilon_{\mathrm{wp}} = 4.0 \) is utilised to model the interactions between wall atoms and polymer beads. A dissipative particle dynamics (DPD) thermostat \cite{Soddemann_Duenweg_2003, Pastorino_Binder_2007, Binder_DPD_2011} maintains a uniform temperature of \( T = 1 \) within the channel. The DPD thermostat is characterized by a friction coefficient \( \lambda_{\rm DPD} = 4.5 \) and a cutoff distance (related to the dissipative and random forces associated with the DPD thermostat) \( r_{\rm c,DPD} = 2 \times 2^{1/6} \) \cite{Pastorino_confproc_2015, Pastorino_2014}. The cutoff pertaining to the conservative WCA interactions between the polymer beads is $2^{1/6}$. A constant force $f_x=0.095$ is applied in the flow direction ($x$-direction) to all polymer beads. The nature of the interaction between the wall particles and the polymer beads along with the structure of the channel walls ensure that the resultant flow profile is characterized by no-slip boundary conditions at the walls, with flow velocity being maximum at the channel centre.

\subsection{Results}

\begin{figure}[ht!]

      \includegraphics[width=6.8in]{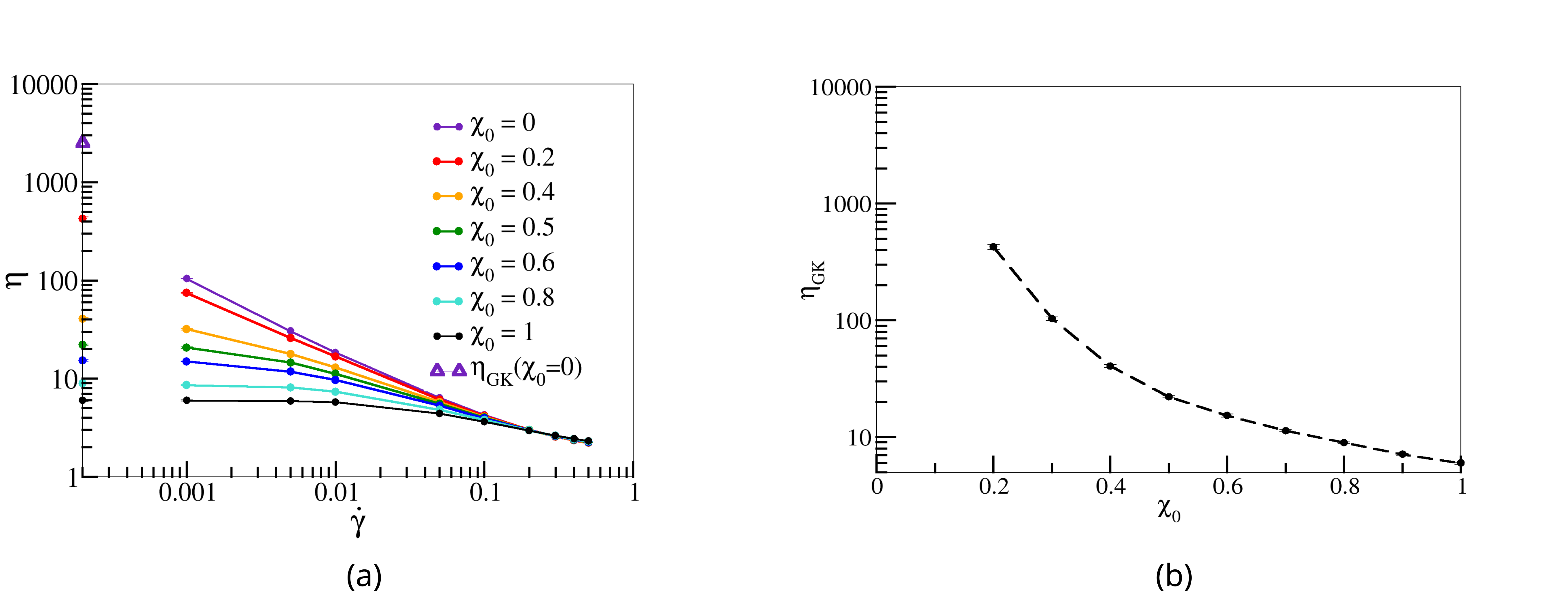}

\caption{
\textbf{(a)} Shear viscosity $\eta$ as function of shear rate $\dot{\gamma}$ for dense ($\rho = 0.8$) binary blends of flexible ($\kappa=0$) and stiffer ($\kappa=10$) ring polymers, corresponding to different proportions ($\chi_0$) of flexible rings . Each ring consists of $N = 15$ monomers. Corresponding zero-shear viscosities, $\eta_\text{GK}$ (calculated by the Green-Kubo relation) are shown on the $y$-axis. Note that the value of  $\eta_\text{GK}$ corresponding to $\kappa=10$ as exhibited on the $y$-axis has been obtained from a least square fitting and not from MD simulations, cf. Appendix \ref{sec_fitting}.
\textbf{(b)} Zero-shear viscosity from the Green-Kubo relation $\eta_\text{GK}$ as a function of $\chi_0$ in a binary mixture of ring polymers consisting of flexible and stiff ($\kappa=10$) rings. A box size of $15\times15\times15$ was considered for all the simulations. Figure~\ref{fig1}b has been adapted with permission from R. Datta, F. Berressem, F. Schmid, A. Nikoubashman, and P. Virnau. Viscosity of Flexible and Semiflexible
Ring Melts: Molecular Origins and Flow-Induced Segregation. Macromolecules, 56(18):7247–7255, 2023. Copyright 2023 American Chemical Society \cite{Datta_2023}. All lines are guides for the eye.
}
\label{fig1}
\end{figure}

Reference \cite{Datta_2023} studies binary mixtures of polymers at equilibrium and under flow. Their results show that an equimolar binary blend of flexible ($\kappa=0$) and stiffer ($\kappa=10$) ring polymers at $\rho=0.8$ and $T=1$ remains homogeneously mixed at equilibrium ($N=15$ for both species of polymers). 
As shown in Figure~\ref{fig1}a, such binary blends, subjected to shear, exhibit shear-thinning, irrespective of the proportion of flexible rings in the mixture. There is, however, a steep decrease in the zero-shear viscosity of flexible and stiffer ring mixtures, with an increase in the proportion of flexible rings ($\chi_0$).   

While a $50/50$ binary blend of flexible and stiffer ($\kappa=10$) ring polymers remains homogeneously mixed in equilibrium, Reference \cite{Datta_2023} describes a method by which such a blend can be made to undergo phase segregation. According to this method, such a mixture is confined within particle-based walls. The structure of the walls and the interactions between the walls and the polymer particles have been tuned such that no-slip boundary conditions are maintained at the channel walls during flow. At equilibrium, the blend does not phase segregate even in the presence of walls (Figure~\ref{fig2}a). Still, when subjected to flow via the application of a constant force $f_x=0.095$ on all polymer beads, the mixture phase segregates, with the stiffer rings prefer to accumulate near the channel center, while the flexible rings preferring to migrate towards the channel walls (Figures~\ref{fig2}b and \ref{fig2}c).

\begin{figure}[ht!]

      \includegraphics[width=6.6in]{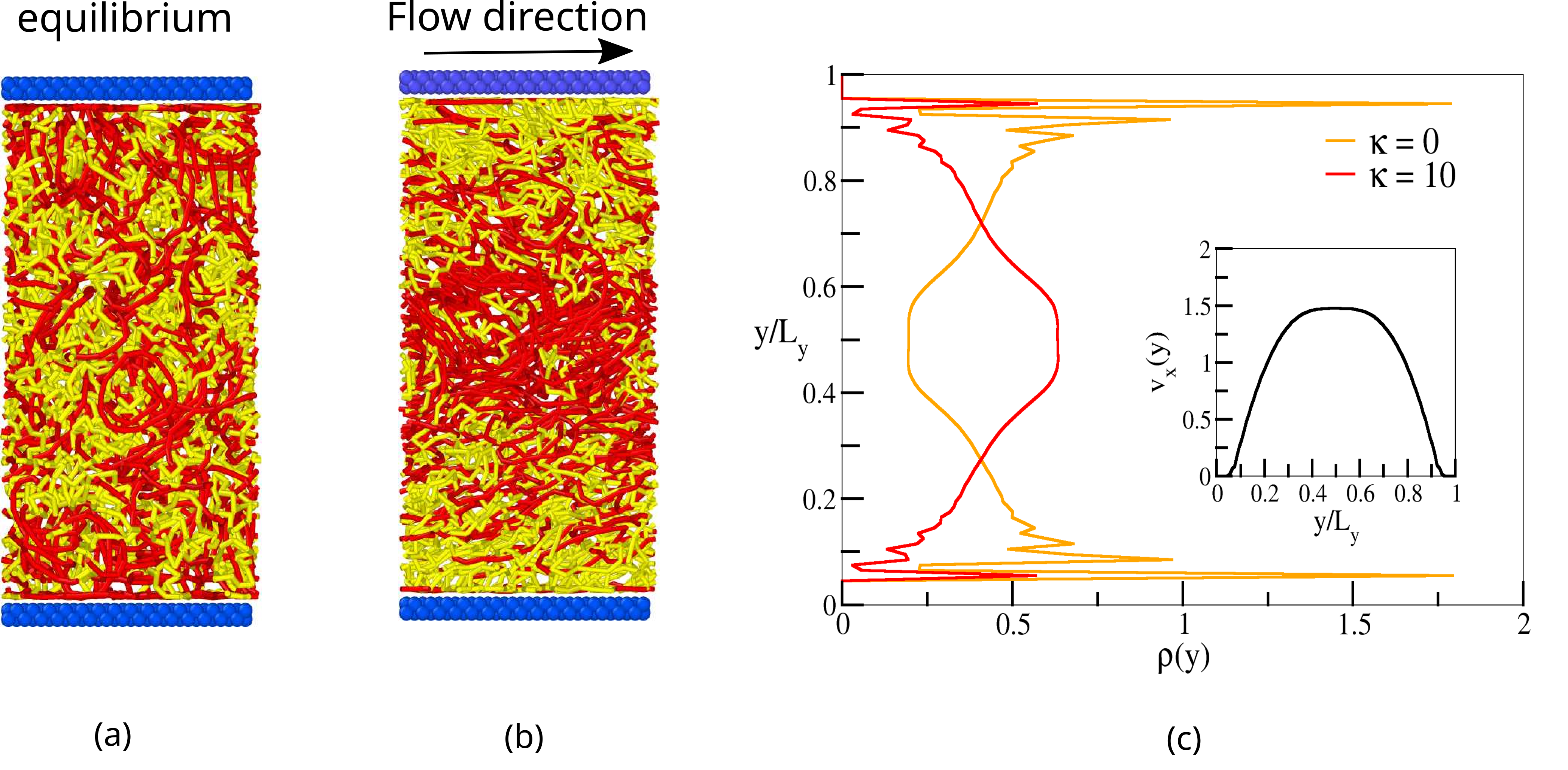}

\caption{\textbf{(a)} An equilibrated binary mixture of flexible ($\kappa=0$, yellow) and stiffer ($\kappa=10$, red) rings at $\chi_0=0.5$, confined within a channel bounded by particle-based walls (blue). \textbf{(b)} The same mixture under flow, induced by applying a constant force $f_x=0.095$ along the $x$-axis to all particles. \textbf{(c)} Density profiles of the respective components across the channel cross-section. (The inset depicts the velocity profile in the flow direction, confirming the implementation of no-slip boundary conditions.) Adapted with permission from R. Datta, F. Berressem, F. Schmid, A. Nikoubashman, and P. Virnau. Viscosity of Flexible and Semiflexible
Ring Melts: Molecular Origins and Flow-Induced Segregation. Macromolecules, 56(18):7247–7255, 2023. Copyright 2023 American Chemical Society} \cite{Datta_2023}.
\label{fig2}
\end{figure}

\section{The hybrid macroscopic-microscopic model} \label{sec_macroscopic}
This section is devoted to a hybrid model for the mixture flow that combines macroscopic fluid dynamics with the results of the microscopic model/MD in order to describe complex rheology. The macroscopic model for the description of phase separation effects is given by the following Cahn-Hilliard-Navier-Stokes (CHNS) system
\begin{alignat}{2}
    \pt\phi + \divx(\phi\fatu) &= \divx(M(\phi)\gradx\mu) && \qquad \text{in $[0,T]\times\Omega$,} \label{CHNS_phi} \\[2mm]
    \mu &= - \gamma\Delta_{\fatx}\phi + f'(\phi) && \qquad \text{in $[0,T]\times\Omega$,}\label{CHNS_mu} \\[2mm]
    \pt\fatu + (\fatu\cdot\gradx)\fatu + \gradx p + \phi\gradx\mu &= \divx[\eta(\gammadot,\phi)\fatD{\fatu}] + \fatF && \qquad \text{in $[0,T]\times\Omega$,} \label{CHNS_u} \\[2mm]
    \divx(\fatu) &= 0 && \qquad \text{in $[0,T]\times\Omega$,} \label{CHNS_div} \\[2mm]
    \fatu &= \bm{0} && \qquad \text{on $[0,T]\times\Gamma$,} \label{CHNS_bdry_u} \\[2mm]
    \partial_{\fatn}\mu &= 0 && \qquad \text{on $[0,T]\times\Gamma$,} \label{CHNS_bdry_mu} \\[2mm]
    \pt\phi - s\Delta_{\Gamma}\phi + g'(\phi) + \partial_{\fatn}\phi &= 0 && \qquad \text{on $[0,T]\times\Gamma$,} \label{CHNS_bdry_phi} \\[2mm]
    (\mu,\phi,\fatu,p)(t,\fatx+L_j\fate_j) &= (\mu,\phi,\fatu,p)(t,\fatx)  && \qquad \text{for all $(t,\fatx)\in [0,T]\times \Gamma_{j}$, $j\in\{1,\dots,d-1\}$,} \label{CHNS_periodic} \\[2mm]
    (\phi, \fatu)(0,\cdot) &= (\phi_0, \fatu_0) && \qquad \text{in $\Omega$.} \label{CHNS_initial}
\end{alignat}
Here, $\fate_j$ denotes the $j$-th unit vector in $\reals^d$, $T>0$ is some final time and 
\begin{align}
    \Omega = \prod_{k\,=\,1}^{d} (0,L_k) \qquad L_k > 0, \;\; k\in\{1,\dots,d\}, \qquad d\in\{2,3\}. \notag
\end{align}
The boundary portions of $\Omega$ mentioned above are given by
\begin{align}
    \Gamma = \prod_{k\,=\,1}^{d-1} (0,L_k) \times \{0,L_d\}, \qquad \Gamma_j = \partial\Omega \cap \{x_j = 0\},
    \;\;j\in\{1,\dots,d-1\}. \notag
\end{align}

Thus, \eqref{CHNS_periodic} encodes periodicity in all spatial directions except for the $x_d$-direction. By $\partial_{\fatn}$, $\Delta_{\Gamma}$ and $\nabla_{\Gamma}$ we denote the outward normal derivative, the Laplace-Beltrami operator and the tangential gradient on $\Gamma$, respectively. Note that in our context, the latter reduces to 
\begin{align}
    \Delta_\Gamma = \sum_{k\,=\,1}^{d-1} \partial^{2}_{x_k}, \qquad \nabla_\Gamma = \sum_{k\,=\,1}^{d-1} \fate_k\partial_{x_k}. \notag
\end{align}
Moreover, $\phi$ denotes the volume fraction of the stiffer polymer, $\bm{u}$ the mixture velocity and $p$ the pressure. Furthermore,
\[\fatD{\fatu} = \frac{1}{2}\left(\gradx\fatu + (\gradx\fatu)^T\right) \qquad \text{and} \qquad \gammadot=\gammadot(\fatD{\fatu})=\sqrt{2\fatD{\fatu}:\fatD{\fatu}}\]
denote the symmetric velocity gradient and the shear rate, respectively. The quantity $\mu$ can be interpreted as the chemical potential of the flexible polymer as it equals the variational derivative of the free energy of the system
\begin{align}
    E(\phi,\fatu) = \into \left(\frac{1}{2}\,|\fatu|^2 + f(\phi) + \frac{\gamma}{2}\,|\gradx\phi|^2 \right)\dx + \int_{\Gamma} \left(g(\phi) + \frac{s}{2}\,|\nabla_{\Gamma}\phi|^2\right)\dsx \label{energy}
\end{align}
with respect to $\phi$, i.e., 
\[\frac{\delta E}{\delta \phi} = - \gamma\Delta_{\fatx}\phi + f'(\phi) = \mu.\]
Complex non-Newtonian effects arising in the polymeric mixture are modelled by the viscosity $\eta$ that is obtained from the MD simulations. Moreover, $\fatF$ denotes an external body force, $f,g$ are potential functions and $M$ is the so-called mobility function. \\

\noindent \textbf{Assumptions.} We make the following assumptions throughout this section. \hypertarget{A1}{}
\begin{itemize}
    \item[(A1)]{$\gamma,s>0$ are positive constants. \hypertarget{A2}{}
    }
    \item[(A2)]{$\fatF$ is constant in time and space. \hypertarget{A3}{}
    }
    \item[(A3)]{$M\in C(\reals)$ and $M(\phi)\geq 0$ for all $\phi\in\reals$. \hypertarget{A4}{}
    }
    \item[(A4)]{$\eta\in C([0,\infty)\times\reals)$ and there exist $\overline{\eta}\geq\underline{\eta}>0$ such that $\underline{\eta}\leq\eta(\gammadot,\phi)\leq \overline{\eta} $ for all $(\gammadot,\phi)\in[0,\infty)\times\reals$. \hypertarget{A5}{}
    }
    \item[(A5)]{There exist convex functions $f_{\mathrm{vex}}, g_{\mathrm{vex}}\in C^2(\reals)$ and concave functions $f_{\mathrm{cav}}, g_{\mathrm{cav}}\in C^2(\reals)$ such that $(f,g) = (f_{\mathrm{vex}} + f_{\mathrm{cav}}, g_{\mathrm{vex}} + g_{\mathrm{cav}})$.
    Moreover, there are numbers $q>1$, $c\geq 0$, $c'>0$ such that for all $\phi\in\reals$
    \begin{gather*}
        \min\{f(\phi), g(\phi)\} \geq - c, \qquad \max\{|f'_{\mathrm{vex}}(\phi)|,|f'_{\mathrm{cav}}(\phi)|\} \leq c'(1 + |\phi|^3), \\[2mm] 
        \max\{|g'_{\mathrm{vex}}(\phi)|,|g'_{\mathrm{cav}}(\phi)|\} \leq c'(1 + |\phi|^{q}).
    \end{gather*}
    }
\end{itemize}

\subsection{General notation}
We introduce some general notation. \\

\noindent \textbf{Vector notation and boxes.} For vectors $\fatx\in\reals^d$ we typically use bold lower case symbols. For vectors in $\fatx'\in\reals^{d-1}$ we shall add a prime symbol to the notation. Analogously, we shall write
\begin{align}
    (\bm{0},\bm{L}) = \prod_{k\,=\,1}^d (0,L_k) \qquad \text{and} \qquad (\bm{0}',\bm{L}') = \prod_{k\,=\,1}^{d-1} (0,L_k). \notag
\end{align}

\noindent \textbf{Norms and function spaces.} The maximum norm $||\cdot||_\infty$ and the $p$-norms $||\cdot||_p$,  $p\in[1,\infty)$, of numbers, vectors and matrices read as follows:
\begin{align*}
    ||\mathcal{A}||_\infty = \left\{
    \begin{array}{cl}
        |\mathcal{A}| & \text{if $\mathcal{A}\in\reals$,} \\[2mm]
        \max\{|a_i|\,|\,i=1,\dots,d\} & \text{if $\mathcal{A}\in\reals^d$,} \\[2mm]
        \max\{|a_{ij}|\,|\,i,j=1,\dots,d\} & \text{if $\mathcal{A}\in\reals^{d\times d}$,} 
    \end{array}
    \right. \qquad 
    ||\mathcal{A}||_p = \left\{
    \begin{array}{cl}
        |\mathcal{A}| & \text{if $\mathcal{A}\in\reals$,} \\[2mm]
        \left(\sum_{i\,=\,1}^d |a_i|^p\right)^{1/p} & \text{if $\mathcal{A}\in\reals^d$,} \\[2mm]
        \left(\sum_{i,j\,=\,1}^d |a_{ij}|^p\right)^{1/p} & \text{if $\mathcal{A}\in\reals^{d\times d}$.} 
    \end{array}
    \right. 
\end{align*}
For the scalar product of numbers, vectors and matrices we write 
\begin{gather*}
    \langle\mathcal{A},\mathcal{B}\rangle = \left\{
    {\arraycolsep=2pt
    \begin{array}{rlll}
        & & \mathcal{A}\mathcal{B} & \text{if $\mathcal{A},\mathcal{B}\in\reals$,}  \\[2mm]
        \mathcal{A}\cdot\mathcal{B} &=& \sum_{i\,=\,1}^d a_{i}b_{i} & \text{if $\mathcal{A},\mathcal{B}\in\reals^d$,} \\[2mm]
        \mathcal{A}:\mathcal{B} &=& \sum_{i,j\,=\,1}^d a_{ij}b_{ij} & \text{if $\mathcal{A},\mathcal{B}\in\reals^{d\times d}$.}
    \end{array}
    }
    \right. 
\end{gather*}
$C(\overline{\Omega})$ denotes the space of continuous functions defined on $C(\overline{\Omega})$ endowed with the maximum norm given by
\begin{align}
    ||\mathcal{F}||_{C(\overline{\Omega})} = \max_{\overline{\Omega}}\{||\mathcal{F}||_\infty\}
    \notag 
\end{align}
for all $\mathcal{F}\in C(\overline{\Omega})\cup C(\overline{\Omega})^d\cup C(\overline{\Omega})^{d\times d}$. The usual Lebesgue spaces are denoted by $L^p(\Omega), p\in[1,\infty]$, and the $L^2$-based Sobolev space by $H^1(\Omega)=W^{1,2}(\Omega)$. They are endowed with the following norms:
\begin{align}
    ||\mathcal{F}||_{L^p(\Omega)} = \left\{\begin{array}{cl}
        \left({\displaystyle\int_\Omega} ||\mathcal{F}||^p_p \notag\;\dx\right)^{1/p} & \text{if $p<\infty$,} \\[4mm]
        \underset{\Omega}{\mathrm{ess\,sup}}\{||\mathcal{F}||_\infty\} & \text{if $p=\infty$,}
    \end{array}
    \right. \qquad ||\mathcal{G}||_{H^1(\Omega)} = \left(||\mathcal{G}||^2_{L^2(\Omega)} + ||\gradx\mathcal{G}||^2_{L^2(\Omega)}\right)^{1/2}
\end{align}
for all $\mathcal{F}\in L^p(\Omega)\cup L^p(\Omega)^d\cup L^p(\Omega)^{d\times d}$ and $\mathcal{G}\in H^1(\Omega)\cup H^1(\Omega)^d$. For an introduction to these spaces we refer to \citep{Adams} (see also \citep[Section 2]{Necas} for the Sobolev spaces). Due to \eqref{CHNS_periodic}, we need the periodic subspaces $H^1_{\mathrm{per}}(\Omega)$ and $H^1_{0,\mathrm{per}}(\Omega)$ of $H^1(\Omega)$ that are given by
\begin{align}
    H^1_{\mathrm{per}}(\Omega) &= \Big\{f\in H^1(\Omega)\,\Big|\,\mathrm{tr}[f](\fatx+L_j\fate_j) = \mathrm{tr}[f](\fatx) \;\;\text{for a.e. $\fatx\in\Gamma_j$, $j\in\{1,\dots,d-1\}$}\Big\}, \notag \\[2mm]
    H^1_{0,\mathrm{per}}(\Omega) &= \Big\{f\in H^1_{\mathrm{per}}(\Omega)\,\Big|\,\mathrm{tr}[f]\big|_\Gamma = 0\Big\}. \notag
\end{align}
For the trace operator $\mathrm{tr}:H^1(\Omega)\to L^q(\partial\Omega)$ we refer to Theorem \ref{thm_trace}. An introduction to the spaces $L^p(\partial\Omega)$, $p\in[1,\infty]$, defined over the boundary of $\Omega$ and to the boundary integral (with surface element $\dsx$) can be found in \citep[Sections 2.4 and 3.1]{Necas} and \citep[Sections 2.1 and 2.2]{Mitrea}. We are, however, in a very simple situation since we only need to deal with flat boundary parts. For the Lebesgue and Sobolev spaces over $\Gamma$ we have the following simple definitions:     
\begin{align}
    L^p(\Gamma) &= \Big\{f:\Gamma\to\reals\,\Big|\,f(\cdot,0),f(\cdot,L_d)\in L^p((\bm{0}',\bm{L}'))\Big\}, \qquad p\in[1,\infty],\notag \\[2mm]
    H^1(\Gamma) &= \Big\{g:\Gamma\to\reals\,\Big|\,g(\cdot,0),g(\cdot,L_d)\in H^1((\bm{0}',\bm{L}'))\Big\}. \notag
\end{align}
The corresponding norms read
\begin{align}
    ||\mathcal{F}||_{L^p(\Gamma)} &= 
    {\arraycolsep=0pt
    \left\{\begin{array}{rll}
        \left({\displaystyle\int_{\Gamma}} ||\mathcal{F}||^p_p \;\dsx\right)^{1/p} &\;= \left({\displaystyle\int_{(\bm{0}'\bm{L}')}} \Big[||\mathcal{F}(\cdot,0)||^p_p + ||\mathcal{F}(\cdot,L_d)||^p_p\Big]\,\dx'\right)^{1/p} & \quad \text{if $p<\infty$,} \\[6mm]
        \underset{\Gamma}{\mathrm{ess\,sup}}\big\{||\mathcal{F}||_\infty\big\} &\;= \max\!\left\{\underset{(\bm{0}'\bm{L}')}{\mathrm{ess\,sup}}\big\{||\mathcal{F}(\cdot,0)||_\infty\big\}, \underset{(\bm{0}'\bm{L}')}{\mathrm{ess\,sup}}\big\{||\mathcal{F}(\cdot,L_d)||_\infty\big\}\right\} & \quad \text{if $p=\infty$,}
    \end{array}
    \right. 
    } \notag \\[2mm]
    ||\mathcal{G}||_{H^1(\Gamma)} &= \left(||\mathcal{G}||^2_{L^2(\Gamma)} + ||\nabla_{\Gamma}\mathcal{G}||^2_{L^2(\Gamma)}\right)^{1/2} = \left(||\mathcal{G}||^2_{L^2(\Gamma)} + ||\nabla_{\!\fatx'}\,\mathcal{G}(\cdot,0)||^2_{L^2((\bm{0}',\bm{L}'))} + ||\nabla_{\!\fatx'}\,\mathcal{G}(\cdot,L_d)||^2_{L^2((\bm{0}',\bm{L}'))}\right)^{1/2} \notag
\end{align}
for all $\mathcal{F}\in L^p(\Gamma)\cup L^p(\Gamma)^d\cup L^p(\Gamma)^{d\times d}$ and $\mathcal{G}\in H^1(\Gamma)$. In the sequel, we would like to apply the trace inequality (cf. Theorem \ref{thm_trace}) to estimate the norm in $L^p(\Gamma)$ of certain functions. Therefore, we note that if $f\in L^p(\partial\Omega)$ then $f\in L^p(\Gamma)$ and 
\begin{align}
    ||f||_{L^p(\Gamma)} \leq ||f||_{L^p(\partial\Omega)}. \notag
\end{align}
If $\mathcal{F},\mathcal{G}$ are $\reals$-, $\reals^d$- or $\reals^{d\times d}$-valued functions defined on $\Omega$ or $\Gamma$, we write
\begin{align*}
    \langle\mathcal{F},\mathcal{G}\rangle_\Omega = \int_\Omega\langle\mathcal{F},\mathcal{G}\rangle\;\dx \qquad \text{or} \qquad 
    \langle\mathcal{F},\mathcal{G}\rangle_{\Gamma} = \int_{\Gamma}\langle\mathcal{F},\mathcal{G}\rangle\;\dsx,
\end{align*}
respectively, provided the integral expressions involved are well-defined (cf. Theorem \ref{thm_hölder}). Finally, if $(X,||\cdot||_X)$ is a Banach space, the corresponding Lebesgue-Bochner spaces over the interval $(0,T)$ are denoted by $L^p(0,T;X)$, $p\in[1,\infty]$. These spaces are endowed with the following norms:
\begin{align*}
    ||\mathcal{F}||_{L^p(0,T;X)} = \left\{\begin{array}{cl}
        \left({\displaystyle\int_0^{\,T}} ||\mathcal{F}||^p_X \notag\;\dt\right)^{1/p} & \text{if $p<\infty$,} \\[4mm]
        \underset{(0,T)}{\mathrm{ess\,sup}}\{||\mathcal{F}||_X\} & \text{if $p=\infty$}
    \end{array}
    \right.
\end{align*}
for all $\mathcal{F}\in L^p(0,T;X)$. For an introduction to the Lebesgue-Bochner spaces we refer to \citep[Chapter II]{Diestel}. \\

\subsection{The numerical scheme}
To solve the CHNS model \eqref{CHNS_phi}--\eqref{CHNS_initial} numerically, we employ a finite element method. The precise choice of the method is inspired by the numerical schemes presented in \citep[Sections 2.5, 2.6, 3]{Brunk} and \citep[Section 6]{Petcu}. \\

\noindent \textbf{Time discretization.} The admissible time steps are the elements of the set $\Theta = \{T/m\,|\,m\in\naturals\}$. For a given time step $\deltat$ we set $N_{\deltat} = T/\deltat$, $t_{k,\deltat} = k\deltat$ for $k\in\{0,\dots,N_{\deltat}\}$,
\[I_{k,\deltat}=\left\{\begin{array}{cl}
[0,t_{1,\deltat}] & \text{if $k=1$,} \\[2mm]
(t_{k-1,\deltat}, t_{k,\deltat}] & \text{if $k\in\{2,\dots,N_{\deltat}\}$,}
\end{array}\right. \qquad \mathcal{I}_{\deltat} = \{I_{k,\deltat}\,|\,k\in\{1,\dots,N_{\deltat}\}\}.\]
To simplify the notation, we frequently omit the subscript $\deltat$ when there is no risk of confusion. Moreover, given an admissible time step $\deltat$ and a vector space $V$, we introduce the spaces 
\begin{align*}
    \Pi^0(\mathcal{I}_{\deltat}; V) &= \left\{f:[0,T]\to V\left|\,\exists\, \{f^k\}_{k\,=\,1}^{N}\subset V: f = \sum_{k\,=\,1}^{N}f^k\mathds{1}_{I_k}(\cdot) \right.\right\}, \\[2mm]
    \Pi^1_c(\mathcal{I}_{\deltat}; V) &= \left\{f:[0,T]\to V\left|\,\exists\, \{f^k\}_{k\,=\,0}^{N}\subset V: f = \sum_{k\,=\,1}^{N}\left[f^{k-1} + \frac{(\cdot)-t_{k-1}}{t_k-t_{k-1}}\,(f^k-f^{k-1})\right]\mathds{1}_{I_k}(\cdot)\right. \right\},
\end{align*}
where $\mathds{1}_{I_{k}}$ denotes the indicator function of the set $I_k$, $k\in\{1,\dots,N\}$. Consequently, for $f\in \Pi^0(\mathcal{I}_{\deltat};V)\cup \Pi^1_c(\mathcal{I}_{\deltat};V)$ we write $f^k$ instead of $f(t_k)$, $k\in\{0,\dots,N\}$. Moreover, we set $f^{k+1/2}=\frac{1}{2}(f^k + f^{k+1})$, $k\in\{0,\dots,N-1\}$. \\

\noindent \textbf{Space discretization.} The admissible values for $h$ are the elements of the set $(0,H]$, where $H<\mathrm{diam}(\Omega)$. For each admissible $h$, we choose an admissible triangulation $\mathcal{T}_h$ of $\Omega$ (cf. Definition \ref{def_triangulation}) with mesh parameter $h$. We require the choices for the triangulations to be such that the family $\{\mathcal{T}_h\}_{h\,\in\,(0,H]}$ is uniform. Moreover, given $h\in(0,H]$ we define the following discrete function spaces:
\begin{gather*}
    \mathcal{R} = \Pi^0(\overline{\Omega}), \qquad \mathcal{Q}_h = \left\{f\in H^1_{\mathrm{per}}(\Omega)\cap C(\overline{\Omega})\left|\,f|_K\in \Pi^1(K) \quad \forall K\in\mathcal{T}_h\right.\right\}, \\[2mm]
    \mathcal{V}_h = \left\{f\in H^1_{0,\mathrm{per}}(\Omega)\cap C(\overline{\Omega})\left|\,f|_K\in \Pi^2(K) \quad \forall K\in\mathcal{T}_h\right.\right\}, \qquad X_h = \mathcal{Q}_h\times\mathcal{Q}_h\times\mathcal{V}_h^d\times\mathcal{Q}_h\times\mathcal{R}.
\end{gather*}

\noindent \textbf{Numerical scheme.} Given $(\phi_h^0,\fatu_h^0)\in \mathcal{Q}_h\times\mathcal{V}_h^d$ find $(\phi_h,\fatu_h)\in\Pi^1_c(\mathcal{I}_{\deltat};\mathcal{Q}_h\times\mathcal{V}_h^d)$, $(\mu_h,p_h)\in\Pi^0(\mathcal{I}_{\deltat};\mathcal{Q}_h\times\mathcal{Q}_h)$ and $r_h\in\Pi^0(\mathcal{I}_{\deltat};\mathcal{R})$ such that
\begin{align}
    \left\langle\frac{\phi_h^{n+1}-\phi_h^{n}}{\deltat}\,, \psi_h\right\rangle_{\!\Omega} &= \left\langle\phi_h^{n+1/2}\fatu_h^{n+1}, \gradx\psi_h\right\rangle_{\Omega} - \left\langle M(\phi_h^{n})\gradx\mu_h^{n+1}, \gradx\psi_h\right\rangle_{\Omega}\,, \label{discrete_phi} \\[4mm]
    \left\langle\mu_h^{n+1}, w_h\right\rangle_{\Omega} &= \gamma\left\langle\gradx\phi_h^{n+1}, \gradx w_h\right\rangle_{\Omega} + \left\langle f_{\mathrm{vex}}'(\phi_h^{n+1}), w_h\right\rangle_{\Omega} + \left\langle f_{\mathrm{cav}}'(\phi_h^{n}), w_h\right\rangle_{\Omega} + \left\langle\frac{\phi_h^{n+1}-\phi_h^{n}}{\deltat}\,, w_h\right\rangle_{\!\Gamma} \notag \\[2mm]
    &\phantom{\;=} + s\left\langle\nabla_{\Gamma}\phi_h^{n+1}, \nabla_{\Gamma}w_h\right\rangle_{\Gamma} + \left\langle g_{\mathrm{vex}}'(\phi_h^{n+1}), w_h\right\rangle_{\Gamma} + \left\langle g_{\mathrm{cav}}'(\phi_h^{n}), w_h\right\rangle_{\Gamma}\,, \label{discrete_mu} \\[4mm]
    \left\langle\frac{\fatu_h^{n+1}-\fatu_h^{n}}{\deltat}\,, \fatv_h\right\rangle_{\!\Omega} &= \frac{1}{2}\left\langle(\fatu_h^{n+1/2}\cdot\gradx)\fatv_h, \fatu_h^{n+1}\right\rangle_{\Omega} - \frac{1}{2}\left\langle(\fatu_h^{n+1/2}\cdot\gradx)\fatu_h^{n+1}, \fatv_h\right\rangle_{\Omega} + \left\langle \fatF, \fatv_h\right\rangle_{\Omega} \notag \\[2mm]
    &\phantom{\;=} - \left\langle\eta(\gammadot_h^{n},\phi_h^{n})\fatD{\fatu_h^{n+1}}, \fatD{\fatv_h}\right\rangle_{\Omega} + \left\langle p_h^{n+1}, \divx(\fatv_h)\right\rangle_{\Omega} - \left\langle\phi_h^{n+1/2}\gradx\mu_h^{n+1}, \fatv_h\right\rangle_{\Omega}\,, \label{discrete_u} \\[4mm]
    \left\langle r_h^{n+1}, q_h\right\rangle_{\Omega} &= -\left\langle\divx(\fatu_h^{n+1}), q_h\right\rangle_{\Omega} \label{discrete_div} \\[4mm]
    \left\langle p_h^{n+1}, s_h\right\rangle_{\Omega} &= 0 \label{discrete_mean_value}
\end{align}
for all $(\psi_h,w_h,\fatv_h,q_h,s_h)\in X_h$ and all $n\in\{0,\dots,N-1\}$. Here, $\gammadot_h^n=\sqrt{2\fatD{\fatu_h^n}:\fatD{\fatu_h^n}}$, $n\in\{0,\dots,N-1\}$. 

\begin{remark}
    To ensure the uniqueness of $p_h$, we require it to have zero mean. For practical reasons (connected to the implementation of the numerical scheme), we do not incorporate this constraint in the space for the discrete pressure. Instead, we include it as a separate equation -- see \eqref{discrete_mean_value} -- that is coupled to the other equations via the quantity $r_h$ that acts as a Lagrange multiplier.
\end{remark}

\subsection{Properties of the numerical scheme}
We summarise some of the basic properties of our numerical scheme.
\begin{theorem}\label{thm_properties_scheme}
    The numerical scheme \eqref{discrete_phi}--\eqref{discrete_mean_value} is \hypertarget{thm_i}{}
    \begin{itemize}
        \item[$\mathrm{(i)}$]{mass conserving, i.e. $\langle\phi_h^{n+1}, 1\rangle_\Omega = \langle\phi_h^{n}, 1\rangle_\Omega$ for all $n\in\{0,\dots,N-1\}$; \hypertarget{thm_ii}{}
        }
        \item[$\mathrm{(ii)}$]{energy-stable, i.e. if $\fatF=\bm{0}$ then $E(\phi_h^{n+1},\fatu_h^{n+1}) \leq E(\phi_h^{n},\fatu_h^{n})$ for all $n\in\{0,\dots,N-1\}$; \hypertarget{thm_iii}{}
        }
        \item[$\mathrm{(iii)}$]{solvable, i.e. given $n\in\{0,\dots,N-1\}$ and $(\phi_h^n,\fatu_h^n)\in \mathcal{Q}_h\times\mathcal{V}_h^d$ such that $\uM_h^n=\inf_{\fatx\,\in\,\Omega}\{M(\phi_h^n(\fatx))\}>0$, there exists 
        \[(\mu_h^{n+1},\phi_h^{n+1},\fatu_h^{n+1},p_h^{n+1},r_h^{n+1})\in X_h\]
        such that \eqref{discrete_phi}--\eqref{discrete_mean_value} are satisfied.
        }
    \end{itemize}
\end{theorem}

\noindent The proof of the solvability of the numerical scheme \eqref{discrete_phi}--\eqref{discrete_mean_value} is based on Schaefer's fixed point theorem (cf. Theorem \ref{schaefer}). Therefore, we introduce and analyse the following modified version of the $(n+1)$-st step of our numerical scheme. 

\begin{problem}\label{problem}
    Given $\lambda\in[0,1]$, $n\in\{0,\dots,N-1\}$ and $(\phi_h^n,\fatu_h^n)\in \mathcal{Q}_h\times\mathcal{V}_h^d$ find $(\mu_h^{n+1}, \phi_h^{n+1}, \fatu_h^{n+1}, p_h^{n+1}, r_h^{n+1})\in X_h$ such that
    \begin{align}
        \left\langle\mu_h^{n+1}, \psi_h\right\rangle_{\Omega} + \left\langle M(\phi_h^{n})\gradx\mu_h^{n+1}, \gradx\psi_h\right\rangle_{\Omega} &= \lambda\left[\left\langle\mu_h^{n+1}, \psi_h\right\rangle_{\Omega} + \left\langle\phi_h^{n+1/2}\fatu_h^{n+1}, \gradx\psi_h\right\rangle_{\Omega} - \left\langle\frac{\phi_h^{n+1}-\phi_h^{n}}{\deltat}\,, \psi_h\right\rangle_{\!\Omega}\,\right], \label{lambda_discrete_mu} \\[4mm]
        \left\langle\frac{\phi_h^{n+1}-\phi_h^{n}}{\deltat}\,, w_h\right\rangle_{\!\Omega} &= \lambda\left[\left\langle\mu_h^{n+1}, w_h\right\rangle_{\Omega} - \gamma\left\langle\gradx\phi_h^{n+1}, \gradx w_h\right\rangle_{\Omega} - \left\langle f_{\mathrm{vex}}'(\phi_h^{n+1}), w_h\right\rangle_{\Omega} \vphantom{\frac{\phi_h^{n+1}-\phi_h^{n}}{\deltat}}\right. \notag \\[2mm]
        &\phantom{\;= \lambda\;} - \left\langle f_{\mathrm{cav}}'(\phi_h^{n}), w_h\right\rangle_{\Omega} - \left\langle g_{\mathrm{vex}}'(\phi_h^{n+1}), w_h\right\rangle_{\Gamma} - \left\langle g_{\mathrm{cav}}'(\phi_h^{n}), w_h\right\rangle_{\Gamma} \notag \\[2mm]
        &\phantom{\;= \lambda\;} - s\left\langle\nabla_{\Gamma}\phi_h^{n+1}, \nabla_{\Gamma}w_h\right\rangle_{\Gamma} + \left\langle\frac{\phi_h^{n+1}-\phi_h^{n}}{\deltat}\,, w_h\right\rangle_{\!\Omega} \notag \\[2mm]
        &\phantom{\;= \lambda\;} \left. - \left\langle\frac{\phi_h^{n+1}-\phi_h^{n}}{\deltat}\,, w_h\right\rangle_{\!\Gamma}\,\right], \label{lambda_discrete_phi} \\[4mm]
        \left\langle\eta(\gammadot_h^{n},\phi_h^{n})\fatD{\fatu_h^{n+1}}, \fatD{\fatv_h}\right\rangle_{\Omega} &= \lambda\left[\frac{1}{2}\left\langle(\fatu_h^{n+1/2}\cdot\gradx)\fatv_h, \fatu_h^{n+1}\right\rangle_{\Omega} - \frac{1}{2}\left\langle(\fatu_h^{n+1/2}\cdot\gradx)\fatu_h^{n+1}, \fatv_h\right\rangle_{\Omega} \right. \notag \\[2mm]
        &\phantom{\;= \lambda\;} + \left\langle \fatF, \fatv_h\right\rangle_{\Omega} + \left\langle p_h^{n+1}, \divx(\fatv_h)\right\rangle_{\Omega} - \left\langle\phi_h^{n+1/2}\gradx\mu_h^{n+1}, \fatv_h\right\rangle_{\Omega} \notag \\[2mm] 
        &\phantom{\;= \lambda\;} \left. - \left\langle\frac{\fatu_h^{n+1}-\fatu_h^{n}}{\deltat}\,, \fatv_h\right\rangle_{\!\Omega}\,\right], \label{lambda_discrete_u} \\[4mm]
        \left\langle p_h^{n+1}, q_h\right\rangle_{\Omega} &= \lambda\left[\left\langle p_h^{n+1}, q_h\right\rangle_{\Omega} - \left\langle\divx(\fatu_h^{n+1}), q_h\right\rangle_{\Omega} - \left\langle r_h^{n+1}, q_h\right\rangle_{\Omega}\,\right] \label{lambda_discrete_p} \\[4mm]
        \left\langle r_h^{n+1}, s_h\right\rangle_{\Omega} &= \lambda\left[\left\langle r_h^{n+1}, s_h\right\rangle_{\Omega} + \left\langle p_h^{n+1}, s_h\right\rangle_{\Omega} \,\right] \label{lambda_discrete_r}
    \end{align}
    for all $(\psi_h,w_h,\fatv_h,q_h,s_h)\in X_h$.
\end{problem}

\begin{remark}
    We note that for $\lambda=1$ equations \eqref{lambda_discrete_mu}--\eqref{lambda_discrete_r} reduce to equations \eqref{discrete_phi}--\eqref{discrete_mean_value}. Thus, finding a solution to Problem \ref{problem} with $\lambda=1$ is equivalent to finding a solution to the numerical scheme.
\end{remark}

To be able to apply Schaefer's fixed point theorem, we need suitable a priori estimates for the solutions to Problem \ref{problem} that are uniform in $\lambda$. The main source for these estimates is the following energy inequality. 

\begin{lemma}[energy inequality]\label{lemma_discrete_energy_inequality}
    Let $\lambda\in(0,1]$, $(\phi_h^n,\fatu_h^n)\in\mathcal{Q}_h\times\mathcal{V}_h^d$ and $(\mu_h^{n+1}, \phi_h^{n+1},\fatu_h^{n+1},p_h^{n+1},r_h^{n+1})\in X_h$ such that equations \eqref{lambda_discrete_mu}--\eqref{lambda_discrete_r} are satisfied. Set $E_h^{k}=E(\phi_h^{k},\fatu_h^{k})$ for $k\in\{n,n+1\}$. Then there exists a constant $C_1=C_1(\underline{\eta},\Omega,d)>0$ such that
    \begin{align}
        \frac{E_h^{n+1}-E_h^{n}}{\deltat} &\leq - \frac{1}{2\deltat}\left|\left|\fatu_h^{n+1}-\fatu_h^{n}\right|\right|^2_{L^2(\Omega)} - \frac{\gamma}{2\deltat}\left|\left|\gradx(\phi_h^{n+1}-\phi_h^{n})\right|\right|^2_{L^2(\Omega)} \notag \\[2mm]
        &\phantom{=\;} - \frac{s}{2\deltat}\left|\left|\nabla_{\Gamma}(\phi_h^{n+1}-\phi_h^{n})\right|\right|^2_{L^2(\Gamma)} - \frac{1}{2\lambda}\,\underline{\eta}\,||\fatD{\fatu_h^{n+1}}||^2_{L^2(\Omega)} \notag \\[2mm]
        &\phantom{=\;} - \left\langle \frac{M(\phi_h^{n})}{\lambda}\,\gradx\mu_h^{n+1}, \gradx\mu_h^{n+1}\right\rangle_{\Omega} - \left|\left|\frac{\phi_h^{n+1}-\phi_h^{n}}{\deltat}\right|\right|^2_{L^2(\Gamma)} + C_1||\fatF||^2_{L^2(\Omega)} \notag \\[2mm]
        &\phantom{=\;} - \left(\frac{1}{\lambda}-1\right)\left(||\mu_h^{n+1}||^2_{L^2(\Omega)} + \left|\left|\frac{\phi_h^{n+1}-\phi_h^{n}}{\deltat}\right|\right|^2_{L^2(\Omega)} + ||p_h^{n+1}||^2_{L^2(\Omega)} + ||r_h^{n+1}||^2_{L^2(\Omega)}\right)\label{discrete_energy_inequality}
    \end{align}
\end{lemma}

\begin{proof}
    We choose an arbitrary $n\in\{0,\dots,N-1\}$ and observe that
    \begin{align}
        \frac{E_h^{n+1}-E_h^{n}}{\deltat} &= \left\langle\frac{\fatu_h^{n+1}-\fatu_h^{n}}{\deltat}\,, \fatu_h^{n+1}\right\rangle_{\!\Omega} - \frac{1}{2\deltat}\left|\left|\fatu_h^{n+1}-\fatu_h^{n}\right|\right|^2_{L^2(\Omega)} + \gamma\left\langle\gradx\frac{\phi_h^{n+1}-\phi_h^{n}}{\deltat}\,, \gradx\phi_h^{n+1} \right\rangle_{\!\Omega} \notag \\[2mm]
        &\phantom{=\;} - \frac{\gamma}{2\deltat}\left|\left|\gradx(\phi_h^{n+1}-\phi_h^{n})\right|\right|^2_{L^2(\Omega)} - \frac{s}{2\deltat}\left|\left|\nabla_{\Gamma}(\phi_h^{n+1}-\phi_h^{n})\right|\right|^2_{L^2(\Gamma)} + \left\langle\frac{f(\phi_h^{n+1})-f(\phi_h^{n})}{\deltat}\,, 1\right\rangle_{\!\Omega}\notag \\[2mm]
        &\phantom{=\;} + s\left\langle\nabla_{\Gamma}\frac{\phi_h^{n+1}-\phi_h^{n}}{\deltat}\,, \nabla_{\Gamma}\phi_h^{n+1} \right\rangle_{\!\Gamma} + \left\langle\frac{g(\phi_h^{n+1})-g(\phi_h^{n})}{\deltat}\,, 1\right\rangle_{\!\Gamma}\,. \label{discrete_energy_inequality_1}
    \end{align}
    Then, plugging  
    \[(\psi_h, w_h, \fatv_h, q_h, s_h)=\left(\mu_h^{n+1}, \frac{\phi_h^{n+1}-\phi_h^{n}}{\deltat}\,, \fatu_h^{n+1}, p_h^{n+1}, r_h^{n+1}\right)\]
    into \eqref{lambda_discrete_mu}--\eqref{lambda_discrete_r} and combining the resulting equations with \eqref{discrete_energy_inequality_1}, we obtain
    \begin{align}
        \frac{E_h^{n+1}-E_h^{n}}{\deltat} &= - \frac{1}{2\deltat}\left|\left|\fatu_h^{n+1}-\fatu_h^{n}\right|\right|^2_{L^2(\Omega)} - \frac{\gamma}{2\deltat}\left|\left|\gradx(\phi_h^{n+1}-\phi_h^{n})\right|\right|^2_{L^2(\Omega)} \notag \\[2mm]
        &\phantom{=\;} - \frac{s}{2\deltat}\left|\left|\nabla_{\Gamma}(\phi_h^{n+1}-\phi_h^{n})\right|\right|^2_{L^2(\Gamma)} - \left\langle\frac{\eta(\gamma_h^{n},\phi_h^{n})}{\lambda}\,\fatD{\fatu_h^{n+1}}, \fatD{\fatu_h^{n+1}}\right\rangle_{\Omega} \notag \\[2mm]
        &\phantom{=\;} + \langle\fatF,\fatu_h^{n+1}\rangle_{\Omega} - \left\langle \frac{M(\phi_h^{n})}{\lambda}\,\gradx\mu_h^{n+1}, \gradx\mu_h^{n+1}\right\rangle_{\Omega} - \left|\left|\frac{\phi_h^{n+1}-\phi_h^{n}}{\deltat}\right|\right|^2_{L^2(\Gamma)} \notag \\[2mm]
        &\phantom{=\;} + \left\langle\frac{f(\phi_h^{n+1})-f(\phi_h^{n})}{\deltat}\,, 1\right\rangle_{\!\Omega} - \left\langle f_{\mathrm{vex}}'(\phi_h^{n+1})+f_{\mathrm{cav}}'(\phi_h^{n}), \frac{\phi_h^{n+1}-\phi_h^{n}}{\deltat}\right\rangle_{\!\Omega} \notag \\[2mm]
        &\phantom{=\;} + \left\langle\frac{g(\phi_h^{n+1})-g(\phi_h^{n})}{\deltat}\,, 1\right\rangle_{\!\Gamma} - \left\langle g_{\mathrm{vex}}'(\phi_h^{n+1}) + g_{\mathrm{cav}}'(\phi_h^{n}), \frac{\phi_h^{n+1}-\phi_h^{n}}{\deltat}\right\rangle_{\!\Gamma} \notag \\[2mm]
        &\phantom{=\;} - \left(\frac{1}{\lambda}-1\right)\left(||\mu_h^{n+1}||^2_{L^2(\Omega)} + \left|\left|\frac{\phi_h^{n+1}-\phi_h^{n}}{\deltat}\right|\right|^2_{L^2(\Omega)} + ||p_h^{n+1}||^2_{L^2(\Omega)} + ||r_h^{n+1}||^2_{L^2(\Omega)}\right)
        \,. \label{discrete_energy_inequality_2}
    \end{align}
    Using Hölder's inequality \eqref{hölder}, Poincaré's inequality \eqref{poincare}, Korn's inequality \eqref{korn} and Young's inequality \eqref{young}, we deduce that
    \begin{align}
        |\langle\fatF,\fatu_h^{n+1}\rangle_{\Omega}| &\leq ||\fatF||_{L^2(\Omega)}||\fatu_h^{n+1}||_{L^2(\Omega)} \leq C_P^{(2)}||\fatF||_{L^2(\Omega)}||\gradx\fatu_h^{n+1}||_{L^2(\Omega)} \leq \sqrt{2} C_P^{(2)} ||\fatF||_{L^2(\Omega)}||\fatD{\fatu_h^{n+1}}||_{L^2(\Omega)} \notag \\[2mm]
        &\leq \frac{\lambda(C_P^{(2)})^2}{\underline{\eta}}\,||\fatF||^2_{L^2(\Omega)} + \frac{1}{2\lambda}\,\underline{\eta}\,||\fatD{\fatu_h^{n+1}}||^2_{L^2(\Omega)} \leq \frac{(C_P^{(2)})^2}{\underline{\eta}}\,||\fatF||^2_{L^2(\Omega)} + \frac{1}{2\lambda}\,\underline{\eta}\,||\fatD{\fatu_h^{n+1}}||^2_{L^2(\Omega)} \,.\label{F_estimate}
    \end{align}
    Moreover, due to assumption (\hyperlink{A4}{A4}),
    \begin{align}
        - \left\langle\frac{\eta(\gamma_h^{n},\phi_h^{n})}{\lambda}\,\fatD{\fatu_h^{n+1}}, \fatD{\fatu_h^{n+1}}\right\rangle_{\Omega} \leq -\frac{1}{\lambda}\,\underline{\eta}\,||\fatD{\fatu_h^{n+1}}||^2_{L^2(\Omega)}. \label{eta_estimate}
    \end{align}
    Next, we apply the mean value theorem to obtain functions $\zeta_{1,h}^{n},\zeta_{2,h}^{n}:\Omega\to\reals$ such that 
    \begin{align}
        f_{\mathrm{vex}}(\phi_h^{n+1}) - f_{\mathrm{vex}}(\phi_h^{n}) - f_{\mathrm{vex}}'(\phi_h^{n+1})(\phi_h^{n+1}-\phi_h^{n}) &= -\frac{1}{2}\,f_{\mathrm{vex}}''(\zeta_{1,h}^{n})(\phi_h^{n+1}-\phi_h^{n})^2\leq 0\,, \label{fvex_estimate} \\[2mm]
        f_{\mathrm{cav}}(\phi_h^{n+1}) - f_{\mathrm{cav}}(\phi_h^{n}) - f_{\mathrm{cav}}'(\phi_h^{n})(\phi_h^{n+1}-\phi_h^{n}) &= \frac{1}{2}\,f_{\mathrm{cav}}''(\zeta_{2,h}^{n})(\phi_h^{n+1}-\phi_h^{n})^2\leq 0\,. \label{fcav_estimate}
    \end{align}
    Analogous inequalities hold for $g_{\mathrm{vex}}, g_{\mathrm{cav}}$. Combining them and \eqref{F_estimate}--\eqref{fcav_estimate} with \eqref{discrete_energy_inequality_2}, we obtain \eqref{discrete_energy_inequality}.
\end{proof}

\begin{corollary}\label{corollary_apriori_estimates_1}
    Let the assumptions of Lemma \ref{lemma_discrete_energy_inequality} be satisfied. In addition, let $\uM_h^n=\inf_{\fatx\,\in\,\Omega}\{M(\phi_h^n(\fatx))\}>0$ and $C_2 = C_1\deltat ||\fatF||^2_{L^2(\Omega)} + E_h^n + c(|\Omega| + |\Gamma|)$. Then there exists a constant 
    \[C_3 = C_3(\deltat,||\fatF||_{L^2(\Omega)},\gamma,s,\uM_h^n,\underline{\eta},\overline{\eta},c,c',\Omega,\Gamma,d,E_h^n,\beta,||\phi_h^n||_{H^1(\Omega)},||\fatu_h^n||_{H^1(\Omega)})\]
    such that
    \begin{align}
        ||\fatu_h^{n+1}||^2_{H^1(\Omega)} + ||\phi_h^{n+1}||^2_{H^1(\Omega)} + ||\mu_h^{n+1}||^2_{H^1(\Omega)} + ||p_h^{n+1}||^2_{L^2(\Omega)} + ||r_h^{n+1}||^2_{L^2(\Omega)} \leq C_3\,. \label{discrete_stability}    
    \end{align}
\end{corollary}

\begin{proof}
    The following estimates follow immediately from \eqref{discrete_energy_inequality} and Assumption (\hyperlink{A5}{A5}):
    \begin{gather}
        ||\fatu_h^{n+1}||^2_{L^2(\Omega)} \leq 2C_2, \qquad\qquad \left|\left|\frac{\fatu_h^{n+1}-\fatu_h^n}{\deltat}\right|\right|^2_{L^2(\Omega)} \leq \frac{2C_2}{(\deltat)^2}, \label{estimates1} \\[2mm]
        ||\fatD{\fatu_h^{n+1}}||^2_{L^2(\Omega)}\leq\frac{1}{\lambda}\,||\fatD{\fatu_h^{n+1}}||^2_{L^2(\Omega)} \leq \frac{2C_2}{\underline{\eta}\deltat}, \qquad\qquad ||\gradx\mu_h^{n+1}||^2_{L^2(\Omega)} \leq \frac{C_2}{\uM_h^n\deltat}, \label{estimates2} \\[2mm]
        ||\gradx\phi_h^{n+1}||^2_{L^2(\Omega)} \leq \frac{2C_2}{\gamma}, \qquad\qquad \left|\left|\frac{\phi_h^{n+1}-\phi_h^n}{\deltat}\right|\right|^2_{L^2(\Gamma)} \leq \frac{C_2}{\deltat}, \qquad\qquad ||\nabla_{\Gamma}\phi_h^{n+1}||^2_{L^2(\Gamma)} \leq \frac{2C_2}{s}, \label{estimates3} \\[2mm]
        \left(\frac{1}{\lambda}-1\right)\max\left\{||\mu_h^{n+1}||^2_{L^2(\Omega)}, \left|\left|\frac{\phi_h^{n+1}-\phi_h^{n}}{\deltat}\right|\right|^2_{L^2(\Omega)}, ||p_h^{n+1}||^2_{L^2(\Omega)}, ||r_h^{n+1}||^2_{L^2(\Omega)}\right\} \leq \frac{C_2}{\deltat}. \label{estimates4}
    \end{gather}
    Using Korn's inequality \eqref{korn}, we deduce from the first estimates in \eqref{estimates1} and \eqref{estimates2}
    that
    \begin{align}
        ||\fatu_h^{n+1}||^2_{H^1(\Omega)} &= ||\fatu_h^{n+1}||^2_{L^2(\Omega)} + ||\gradx\fatu_h^{n+1}||^2_{L^2(\Omega)} \leq ||\fatu_h^{n+1}||^2_{L^2(\Omega)} + 2||\fatD{\fatu_h^{n+1}}||^2_{L^2(\Omega)} \leq \left(2 + \frac{4}{\underline{\eta}\deltat}\right)C_2 \label{est_u_final}.
    \end{align}
    Next, combining the first two estimates in \eqref{estimates3}
    with Poincaré's inequality \eqref{poincare_bdry}, Young's inequality \eqref{young} and the trace inequality \eqref{trace_inequality}, we deduce that
    \begin{align}
        ||\phi_h^{n+1}||^2_{H^1(\Omega)} &\leq C_P^{(3)}\left(||\gradx\phi_h^{n+1}||^2_{L^2(\Omega)} + ||\phi_h^{n+1}||^2_{L^2(\Gamma)}\right) \notag \\[2mm]
        &= C_P^{(3)}\left(||\gradx\phi_h^{n+1}||^2_{L^2(\Omega)} + \left|\left|\frac{\phi_h^{n+1}-\phi_h^n}{\deltat}\,\deltat + \phi_h^n\right|\right|^2_{L^2(\Gamma)}\right) \notag \\[2mm]
        &\leq C_P^{(3)}\left(||\gradx\phi_h^{n+1}||^2_{L^2(\Omega)} + 2\left((\deltat)^2\left|\left|\frac{\phi_h^{n+1}-\phi_h^n}{\deltat}\right|\right|^2_{L^2(\Gamma)} + ||\phi_h^n||^2_{L^2(\Gamma)}\right)\right) \notag \\[2mm]
        &\leq 2C_P^{(3)}\left[\left(\frac{1}{\gamma} + \deltat\right)C_2 + C^2_{\mathrm{tr}}\,||\phi_h^n||^2_{H^1(\Omega)}\right].
        \label{est_phi_final}
    \end{align}
    Next, we estimate $||\mu_h^{n+1}||^2_{H^1(\Omega)}$. Similarly to \eqref{est_phi_final}, we observe that
    \begin{align}
        ||\phi_h^{n+1}||^2_{H^1(\Gamma)} \leq 2\left[\left(\frac{1}{s} + \deltat\right)C_2 + C^2_{\mathrm{tr}}\,||\phi_h^n||^2_{H^1(\Omega)}\right]. \label{est_phi_bdry}
    \end{align}
    Then, testing \eqref{lambda_discrete_phi} with $w_h=\left\langle\mu_h^{n+1},1\right\rangle_\Omega$ and using Hölder's inequality \eqref{hölder} and Young's inequality \eqref{young} as well as Assumption (\hyperlink{A5}{A5}), we obtain
    \begin{align}
        \left|\left\langle\mu_h^{n+1},1\right\rangle_\Omega\right|^2 &\leq \left|\left\langle\mu_h^{n+1},1\right\rangle_\Omega\right|\,|\Omega|^{1/2}\left[\left(\frac{1}{\lambda}-1\right)\left|\left|\frac{\phi_h^{n+1}-\phi_h^n}{\deltat}\right|\right|_{L^2(\Omega)} + ||f'_{\mathrm{vex}}(\phi_h^{n+1})||_{L^2(\Omega)} + ||f'_{\mathrm{cav}}(\phi_h^{n})||_{L^2(\Omega)}\right] \notag \\[2mm]
        &\phantom{=\;} + \left|\left\langle\mu_h^{n+1},1\right\rangle_\Omega\right|\,|\Gamma|^{1/2}\left[\left|\left|\frac{\phi_h^{n+1}-\phi_h^n}{\deltat}\right|\right|_{L^2(\Gamma)} + ||g'_{\mathrm{vex}}(\phi_h^{n+1})||_{L^2(\Gamma)} + ||g'_{\mathrm{cav}}(\phi_h^{n})||_{L^2(\Gamma)}\right] \notag \\[2mm]
        &\leq \frac{1}{2}\left(\frac{1}{\lambda}-1\right)\left[|\Omega|\,\left|\left\langle\mu_h^{n+1},1\right\rangle_\Omega\right|^2 + \left|\left|\frac{\phi_h^{n+1}-\phi_h^n}{\deltat}\right|\right|^2_{L^2(\Omega)}\right] + \frac{5}{2}\,|\Omega|\,||f'_{\mathrm{vex}}(\phi_h^{n+1})||^2_{L^2(\Omega)} \notag \\[2mm]
        &\phantom{=\;} + \frac{5}{2}\,|\Omega|\,||f'_{\mathrm{cav}}(\phi_h^{n})||^2_{L^2(\Omega)} + \frac{5}{2}\,|\Gamma|\left|\left|\frac{\phi_h^{n+1}-\phi_h^n}{\deltat}\right|\right|^2_{L^2(\Gamma)} + \frac{5}{2}\,|\Gamma|\,||g'_{\mathrm{vex}}(\phi_h^{n+1})||^2_{L^2(\Gamma)}\notag \\[2mm]
        &\phantom{=\;} + \frac{5}{2}\,|\Gamma|\,||g'_{\mathrm{cav}}(\phi_h^{n})||^2_{L^2(\Gamma)} + \frac{1}{2}\left|\left\langle\mu_h^{n+1},1\right\rangle_\Omega\right|^2 \notag \\[2mm]
        &\leq \frac{1}{2}\left(\frac{1}{\lambda}-1\right)\left[|\Omega|^2\,||\mu_h^{n+1}||^2_{L^2(\Omega)} + \left|\left|\frac{\phi_h^{n+1}-\phi_h^n}{\deltat}\right|\right|^2_{L^2(\Omega)}\right] + 10(c')^2|\Omega|\left(|\Omega| + ||\phi_h^{n+1}||^6_{L^6(\Omega)}\right) \notag \\[2mm]
        &\phantom{=\;} + \frac{5}{2}\,|\Gamma|\left|\left|\frac{\phi_h^{n+1}-\phi_h^n}{\deltat}\right|\right|^2_{L^2(\Gamma)} + 10(c')^2|\Gamma|\left(|\Gamma| + ||\phi_h^{n+1}||^{q}_{L^{q}(\Gamma)}\right) + \frac{1}{2}\left|\left\langle\mu_h^{n+1},1\right\rangle_\Omega\right|^2. \notag
    \end{align}
    Rearranging terms and applying Sobolev's inequalities \eqref{sobolev} and \eqref{sobolev_bdry} as well as the second inequality in \eqref{estimates3} and the inequalities \eqref{estimates4}, \eqref{est_phi_final}, \eqref{est_phi_bdry} we obtain
    \begin{align}
        \left|\left\langle\mu_h^{n+1},1\right\rangle_\Omega\right|^2 &\leq \left(\frac{1}{\lambda}-1\right)\left[|\Omega|^2\,||\mu_h^{n+1}||^2_{L^2(\Omega)} + \left|\left|\frac{\phi_h^{n+1}-\phi_h^n}{\deltat}\right|\right|^2_{L^2(\Omega)}\right] + 20(c')^2|\Omega|\left(|\Omega| + \left[C^{(1)}_S||\phi_h^{n+1}||_{H^1(\Omega)}\right]^6\right) \notag \\[2mm]
        &\phantom{=\;} + 5|\Gamma|\left|\left|\frac{\phi_h^{n+1}-\phi_h^n}{\deltat}\right|\right|^2_{L^2(\Gamma)} + 20(c')^2|\Gamma|\left(|\Gamma| + \left[C^{(2)}_S||\phi_h^{n+1}||_{H^1(\Gamma)}\right]^{q}\right) \notag \\[2mm]
        &\leq \frac{C_2}{\deltat}\,(|\Omega|^2 + 5|\Gamma| + 1) + 20(c')^2(|\Omega|^2+|\Gamma|^2) \notag \\
        &\phantom{=\;} + 20(c')^2|\Omega|\left(2C_P^{(3)}(C_S^{(1)})^2\left[\left(\frac{1}{\gamma} + \deltat\right)C_2 + C^2_{\mathrm{tr}}\,||\phi_h^n||^2_{H^1(\Omega)}\right]\right)^3 \notag \\[2mm]
        &\phantom{=\;} + 20(c')^2|\Gamma|\left(2(C_S^{(2)})^2\left[\left(\frac{1}{s} + \deltat\right)C_2 + C^2_{\mathrm{tr}}\,||\phi_h^n||^2_{H^1(\Omega)}\right]\right)^{q/2}. \notag 
    \end{align}
    Combining this with the second inequality in \eqref{estimates2} and Poincaré's inequality \eqref{mean_poincare}, we obtain
    \begin{align}
        ||\mu_h^{n+1}||^2_{H^1(\Omega)} & \leq C_P^{(1)}\left(||\gradx\mu_h^{n+1}||^2_{L^2(\Omega)} + \left|\left\langle\mu_h^{n+1},1\right\rangle_\Omega\right|^2\right) \notag \\[2mm]
        &\leq C_P^{(1)}\left[\frac{C_2}{\uM_h^n\deltat} + \frac{C_2}{\deltat}\,(|\Omega|^2 + 5|\Gamma| + 1) + 20(c')^2(|\Omega|^2+|\Gamma|^2) \right. \notag \\[2mm]
        &\phantom{=\;} \qquad\quad + 20(c')^2|\Omega|\left(2C_P^{(3)}(C_S^{(1)})^2\left[\left(\frac{1}{\gamma} + \deltat\right)C_2 + C^2_{\mathrm{tr}}\,||\phi_h^n||^2_{H^1(\Omega)}\right]\right)^3 \notag \\[2mm]
        &\phantom{=\;} \qquad\quad \left. + 20(c')^2|\Gamma|\left(2(C_S^{(2)})^2\left[\left(\frac{1}{s} + \deltat\right)C_2 + C^2_{\mathrm{tr}}\,||\phi_h^n||^2_{H^1(\Omega)}\right]\right)^{q/2}\right]. \label{est_mu_final}
    \end{align}
    To estimate the pressure, we start with the initial observation that
    \begin{align}
        ||p_h^{n+1}||^2_{L^2(\Omega)} &\leq \lambda ||p_h^{n+1}||^2_{L^2(\Omega)} + \left(\frac{1}{\lambda} - 1\right)||p_h^{n+1}||^2_{L^2(\Omega)} \notag \\[2mm]
        &\leq 2\lambda \left|\left|p_h^{n+1} - \frac{1}{|\Omega|}\int_\Omega p_h^{n+1}\;\dx\right|\right|^2_{L^2(\Omega)} + \frac{2}{|\Omega|}\left|\left\langle p_h^{n+1}, 1\right\rangle_\Omega\right|^2 + \left(\frac{1}{\lambda} - 1\right)||p_h^{n+1}||^2_{L^2(\Omega)}\,. \label{p_split}
    \end{align}
    The last term on the right-hand side of \eqref{p_split} is controlled by \eqref{estimates4}.
    To estimate the second term, we test \eqref{lambda_discrete_r} with $s_h = \left\langle p_h^{n+1}, 1\right\rangle_\Omega$. Then we apply Young's inequality \eqref{young} and Hölder's inequality \eqref{hölder} and use
    \eqref{estimates4}
    to obtain
    \begin{align}
        \left|\left\langle p_h^{n+1}, 1\right\rangle_\Omega\right|^2 &= \left(\frac{1}{\lambda} - 1\right) \left|\left\langle p_h^{n+1}, 1\right\rangle_\Omega\right| \left|\left\langle r_h^{n+1}, 1\right\rangle_\Omega\right| \leq \frac{1}{2}\left(\frac{1}{\lambda} - 1\right) \left(\left|\left\langle p_h^{n+1}, 1\right\rangle_\Omega\right|^2 + \left|\left\langle r_h^{n+1}, 1\right\rangle_\Omega\right|^2\right) \notag \\[2mm]
        &\leq \frac{1}{2}\left(\frac{1}{\lambda} - 1\right)|\Omega| \left(||p_h^{n+1}||^2_{L^2(\Omega)} + ||r_h^{n+1}||^2_{L^2(\Omega)}\right) 
        \leq \frac{|\Omega|C_2}{\deltat}. \label{est_p_mean}
    \end{align}
    To estimate the first term on the right-hand side of \eqref{p_split}, we use the fact that the $P_2-P_1$ elements we have chosen for the velocity-pressure pair are inf-sup stable (cf. Theorem \eqref{thm_inf-sup}), whence
    \begin{align}
        \left|\left|p_h^{n+1} - \frac{1}{|\Omega|}\int_\Omega p_h^{n+1}\;\dx\right|\right|^2_{L^2(\Omega)} \leq \frac{1}{\beta^2}\sup_{\fatw_h\,\in\,\mathcal{V}_h^d\backslash\{\bm{0}\}}\left\{\left|\int_\Omega\frac{\divx(\fatw_h)\,p_h^{n+1}}{||\fatw_h||_{H^1(\Omega)}}\;\dx \right|^2\right\}. \label{est1_inf-sup}
    \end{align}
    To estimate the right-hand side of this inequality, we consider an arbitrary test function $\fatv_h\in\mathcal{V}_h^d$ in \eqref{lambda_discrete_u}. Using Hölder's inequality \eqref{hölder}, Sobolev's inequality \eqref{sobolev} and Assumption (\hyperlink{A4}{A4}), we deduce that
    \begin{align}
        \left|\left\langle p_h^{n+1}, \divx(\fatv_h)\right\rangle\right| &\leq ||\fatu_h^{n+1/2}||_{L^4(\Omega)}||\gradx\fatv_h||_{L^2(\Omega)}||\fatu_h^{n+1}||_{L^4(\Omega)} + ||\fatu_h^{n+1/2}||_{L^4(\Omega)}||\gradx\fatu_h^{n+1}||_{L^2(\Omega)}||\fatv_h||_{L^4(\Omega)} \notag \\[2mm]
        &\phantom{=\;} + ||\fatF||_{L^2(\Omega)}||\fatv_h||_{L^2(\Omega)} + \frac{\overline{\eta}}{\lambda}\,||\fatD{\fatu_h^{n+1}}||_{L^2(\Omega)}||\fatD{\fatv_h}||_{L^2(\Omega)} \notag \\[2mm]
        &\phantom{=\;} + ||\phi_h^{n+1/2}||_{L^4(\Omega)}||\gradx\mu_h^{n+1}||_{L^2(\Omega)}||\fatv_h||_{L^4(\Omega)} + \left|\left|\frac{\fatu_h^{n+1}-\fatu_h^n}{\deltat}\right|\right|_{L^2(\Omega)}||\fatv_h||_{L^2(\Omega)} \notag \\[2mm]
        &\leq 2(C_S^{(1)})^2 ||\fatu_h^{n+1/2}||_{H^1(\Omega)}||\fatv_h||_{H^1(\Omega)}||\fatu_h^{n+1}||_{H^1(\Omega)} + ||\fatF||_{L^2(\Omega)}||\fatv_h||_{H^1(\Omega)} \notag \\[2mm]
        &\phantom{=\;} + \frac{\overline{\eta}}{\lambda}\,||\fatD{\fatu_h^{n+1}}||_{L^2(\Omega)}||\fatv_h||_{H^1(\Omega)} + (C_S^{(1)})^2||\phi_h^{n+1/2}||_{H^1(\Omega)}||\gradx\mu_h^{n+1}||_{L^2(\Omega)}||\fatv_h||_{H^1(\Omega)} \notag \\[2mm]
        &\phantom{=\;} + \left|\left|\frac{\fatu_h^{n+1}-\fatu_h^n}{\deltat}\right|\right|_{L^2(\Omega)}||\fatv_h||_{H^1(\Omega)}. \notag 
    \end{align}
    Consequently, \eqref{est_u_final}, \eqref{estimates2}, \eqref{est_phi_final} and the second inequality in \eqref{estimates1}
    yield
    \begin{align}
        \frac{1}{\beta^2}&\sup_{\fatw_h\,\in\,\mathcal{V}_h^d\backslash\{\bm{0}\}}\left\{\left|\int_\Omega\frac{\divx(\fatw_h)\,p_h^{n+1}}{||\fatw_h||_{H^1(\Omega)}}\;\dx \right|^2\right\} \notag \\[2mm]
        &\leq \frac{1}{\beta^2}\left((C_S^{(1)})^2 \left[||\fatu_h^{n}||_{H^1(\Omega)} + ||\fatu_h^{n+1}||_{H^1(\Omega)}\right]||\fatu_h^{n+1}||_{H^1(\Omega)} + ||\fatF||_{L^2(\Omega)} + \frac{\overline{\eta}}{\lambda}\,||\fatD{\fatu_h^{n+1}}||_{L^2(\Omega)} \vphantom{\left|\left|\frac{\fatu_h^{n+1}-\fatu_h^n}{\deltat}\right|\right|_{L^2(\Omega)}}\right. \notag \\[2mm]
        &\phantom{=\;\;\;} \qquad \left. +\; (C_S^{(1)})^2\left[||\phi_h^{n}||_{H^1(\Omega)} + ||\phi_h^{n+1}||_{H^1(\Omega)}\right]||\gradx\mu_h^{n+1}||_{L^2(\Omega)} + \left|\left|\frac{\fatu_h^{n+1}-\fatu_h^n}{\deltat}\right|\right|_{L^2(\Omega)}\right)^2 \notag \\[2mm]
        &\leq \frac{49}{\beta^2}\left((C_S^{(1)})^4 \left[||\fatu_h^{n}||^2_{H^1(\Omega)} + ||\fatu_h^{n+1}||^2_{H^1(\Omega)}\right]||\fatu_h^{n+1}||^2_{H^1(\Omega)} + ||\fatF||^2_{L^2(\Omega)} + \frac{\overline{\eta}^2}{\lambda^2}\,||\fatD{\fatu_h^{n+1}}||^2_{L^2(\Omega)} \vphantom{\left|\left|\frac{\fatu_h^{n+1}-\fatu_h^n}{\deltat}\right|\right|_{L^2(\Omega)}}\right. \notag \\[2mm]
        &\phantom{=\;\;\;} \qquad \left. +\; (C_S^{(1)})^4\left[||\phi_h^{n}||^2_{H^1(\Omega)} + ||\phi_h^{n+1}||^2_{H^1(\Omega)}\right]||\gradx\mu_h^{n+1}||^2_{L^2(\Omega)} + \left|\left|\frac{\fatu_h^{n+1}-\fatu_h^n}{\deltat}\right|\right|^2_{L^2(\Omega)}\right) \notag \\[2mm]
        &\leq \frac{49}{\beta^2}\left((C_S^{(1)})^4 \left[||\fatu_h^{n}||^2_{H^1(\Omega)} + \left(2 + \frac{4}{\underline{\eta}\deltat}\right)C_2\right]\left(2 + \frac{4}{\underline{\eta}\deltat}\right)C_2 + ||\fatF||^2_{L^2(\Omega)} + \frac{2\overline{\eta}^2 C_2}{\lambda\underline{\eta}\deltat} \right. \notag \\[2mm]
        &\phantom{=\;\;\;} \qquad \left. +\; (C_S^{(1)})^4\left[||\phi_h^{n}||^2_{H^1(\Omega)} + 2C_P^{(3)}\left[\left(\frac{1}{\gamma}+\deltat\right)C_2 + C^2_{\mathrm{tr}}\,||\phi_h^n||^2_{H^1(\Omega)}\right]\right]\frac{C_2}{\uM_h^n\deltat} + \frac{2C_2}{(\deltat)^2}\right). \label{est2_inf-sup}
    \end{align}
    Thus, combining \eqref{p_split} with \eqref{estimates4}, \eqref{est_p_mean}, \eqref{est1_inf-sup} and \eqref{est2_inf-sup}, we obtain
    \begin{align}
        ||p_h^{n+1}||^2_{L^2(\Omega)} &\leq \frac{3C_2}{\deltat} + \frac{98}{\beta^2}\left((C_S^{(1)})^4 \left[||\fatu_h^{n}||^2_{H^1(\Omega)} + \left(2 + \frac{4}{\underline{\eta}\deltat}\right)C_2\right]\left(2 + \frac{4}{\underline{\eta}\deltat}\right)C_2 + ||\fatF||^2_{L^2(\Omega)} + \frac{2\overline{\eta}^2 C_2}{\underline{\eta}\deltat} \right. \notag \\[2mm]
        &\phantom{=\;\;\;} \qquad\qquad \left. +\; (C_S^{(1)})^4\left[||\phi_h^{n}||^2_{H^1(\Omega)} + 2C_P^{(3)}\left[\left(\frac{1}{\gamma}+\deltat\right)C_2 + C^2_{\mathrm{tr}}\,||\phi_h^n||^2_{H^1(\Omega)}\right]\right]\frac{C_2}{\uM_h^n\deltat} + \frac{2C_2}{(\deltat)^2}\right). \label{est_p_final}
    \end{align}
    Finally, we estimate $||r_h^{n+1}||^2_{L^2(\Omega)}$. To this end, we first apply the divergence theorem to deduce that
    \begin{align*}
        \left\langle\divx(\fatu_h^{n+1}),r_h^{n+1}\right\rangle_{\Omega} = r_h^{n+1}\int_\Omega \divx(\fatu_h^{n+1})\;\dx = r_h^{n+1}\int_{\Gamma} \fatu_h^{n+1}\cdot\fatn\;\dsx = 0.
    \end{align*}
    Here, the last equality is due to the fact that $\fatu_h^{n+1}\in H^1_0(\Omega)$. Consequently, testing \eqref{lambda_discrete_p} with $q_h=r_h^{n+1}$ and using Hölder's inequality \eqref{hölder} and Young's inequality \eqref{young} as well as \eqref{estimates4}
    yields
    \begin{align}
        ||r_h^{n+1}||^2_{L^2(\Omega)} = \left(\frac{1}{\lambda} - 1\right)\left|\left\langle p_h^{n+1}, r_h^{n+1}\right\rangle_\Omega\right| \leq \frac{1}{2}\left(\frac{1}{\lambda} - 1\right)\left(||p_h^{n+1}||^2_{L^2(\Omega)} + ||r_h^{n+1}||^2_{L^2(\Omega)}\right) \leq\frac{C_2}{\deltat}\,. \label{est_r_final}
    \end{align}
\end{proof}

\begin{corollary}[a priori estimates]\label{corollary_apriori_estimates_2}
    Suppose $(\phi_h,\fatu_h)\in\Pi^1_c(\mathcal{I}_{\deltat};\mathcal{Q}_h\times\mathcal{V}_h^d)$, $(\mu_h,p_h)\in\Pi^0(\mathcal{I}_{\deltat};\mathcal{Q}_h\times\mathcal{Q}_h)$ and $r_h\in\Pi^0(\mathcal{I}_{\deltat};\mathcal{R})$ solve the numerical scheme \eqref{discrete_phi}--\eqref{discrete_mean_value} for some initial data $(\phi_h^0,\fatu_h^0)\in \mathcal{Q}_h\times\mathcal{V}_h^d$. Suppose further that
    \[\uM_h=\inf_{(t,\fatx)\,\in\,(0,T)\times\Omega}\{M(\phi_h(t,\fatx))\} > 0.\]
    Then there exist constants 
    \begin{align*}
        C_4 &= C_4(||\fatF||_{L^2(0;T;L^2(\Omega))},\gamma,s,\uM_h,\underline{\eta},\overline{\eta},c,c',T,\Omega,\Gamma,d,E_h^0,\beta,||\phi_h^0||_{H^1(\Omega)},||\fatu_h^0||_{H^1(\Omega)}), \\
        C_5 &= C_5(\deltat,||\fatF||_{L^2(0,T;L^2(\Omega))},\gamma,s,\uM_h,\underline{\eta},\overline{\eta},c,c',T,\Omega,\Gamma,d,E_h^0,\beta,||\phi_h^0||_{H^1(\Omega)},||\fatu_h^0||_{H^1(\Omega)})
    \end{align*}
    such that
    \begin{align}
        ||\phi_h||^2_{L^\infty(0,T;H^1(\Omega))} + ||\fatu_h||^2_{L^\infty(0,T;L^2(\Omega))} + ||\fatu_h||^2_{L^2(0,T;H^1(\Omega))} + ||\mu_h||^2_{L^2(0,T;H^1(\Omega))} \leq C_4, \\[2mm]
        ||p_h||^2_{L^2(0,T;L^2(\Omega))} \leq C_5.
    \end{align}
\end{corollary}

\begin{proof}
    Let $n\in\{0,\dots,N-1\}$ be arbitrary. It follows from \eqref{discrete_energy_inequality} with $\lambda=1$ that
    \begin{align}
        E_h^{n+1} = E_h^0 + \deltat\sum_{k\,=\,0}^{n}\frac{E_h^{k+1}-E_h^k}{\deltat} \leq E_h^0 + C_1\sum_{k\,=\,0}^{n}\deltat\,||\fatF||^2_{L^2(\Omega)} \leq E_h^0 + C_1||\fatF||^2_{L^2(0,T;L^2(\Omega))}\,. \label{est_energy}
    \end{align}
    Consequently,
    \begin{align}
        &||\gradx\phi_h||^2_{L^\infty(0,T;L^2(\Omega))} + ||\nabla_{\Gamma}\phi_h||^2_{L^\infty(0,T;L^2(\Gamma))} + ||\fatu_h||^2_{L^\infty(0,T;L^2(\Omega))} \notag \\[2mm]
        &= \max_{0\,\leq\,k\,\leq\,N}\left\{||\gradx\phi_h^k||^2_{L^2(\Omega)}\right\} + \max_{0\,\leq\,k\,\leq\,N}\left\{||\nabla_{\Gamma}\phi_h^k||^2_{L^2(\Gamma)}\right\} + \max_{0\,\leq\,k\,\leq\,N}\left\{||\fatu_h^k||^2_{L^2(\Omega)}\right\} \notag \\[2mm]
        &\leq \left(2+\frac{2}{s}+\frac{2}{\gamma}\right)\max_{0\,\leq\,k\,\leq\,N}\left\{E_h^k + c(|\Omega|+|\Gamma|)\right\} \leq \left(2+\frac{2}{s}+\frac{2}{\gamma}\right)\left[E_h^0 + c(|\Omega|+|\Gamma|) + C_1||\fatF||^2_{L^2(0,T;L^2(\Omega))}\right].\label{est_LinfL2} 
    \end{align}
    Moreover, using $\psi_h=1$ as a test function in \eqref{discrete_phi}, we deduce that $\langle\phi_h,1\rangle_\Omega = \langle\phi_h^0,1\rangle_\Omega$. Consequently, Poincaré's inequality \eqref{mean_poincare}, \eqref{est_LinfL2} and Hölder's inequality \eqref{hölder} yield
    \begin{align}
        ||\phi_h||^2_{L^\infty(0,T;H^1(\Omega))} &\leq C_P^{(1)}\left(||\gradx\phi_h||^2_{L^\infty(0,T;L^2(\Omega))} + \big|\langle\phi_h^0,1\rangle_\Omega\big|^2\right) \notag \\[2mm]
        &\leq C_P^{(1)}\left(\left(2+\frac{2}{s}+\frac{2}{\gamma}\right)\left[E_h^0 + c(|\Omega|+|\Gamma|) + C_1||\fatF||^2_{L^2(0,T;L^2(\Omega))}\right] + |\Omega|\,||\phi_h^0||^2_{L^2(\Omega)}\right). \label{est_phi_LinfH1}
    \end{align}
    Next, we combine the trace inequality \eqref{trace_inequality} with \eqref{est_phi_LinfH1} to see that
    \begin{align}
        ||\phi_h||^2_{L^\infty(0,T;L^2(\Gamma))} &\leq C_{\mathrm{tr}}^2\,||\phi_h||^2_{L^\infty(0,T;H^1(\Omega))} \notag \\[2mm]
        &\leq C_{\mathrm{tr}}^2 C_P^{(1)}\left(\left(2+\frac{2}{s}+\frac{2}{\gamma}\right)\left[E_h^0 + c(|\Omega|+|\Gamma|) + C_1||\fatF||^2_{L^2(0,T;L^2(\Omega))}\right] + |\Omega|\,||\phi_h^0||^2_{L^2(\Omega)}\right). \label{est_phi_LinfL2_gamma}
    \end{align}
    Together, \eqref{est_LinfL2} and \eqref{est_phi_LinfL2_gamma} yield
    \begin{align}
        &||\phi_h||^2_{L^\infty(0,T;H^1(\Gamma))} = ||\phi_h||^2_{L^\infty(0,T;L^2(\Gamma))} + ||\nabla_{\Gamma}\phi_h||^2_{L^\infty(0,T;L^2(\Gamma))} \notag \\[2mm]
        &\qquad \leq \left(1+C_{\mathrm{tr}}^2 C_P^{(1)}\right)\left(2+\frac{2}{s}+\frac{2}{\gamma}\right)\left[E_h^0 + c(|\Omega|+|\Gamma|) + C_1||\fatF||^2_{L^2(0,T;L^2(\Omega))}\right] + C_{\mathrm{tr}}^2 C_P^{(1)}|\Omega|\,||\phi_h^0||^2_{L^2(\Omega)}.
    \end{align}
    From \eqref{discrete_energy_inequality} we deduce that
    \begin{align}
        &\sum_{k\,=\,1}^N \deltat\,||\fatD{\fatu_h^k}||^2_{L^2(\Omega)} + \sum_{k\,=\,1}^N (\deltat)^2\left|\left|\frac{\fatu_h^k-\fatu_h^{k-1}}{\deltat}\right|\right|^2_{L^2(\Omega)} + \sum_{k\,=\,1}^N \deltat\left|\left|\frac{\phi_h^k-\phi_h^{k-1}}{\deltat}\right|\right|^2_{L^2(\Gamma)} + ||\gradx\mu_h||^2_{L^2(0,T;L^2(\Omega))} \notag \\[2mm]
        &\quad \leq \left(2+\frac{2}{\underline{\eta}}+\frac{1}{\uM_h}\right)\left(C_1\sum_{k\,=\,1}^N \deltat\,||\fatF||^2_{L^2(\Omega)} - \sum_{k\,=\,1}^N (E_h^k-E_h^{k-1})\right) \notag \\[2mm]
        &\quad \leq \left(2+\frac{2}{\underline{\eta}}+\frac{1}{\uM_h}\right)\left(C_1||\fatF||^2_{L^2(0,T;L^2(\Omega))} + E_h^0 + c(|\Omega|+|\Gamma|)\right). \label{L2L2_base}
    \end{align}
    Using Korn's inequality \eqref{korn} and \eqref{L2L2_base} we observe that 
    \begin{align}
        &||\gradx\fatu_h||^2_{L^2(0,T;L^2(\Omega))} = \int_0^{\,T} \left|\left|\,\sum_{k\,=\,1}^N\left[\frac{t_k-t}{\deltat}\,\gradx\fatu_h^{k-1} + \frac{t-t_{k-1}}{\deltat}\,\gradx\fatu_h^k\right]\mathds{1}_{I_k}(t)\right|\right|^2_{L^2(\Omega)}\;\dt \notag \\[2mm]
        &= \sum_{k\,=\,1}^N \int_{t_{k-1}}^{\,t_k} \left|\left|\frac{t_k-t}{\deltat}\,\gradx\fatu_h^{k-1} + \frac{t-t_{k-1}}{\deltat}\,\gradx\fatu_h^k\right|\right|^2_{L^2(\Omega)}\;\dt \leq \sum_{k\,=\,1}^N \int_{t_{k-1}}^{\,t_k} \big(||\gradx\fatu_h^{k-1}||_{L^2(\Omega)} + ||\gradx\fatu_h^k||_{L^2(\Omega)}\big)^2\;\dt \notag \\[2mm]
        &\leq 2\sum_{k\,=\,1}^N \deltat \big(||\gradx\fatu_h^{k-1}||^2_{L^2(\Omega)} + ||\gradx\fatu_h^k||^2_{L^2(\Omega)}\big) \leq 2\deltat\,||\gradx\fatu_h^0||^2_{L^2(\Omega)} + 4\sum_{k\,=\,1}^N \deltat\,||\gradx\fatu_h^k||^2_{L^2(\Omega)} \notag \\[2mm]
        &\leq 2\deltat\,||\fatu_h^0||^2_{H^1(\Omega)} + 8\sum_{k\,=\,1}^N \deltat\,||\fatD{\fatu_h^k}||^2_{L^2(\Omega)} \notag \\[2mm]
        &\leq 2\deltat\,||\fatu_h^0||^2_{H^1(\Omega)} + 8\left(2+\frac{2}{\underline{\eta}}+\frac{1}{\uM_h}\right)\left(C_1||\fatF||^2_{L^2(0,T;L^2(\Omega))} + E_h^0 + c(|\Omega|+|\Gamma|)\right).
    \end{align}
    Combining this inequality with \eqref{est_LinfL2} and Poincaré's inequality \eqref{poincare}, we obtain
    \begin{align}
        &||\fatu_h||^2_{L^2(0,T;H^1(\Omega))} \leq (C_P^{(2)})^2||\gradx\fatu_h||^2_{L^2(0,T;L^2(\Omega))} \notag \\[2mm]
        &\quad \leq 2T\,||\fatu_h^0||^2_{H^1(\Omega)} + 8\left(2+\frac{2}{\underline{\eta}}+\frac{1}{\uM_h}\right)\left(C_1||\fatF||^2_{L^2(0,T;L^2(\Omega))} + E_h^0 + c(|\Omega|+|\Gamma|)\right).
    \end{align}
    The computations in Corollary \ref{corollary_apriori_estimates_1} and \eqref{L2L2_base} yield
    \begin{align}
        &||\langle\mu_h,1\rangle_\Omega||^2_{L^2(0,T)} = \sum_{k\,=\,1}^N \deltat\left|\left\langle\mu_h^{k},1\right\rangle_\Omega\right|^2 \notag \\[2mm]
        &\leq \sum_{k\,=\,1}^N \deltat\left[5|\Gamma|\left|\left|\frac{\phi_h^{k}-\phi_h^{k-1}}{\deltat}\right|\right|^2_{L^2(\Gamma)} \!+ 20(c')^2|\Omega|\left(|\Omega| + \left[C^{(1)}_S||\phi_h^{k}||_{H^1(\Omega)}\right]^6\right) + 20(c')^2|\Gamma|\left(|\Gamma| + \left[C^{(2)}_S||\phi_h^{k}||_{H^1(\Gamma)}\right]^{q}\right)\right] \notag \\[2mm]
        &\leq 5|\Gamma|\left(2+\frac{2}{\underline{\eta}}+\frac{1}{\uM_h}\right)\left(C_1||\fatF||^2_{L^2(0,T;L^2(\Omega))} + E_h^0 + c(|\Omega|+|\Gamma|)\right) + 20T(c')^2|\Omega|\left(|\Omega| + \left[C^{(1)}_S||\phi_h||_{L^\infty(0,T;H^1(\Omega))}\right]^6\right) \notag \\[2mm]
        &\phantom{=\;} + 20T(c')^2|\Gamma|\left(|\Gamma| + \left[C^{(2)}_S||\phi_h||_{L^\infty(0,T;H^1(\Gamma))}\right]^{q}\right).
    \end{align}
    Combining this inequality with \eqref{L2L2_base} and Poincaré's inequality \eqref{mean_poincare}, we obtain
    \begin{align}
        &||\mu_h||^2_{L^2(0,T;H^1(\Omega))} \leq C_P^{(1)}\left(||\gradx\mu_h||^2_{L^2(0,T;L^2(\Omega))} + ||\langle\mu_h,1\rangle_\Omega||^2_{L^2(0,T)}\right) \notag \\[2mm]
        &\leq C_P^{(1)}(1+5|\Gamma|)\left(2+\frac{2}{\underline{\eta}}+\frac{1}{\uM_h}\right)\left(C_1||\fatF||^2_{L^2(0,T;L^2(\Omega))} + E_h^0 + c(|\Omega|+|\Gamma|)\right) \notag \\[2mm]
        &\phantom{=\;} + C_P^{(1)}\left[20T(c')^2|\Omega|\left(|\Omega| + \left[C^{(1)}_S||\phi_h||_{L^\infty(0,T;H^1(\Omega))}\right]^6\right) + 20T(c')^2|\Gamma|\left(|\Gamma| + \left[C^{(2)}_S||\phi_h||_{L^\infty(0,T;H^1(\Gamma))}\right]^{q}\right)\right].
    \end{align}
    Finally, due to \eqref{discrete_mean_value}, the pressure $p_h$ has zero mean at each time instant $t\in[0,T]$. Repeating the arguments leading to \eqref{est2_inf-sup}, applying Poincaré's inequality \eqref{poincare} and Korn's inequality \eqref{korn} and using \eqref{L2L2_base}, we deduce that
    \begin{align}
        &||p_h||^2_{L^2(0,T;L^2(\Omega))} = \sum_{k\,=\,1}^N \deltat\,||p_h^k||^2_{L^2(\Omega)} \leq \frac{1}{\beta^2}\sum_{k\,=\,1}^N \deltat\sup_{\fatw_h\,\in\,\mathcal{V}_h^d\backslash\{\bm{0}\}}\left\{\left|\int_\Omega\frac{\divx(\fatw_h)\,p_h^{k}}{||\fatw_h||_{H^1(\Omega)}}\;\dx \right|^2\right\} \notag \\[2mm]
        &\leq \frac{49}{\beta^2}\sum_{k\,=\,1}^N \deltat\left((C_S^{(1)})^4 \left[||\fatu_h^{k-1}||^2_{H^1(\Omega)} + ||\fatu_h^{k}||^2_{H^1(\Omega)}\right]||\fatu_h^{k}||^2_{H^1(\Omega)} + ||\fatF||^2_{L^2(\Omega)} + \overline{\eta}^2\,||\fatD{\fatu_h^{k}}||^2_{L^2(\Omega)} \vphantom{\left|\left|\frac{\fatu_h^{k}-\fatu_h^{k-1}}{\deltat}\right|\right|_{L^2(\Omega)}}\right. \notag \\[2mm]
        &\qquad\qquad\qquad\quad \left. +\; (C_S^{(1)})^4\left[||\phi_h^{k-1}||^2_{H^1(\Omega)} + ||\phi_h^{k}||^2_{H^1(\Omega)}\right]||\gradx\mu_h^{k}||^2_{L^2(\Omega)} + \left|\left|\frac{\fatu_h^{k}-\fatu_h^{k-1}}{\deltat}\right|\right|^2_{L^2(\Omega)}\right) \notag \\[2mm]
        &\leq \frac{49}{\beta^2\deltat}\left[(C_S^{(1)})^4\left(\deltat\,||\fatu_h^{0}||^2_{H^1(\Omega)} + 2\sum_{k\,=\,1}^N \deltat\,||\fatu_h^{k}||^2_{H^1(\Omega)}\right)\sum_{k\,=\,1}^N \deltat\,||\fatu_h^{k}||^2_{H^1(\Omega)} + \sum_{k\,=\,1}^N (\deltat)^2\left|\left|\frac{\fatu_h^{k}-\fatu_h^{k-1}}{\deltat}\right|\right|^2_{L^2(\Omega)}\right]  \notag \\[2mm]
        &\phantom{=\;} + \frac{49}{\beta^2}\left(||\fatF||^2_{L^2(0,T;L^2(\Omega))} + \overline{\eta}^2 \sum_{k\,=\,1}^N \deltat\,||\fatD{\fatu_h^k}||^2_{L^2(\Omega)} + 2(C_S^{(1)})^4||\phi_h||^2_{L^\infty(0,T;H^1(\Omega))}||\gradx\mu_h||^2_{L^2(0,T;L^2(\Omega))}\right) \notag \\[2mm]
        &\leq \frac{49}{\beta^2\deltat}\left[2(C_S^{(1)})^4(C_P^{(2)})^2\left(\deltat\,||\fatu_h^{0}||^2_{H^1(\Omega)} + 4(C_P^{(2)})^2\sum_{k\,=\,1}^N \deltat\,||\fatD{\fatu_h^{k}}||^2_{L^2(\Omega)}\right)\sum_{k\,=\,1}^N \deltat\,||\fatD{\fatu_h^{k}}||^2_{L^2(\Omega)}\right]  \notag \\[2mm]
        &\phantom{=\;} + \frac{49}{\beta^2}\left(\overline{\eta}^2 \sum_{k\,=\,1}^N \deltat\,||\fatD{\fatu_h^k}||^2_{L^2(\Omega)} + \frac{1}{\deltat}\sum_{k\,=\,1}^N (\deltat)^2\left|\left|\frac{\fatu_h^{k}-\fatu_h^{k-1}}{\deltat}\right|\right|^2_{L^2(\Omega)}\right) \notag \\[2mm]
        &\phantom{=\;} + \frac{49}{\beta^2}\left(||\fatF||^2_{L^2(0,T;L^2(\Omega))} + 2(C_S^{(1)})^4||\phi_h||^2_{L^\infty(0,T;H^1(\Omega))}||\mu_h||^2_{L^2(0,T;H^1(\Omega))}\right) \notag \\[2mm]
        &\leq \frac{49}{\beta^2}\left(2(C_S^{(1)})^4(C_P^{(2)})^2||\fatu_h^0||^2_{H^1(\Omega)} + \overline{\eta}^2 + \frac{1}{\deltat}\right)\left(2+\frac{2}{\underline{\eta}}+\frac{1}{\uM_h}\right)\left(C_1||\fatF||^2_{L^2(0,T;L^2(\Omega))} + E_h^0 + c(|\Omega|+|\Gamma|)\right) \notag \\[2mm]
        &\phantom{=\;} + \frac{392}{\beta^2\deltat}\,(C_S^{(1)})^4(C_P^{(2)})^4\left[\left(2+\frac{2}{\underline{\eta}}+\frac{1}{\uM_h}\right)\left(C_1||\fatF||^2_{L^2(0,T;L^2(\Omega))} + E_h^0 + c(|\Omega|+|\Gamma|)\right)\right]^2 \notag \\[2mm]
        &\phantom{=\;} + \frac{49}{\beta^2}\left(||\fatF||^2_{L^2(0,T;L^2(\Omega))} + 2(C_S^{(1)})^4||\phi_h||^2_{L^\infty(0,T;H^1(\Omega))}||\mu_h||^2_{L^2(0,T;H^1(\Omega))}\right).
    \end{align}
\end{proof}

\begin{remark}\label{remark_apriori_estimates}
    Let $\phi_0\in H^2(\Omega)$ and $\fatu_0\in H^1_0(\Omega)^d\cap H^2(\Omega)^d$. If we choose $\phi_h^0$ to be the unique function in $\mathcal{Q}_h$ that coincides with $\phi_0$ in all vertices of the mesh and $\fatu_h^0$ as the unique function in $\mathcal{V}_h^d$ that coincides with $\fatu_0$ in all vertices and in all midpoints of edges of the mesh, then there exists a constant $C>0$ independent of $h$ such that 
    \[E_h^0\leq C E(\phi_0,\fatu_0), \qquad ||\phi_h^0||_{H^1(\Omega)} \leq C ||\phi_0||_{H^2(\Omega)} \qquad \text{and} \qquad ||\fatu_h^0||_{H^1(\Omega)} \leq C ||\fatu_0||_{H^2(\Omega)}\]    
    for all $h\in(0,H]$. Moreover, if the mobility function $M$ is bounded from below by some constant $\uM>0$, then $\uM_h \geq \uM$ for all $h\in(0,H]$. In particular, if all of the above assumptions are satisfied, then the constants $C_4,C_5$ in Corollary \ref{corollary_apriori_estimates_2} can be chosen independently of $h$.
\end{remark}

\subsubsection{Proof of Theorem \texorpdfstring{\ref{thm_properties_scheme}}{2.1}}\label{sec_proof}
We are now ready to prove Theorem \ref{thm_properties_scheme}. Testing \eqref{discrete_phi} with $\psi_h=1$ verifies assertion (\hyperlink{thm_i}{i}). The correctness of assertion (\hyperlink{thm_ii}{ii}) is an immediate consequence of Lemma \ref{lemma_discrete_energy_inequality} with $\lambda=1$. To prove assertion (\hyperlink{thm_iii}{iii}), we apply Schaefer's fixed point theorem (cf. Theorem \ref{schaefer}). To this end, we equip $X_h$ with the scalar product $\langle\cdot,\cdot\rangle_{X_h}:X_h\times X_h\to\reals$ defined via
\begin{align*}
    \big\langle(\mu_h,\Delta\phi_h,\fatu_h,p_h,r_h),(\psi_h,w_h,\fatv_h,q_h,s_h)\big\rangle_{X_h} &= \left\langle\mu_h,\psi_h\right\rangle_\Omega + \big\langle M(\phi_h^n)\gradx\mu_h,\gradx\psi_h\big\rangle_\Omega + \left\langle\Delta\phi_h,w_h\right\rangle_\Omega \\[2mm]
    &\phantom{=\;\;} + \big\langle \eta(\gammadot_h^n,\phi_h^n)\fatD{\fatu_h},\fatD{\fatv_h}\big\rangle_\Omega + \left\langle p_h, q_h\right\rangle_\Omega + \left\langle r_h, s_h\right\rangle_\Omega
\end{align*}
for all $(\mu_h,\Delta\phi_h,\fatu_h,p_h,r_h),(\psi_h,w_h,\fatv_h,q_h,s_h)\in X_h$. Next, let the map $\mathcal{S}_h:X_h\to X_h$ be such that 
\begin{align}
    &\left\langle\mathcal{S}_h(\mu_h,\Delta\phi_h,\fatu_h,p_h,r_h),(\psi_h,w_h,\fatv_h,q_h,s_h)\right\rangle_{X_h} = \sum_{k\,=\,1}^{23}\left\langle\mathcal{S}_h^{(k)}(\mu_h,\Delta\phi_h,\fatu_h,p_h,r_h),(\psi_h,w_h,\fatv_h,q_h,s_h)\right\rangle_{X_h} \notag \\[2mm]
    &\quad = \left\langle\mu_h,\psi_h\right\rangle_\Omega + \left\langle\left(\phi_h^n + \frac{\deltat}{2}\,\Delta\phi_h\right)\fatu_h,\gradx\psi_h\right\rangle_\Omega - \left\langle\Delta\phi_h,\psi_h\right\rangle_\Omega + \left\langle\mu_h,w_h\right\rangle_\Omega - \gamma\big\langle\gradx\!\left(\phi_h^n+\deltat\Delta\phi_h\right),\gradx w_h\big\rangle_\Omega \notag \\[2mm]
    &\quad \phantom{=\;\;} - \big\langle f'_{\mathrm{vex}}\!\left(\phi_h^n+\deltat\Delta\phi_h\right), w_h\big\rangle_\Omega - \big\langle f'_{\mathrm{cav}}(\phi_h^n), w_h\big\rangle_\Omega - \big\langle g'_{\mathrm{vex}}\!\left(\phi_h^n+\deltat\Delta\phi_h\right), w_h\big\rangle_{\Gamma} - \big\langle g'_{\mathrm{cav}}(\phi_h^n), w_h\big\rangle_{\Gamma} \notag \\[2mm]
    &\quad \phantom{=\;\;} - s\big\langle\nabla_{\Gamma}\!\left(\phi_h^n+\deltat\Delta\phi_h\right),\nabla_{\Gamma} w_h\big\rangle_{\Gamma} + \left\langle\Delta\phi_h, w_h\right\rangle_\Omega - \left\langle\Delta\phi_h, w_h\right\rangle_{\Gamma} + \frac{1}{4}\,\big\langle[(\fatu_h + \fatu_h^n)\cdot\gradx]\fatv_h, \fatu_h\big\rangle_\Omega \notag \\[2mm]
    &\quad \phantom{=\;\;} - \frac{1}{4}\,\big\langle[(\fatu_h + \fatu_h^n)\cdot\gradx]\fatu_h, \fatv_h\big\rangle_\Omega + \left\langle\fatF, \fatv_h\right\rangle_\Omega + \left\langle p_h, \divx(\fatv_h)\right\rangle_\Omega - \left\langle\left(\phi_h^n + \frac{\deltat}{2}\Delta\phi_h\right)\gradx\mu_h, \fatv_h\right\rangle_\Omega \notag \\[2mm]
    &\quad \phantom{=\;\;} - \left\langle\frac{\fatu_h-\fatu_h^n}{\deltat}, \fatv_h\right\rangle_\Omega + \left\langle p_h, q_h\right\rangle_\Omega - \left\langle\divx(\fatu_h), q_h\right\rangle_\Omega - \left\langle r_h, q_h\right\rangle_\Omega + \left\langle r_h,s_h\right\rangle_\Omega + \left\langle p_h, s_h\right\rangle_\Omega \label{reformulated_problem}
\end{align}
for all $(\mu_h,\Delta\phi_h,\fatu_h,p_h,r_h),(\psi_h,w_h,\fatv_h,q_h,s_h)\in X_h$. The existence and the uniqueness of $\mathcal{S}_h$ are guaranteed by the Riesz representation theorem (cf. Theorem \ref{riesz}). Moreover, since in finite-dimensional vector spaces all norms are equivalent (cf. Definition and Theorem \ref{thm_equivalence_of_norms}), it is easy to prove that $\mathcal{S}_h$ is continuous. For a detailed examination of $\mathcal{S}_h$ we refer to Appendix \ref{sec_sh}.

From the discussion above it follows that the equations \eqref{lambda_discrete_mu}--\eqref{lambda_discrete_r} defining Problem \ref{problem} can equivalently be written as
\begin{align}
    \left\langle\fatx_h^{n+1},\faty_h\right\rangle_{X_h} = \left\langle\lambda\mathcal{S}_h(\fatx_h^{n+1}),\faty_h\right\rangle_{X_h} \qquad \text{for all $\faty_h\in X_h$.} \label{fixed_point_problem}
\end{align}
More precisely, if $(\mu_h^{n+1},\phi_h^{n+1},\fatu_h^{n+1},p_h^{n+1},r_h^{n+1})\in X_h$ solves \eqref{lambda_discrete_mu}--\eqref{lambda_discrete_r}, then 
\begin{align*}
    \fatx_h^{n+1} = \left(\mu_h^{n+1},\frac{\phi_h^{n+1}-\phi_h^n}{\deltat},\fatu_h^{n+1},p_h^{n+1},r_h^{n+1}\right)
\end{align*}
solves \eqref{fixed_point_problem}. Conversely, if $\fatx_h^{n+1}=(\mu_h^{n+1},\Delta\phi_h^{n+1},\fatu_h^{n+1},p_h^{n+1},r_h^{n+1})\in X_h$ solves \eqref{fixed_point_problem}, then 
\begin{align*}
    (\mu_h^{n+1},\phi_h^{n}+\deltat\Delta\phi_h^{n+1},\fatu_h^{n+1},p_h^{n+1},r_h^{n+1}) 
\end{align*}
solves \eqref{lambda_discrete_mu}--\eqref{lambda_discrete_r}. Therefore, we may apply Corollary \ref{corollary_apriori_estimates_1} to deduce that the set
\begin{align*}
    \big\{\fatx\in X_h\,\big|\,\fatx=\lambda\mathcal{S}_h(\fatx)\;\text{for some $\lambda\in[0,1]$}\big\}
\end{align*}
is bounded. Consequently, Schaefer's fixed point theorem yields the existence of a solution to Problem \ref{problem} with $\lambda=1$ and thus to the numerical scheme. \hfill $\Box$

\subsection{Numerical simulations}
For the numerical simulations we restrict ourselves to the two-dimensional case ($d=2$). For the solution of the non-linear systems arising in each time step of the numerical scheme \eqref{discrete_phi}--\eqref{discrete_mean_value} we employ Newton's method which we abort as soon as the current Newton iterates $\phi_h^{k+1,*}, \mu_h^{k+1,*}, \fatu_h^{k+1,*}, p_h^{k+1,*}, r_h^{k+1,*}$ satisfy
\begin{align*}
    & ||(\phi_h^{k+1,*}-\phi_h^{k}, \mu_h^{k+1,*}-\mu_h^{k}, \fatu_h^{k+1,*}-\fatu_h^{k})||_{L^2(\Omega)} < 10^{-10} \\[4mm]
    \text{or} \qquad & \frac{||(\phi_h^{k+1,*}-\phi_h^{k}, \mu_h^{k+1,*}-\mu_h^{k}, \fatu_h^{k+1,*}-\fatu_h^{k})||_{L^2(\Omega)}}{||(\phi_h^{k+1,*}, \mu_h^{k+1,*}, \fatu_h^{k+1,*})||_{L^2(\Omega)}} < 10^{-9}.
\end{align*}
The linear systems arising in each Newton step are solved with a sparse-direct solver. For the implementation, we rely on the open source finite element software FEniCS \citep{fenics}.

\subsubsection{Channel flow} We consider the periodic channel with $L_1=1$ and $L_2=3$. The final time is $T=1000$ and the time step is $\deltat=0.01$. The mesh is given by $\mathcal{T}_h = \mathcal{T}^1_h \cup \mathcal{T}^2_h$, where
\begin{align*}
    \mathcal{T}^1_h &= \left\{\mathrm{conv}\!\left\{\left.\left(\frac{\ell}{180},\frac{m}{60}\right), \left(\frac{\ell}{180},\frac{m+1}{60}\right), \left(\frac{\ell+1}{180},\frac{m+1}{60}\right)\right\}\right| (\ell,m)\in\{0,\dots,179\}\times\{0,\dots,59\}\right\}, \\[2mm]
    \mathcal{T}^2_h &= \left\{\mathrm{conv}\!\left\{\left.\left(\frac{\ell}{180},\frac{m}{60}\right), \left(\frac{\ell+1}{180},\frac{m}{60}\right), \left(\frac{\ell+1}{180},\frac{m+1}{60}\right)\right\}\right| (\ell,m)\in\{0,\dots,179\}\times\{0,\dots,59\}\right\}.
\end{align*}
In particular, $h=\sqrt{2}/60$. Moreover, we choose 
\begin{align}
    \gamma=0.001, \qquad s=0.1, \qquad \fatF=(0.01,0)^T, \qquad M(\phi)=\frac{1}{16}\,\phi^2(1-\phi)^2. \label{parameters_and_functions}
\end{align}
The mixture viscosity $\eta$ is determined by a suitable interpolation of the MD data -- more specifically,\footnote{The prefactor $1/3375$ accounts for the different box sizes in the microscopic and macroscopic settings.}
\begin{align}
    \eta(\gammadot,\phi) &= \frac{1}{3375}\left\{
    \begin{array}{rl}
        \eta_0(\gammadot) & \text{if} \quad \phi < 0, \\[2mm]
        5(\eta_1(\gammadot) - \eta_0(\gammadot))\phi + \eta_0(\gammadot) & \text{if} \quad 0 \leq \phi < 0.2, \\[2mm]
        5(\eta_2(\gammadot) - \eta_1(\gammadot))(\phi-0.2) + \eta_1(\gammadot) & \text{if} \quad 0.2 \leq \phi < 0.4, \\[2mm]
        10(\eta_3(\gammadot) - \eta_2(\gammadot))(\phi-0.4) + \eta_2(\gammadot) & \text{if} \quad 0.4 \leq \phi < 0.5, \\[2mm]
        10(\eta_4(\gammadot) - \eta_3(\gammadot))(\phi-0.5) + \eta_3(\gammadot) & \text{if} \quad 0.5 \leq \phi < 0.6, \\[2mm]
        5(\eta_5(\gammadot) - \eta_4(\gammadot))(\phi-0.6) + \eta_4(\gammadot) & \text{if} \quad 0.6 \leq \phi < 0.8, \\[2mm]
        5(\eta_6(\gammadot) - \eta_5(\gammadot))(\phi-0.8) + \eta_5(\gammadot) & \text{if} \quad 0.8 \leq \phi < 1, \\[2mm]
        \eta_6(\gammadot) & \text{if} \quad \phi \geq 1.
    \end{array}\right. \label{viscosity}
\end{align}
Here, the functions $\eta_i$, $i\in\{0,\dots,6\}$, are obtained by a least-squares fitting of the MD data shown in Figure \ref{fig1}a to the Carreau-Yasuda model
\begin{align}
    \eta_i(\gammadot) = \eta_{i,\infty} + (\eta_{i,\infty}-\eta_{i,0})[1+(a_{i,2}\gammadot)^{a_{i,3}}]^{a_{i,1}}, \qquad i\in\{0,\dots,6\}. \label{carreau-yasuda}
\end{align}
The precise values of the parameters $\{\eta_{i,0}, \eta_{i,\infty}, a_{i,1}, a_{i,2}, a_{i,3}\}_{i\,\in\,\{0,\dots,6\}}$ can be found in Appendix \ref{sec_fitting}. The potential $f$ is chosen as 
\begin{align}
    f(\phi) &= \left\{\begin{array}{rl}
        f_{\mathrm{FH}}(\alpha) + f_{\mathrm{FH}}'(\alpha)(\phi-\alpha) + \frac{1}{2}f_{\mathrm{FH}}''(\alpha)(\phi-\alpha)^2 & \text{if $\phi<\alpha$,} \\[2mm]
        f_{\mathrm{FH}}(\phi) & \text{if $\alpha\leq\phi\leq 1-\alpha$,} \\[2mm]
        f_{\mathrm{FH}}(1-\alpha) + f_{\mathrm{FH}}'(1-\alpha)(\phi-(1-\alpha)) + \frac{1}{2}f_{\mathrm{FH}}''(1-\alpha)(\phi-(1-\alpha))^2 & \text{if $\phi>1-\alpha$,}
    \end{array}
    \right. \label{f_potential}
\end{align}
where 
\[f_{\mathrm{FH}}(\phi) = \frac{1}{N_{p_1}}\,\phi\ln(\phi) + \frac{1}{N_{p_2}}\,(1-\phi)\ln(1-\phi) + \chi\phi(1-\phi)\]
is the classical Flory-Huggins potential with the chain lengths $N_{p_1}=N_{p_2}=N=15$. Depending on the value of the Flory-Huggins interaction parameter $\chi$, $f_{\mathrm{FH}}$ has either one minimum (if $\chi\leq \chi_\mathrm{crit}=2/15$) or two minima (if $\chi > \chi_\mathrm{crit}=2/15$). We set
\[(\phi_\star, \phi^\star) = (\min\underset{{\phi\,\in\,[0,1]}}{\mathrm{argmin}}(f_{\mathrm{FH}}(\phi)), \max\underset{{\phi\,\in\,[0,1]}}{\mathrm{argmin}}(f_{\mathrm{FH}}(\phi))).\]
Since $f_{\mathrm{FH}}$ is symmetric with respect to the axis $\phi=0.5$, $\phi^\star = 1-\phi_\star$. The cut-off position $\alpha$ is chosen such that $f$ has the same minima as $f_{\mathrm{FH}}$, i.e. $\alpha\in(0,\phi_\star]$. Note that $f$ possesses the convex-concave splitting $f = f_{\mathrm{vex}} + f_{\mathrm{cav}}$ with
\begin{align*}
    f_{\mathrm{cav}}(\phi) = \chi\phi(1-\phi), \qquad f_{\mathrm{vex}}(\phi) = f(\phi)-f_{\mathrm{cav}}(\phi)\,.
\end{align*}
The potential $g$ is chosen as
\begin{align}
    g(\phi) = f(\phi_\star) + \frac{1}{2}f''(\phi_\star)(\phi-\phi_\star)^2 \label{g_potential}
\end{align}
with the convex-concave splitting $(g_{\mathrm{vex}},g_{\mathrm{cav}})=(g,0)$. The initial velocity reads $\fatu_h^0=\bm{0}$. $\phi_h^0$ is defined as the unique function in $\mathcal{Q}_h$ that satisfies $\phi_h^0(\fatx) = 0.5+\xi_{\fatx}$ for all $\fatx\in V_h$. Here, $V_h$ denotes the set of all vertices in $\mathcal{T}_h$ and $\{\xi_{\fatx}\,|\,\fatx\,\in\,V_h\}$ is a family of random numbers that are uniformly distributed in $[-0.001,0.001]$.

\vspace{1cm}

\noindent \textbf{Results for $\chi=\ln(3)/6$.} In this case the potential function $f$ has minima at $\phi_\star=0.1$ and $\phi^\star=0.9$ (cf. the green circles in Figure \ref{err_minima0109}). As shown in Figure \ref{phi_minima0109}, the two components start to separate immediately. However, it takes some time until the level of separation reaches its maximum. Figure \ref{u_minima0109} shows the behaviour of the velocity $\fatu_h$. While the colours encode the values of $||\fatu_h||_2$, the black arrows indicate the direction of the flow. Since the flow is non-Newtonian, we obtain a flow profile that slightly deviates from the parabolic profile of Newtonian flows. We also note that the mass error as well as the mean divergence error are smaller than the threshold for the Newton error (cf. Figure \ref{err_minima0109}). 
\begin{center}
    \includegraphics[width=0.40\textwidth]{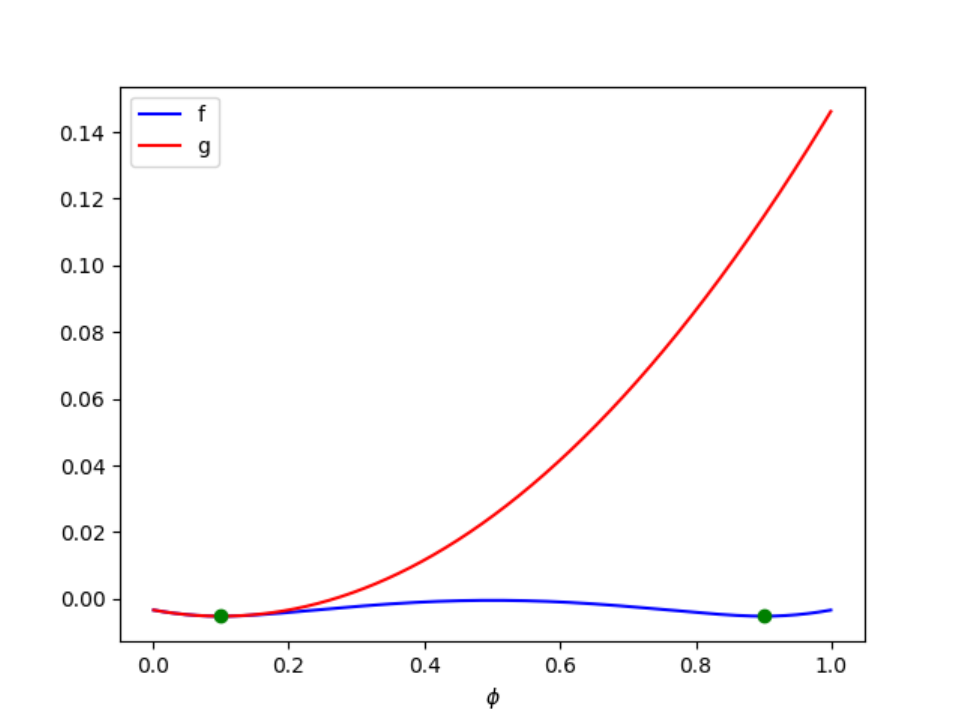} \includegraphics[width=0.40\textwidth]{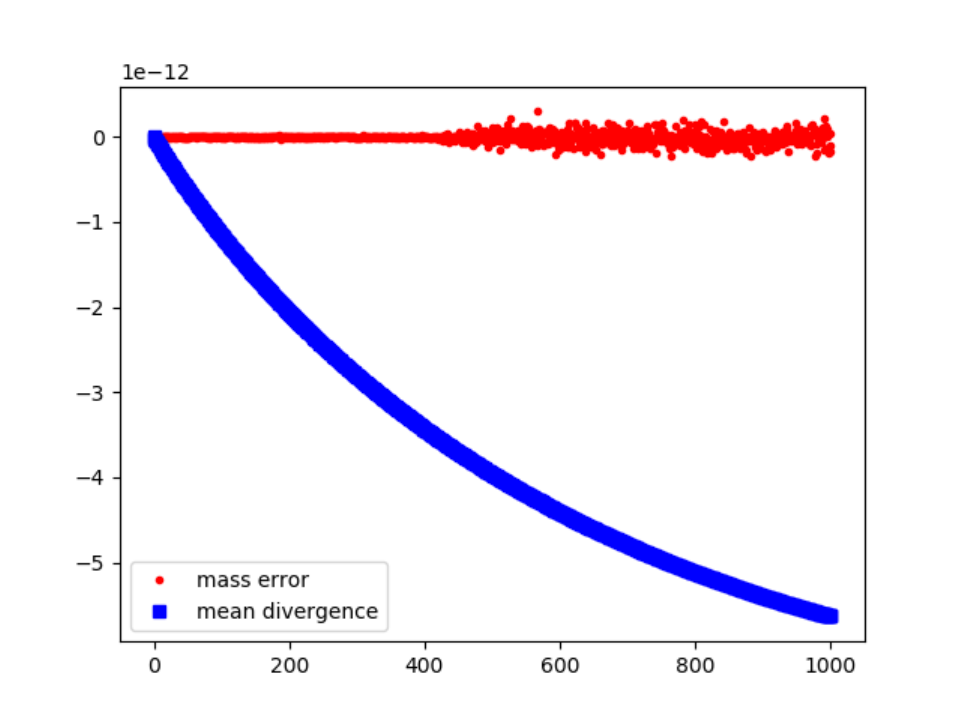} \\
    \refstepcounter{figure}\label{err_minima0109}
    Figure \arabic{figure}: The potential functions $f$ and $g$ (left) and the mass/mean divergence error (right).
\end{center}

\begin{center}
\includegraphics[width=0.7\textwidth]{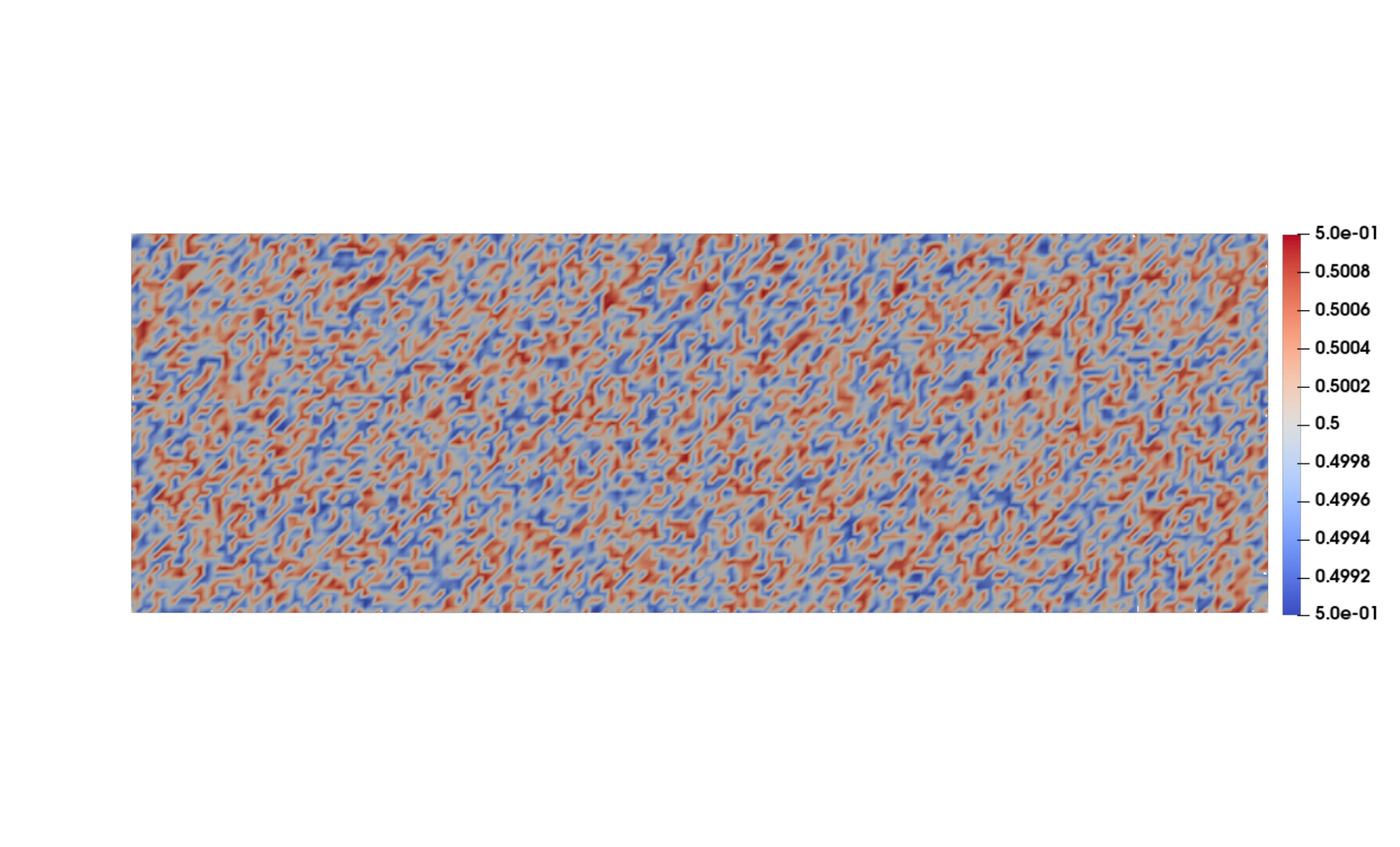} \\[-3.5cm]
\includegraphics[width=0.7\textwidth]{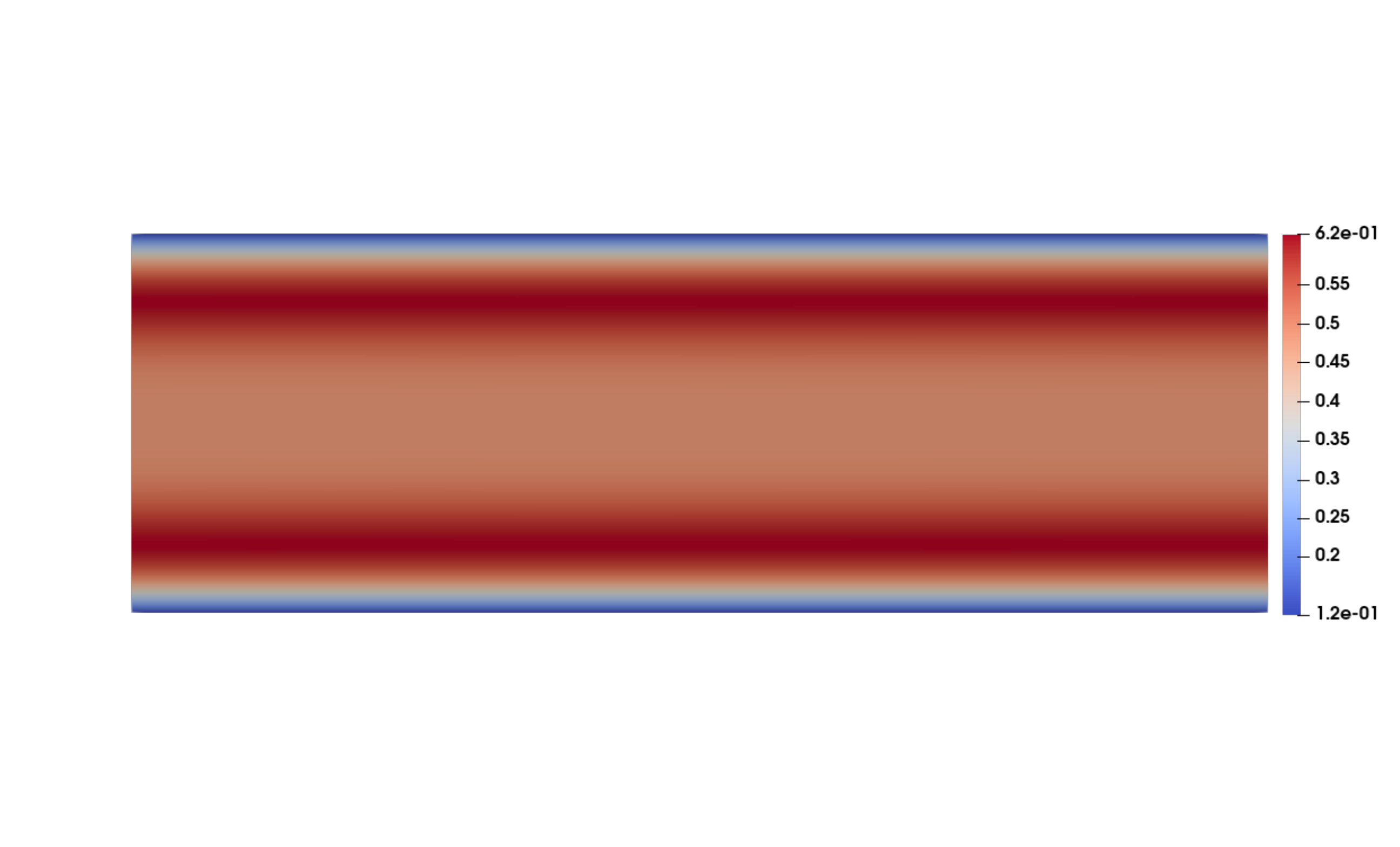} \\[-3.5cm]
\includegraphics[width=0.7\textwidth]{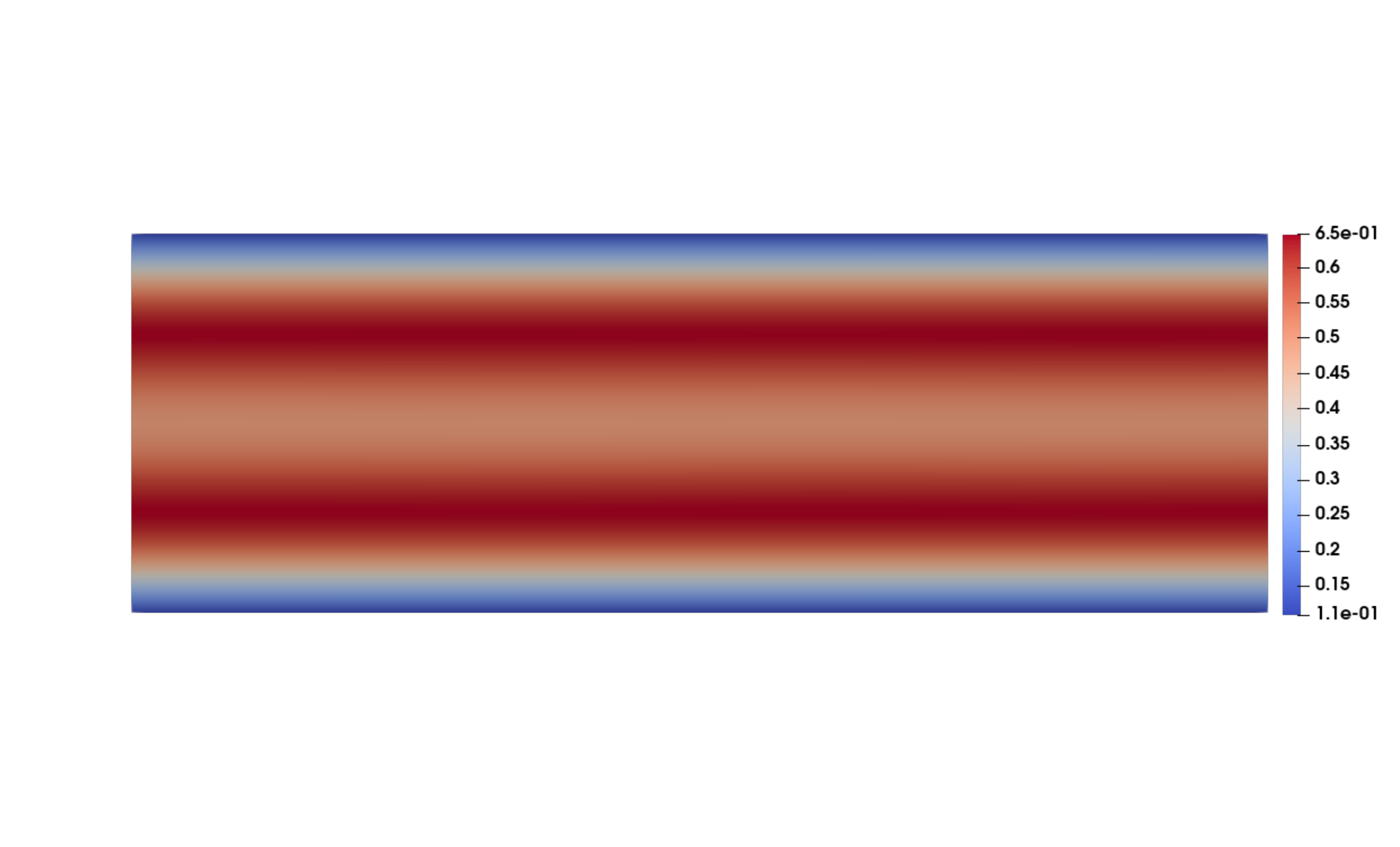} \\[-3.5cm]
\includegraphics[width=0.7\textwidth]{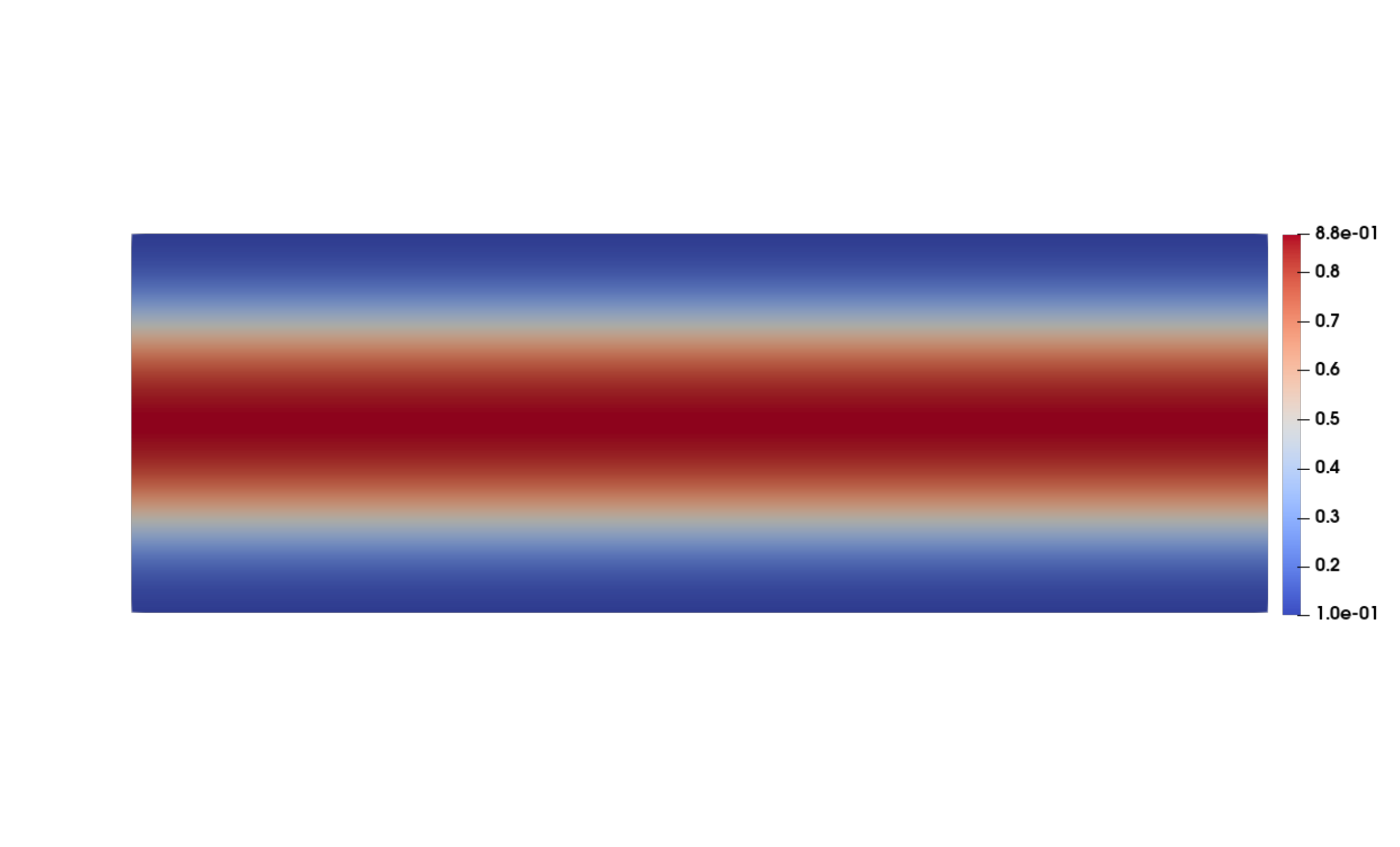} \\[-3.5cm]
\includegraphics[width=0.7\textwidth]{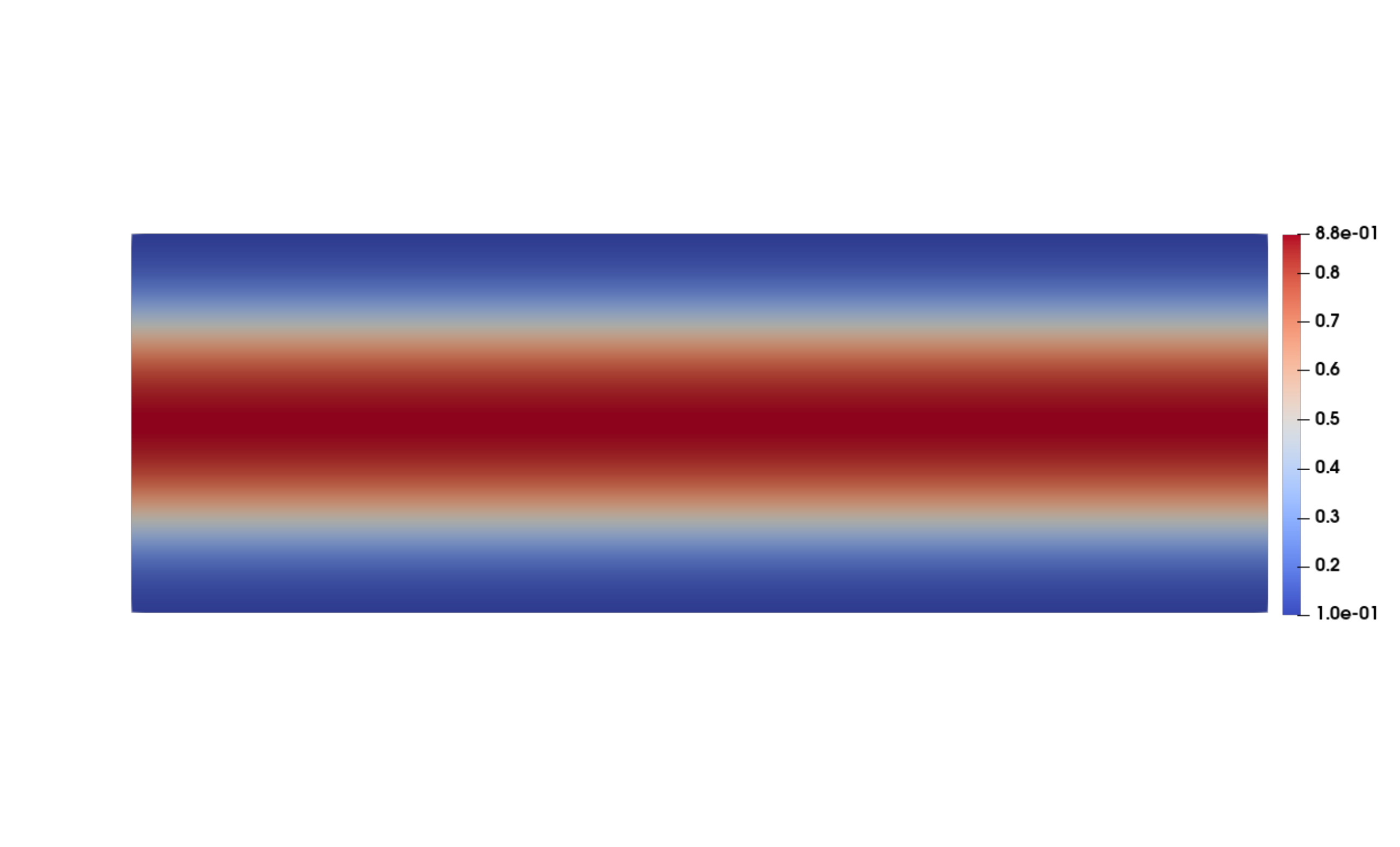} \\[-1cm]
\refstepcounter{figure}\label{phi_minima0109} 
Figure \arabic{figure}: $\phi_h(t,\cdot)$ at times $t=0.00,10.00,50.00,500.00,1000.00$.
\end{center}

\begin{center}
\includegraphics[width=0.7\textwidth]{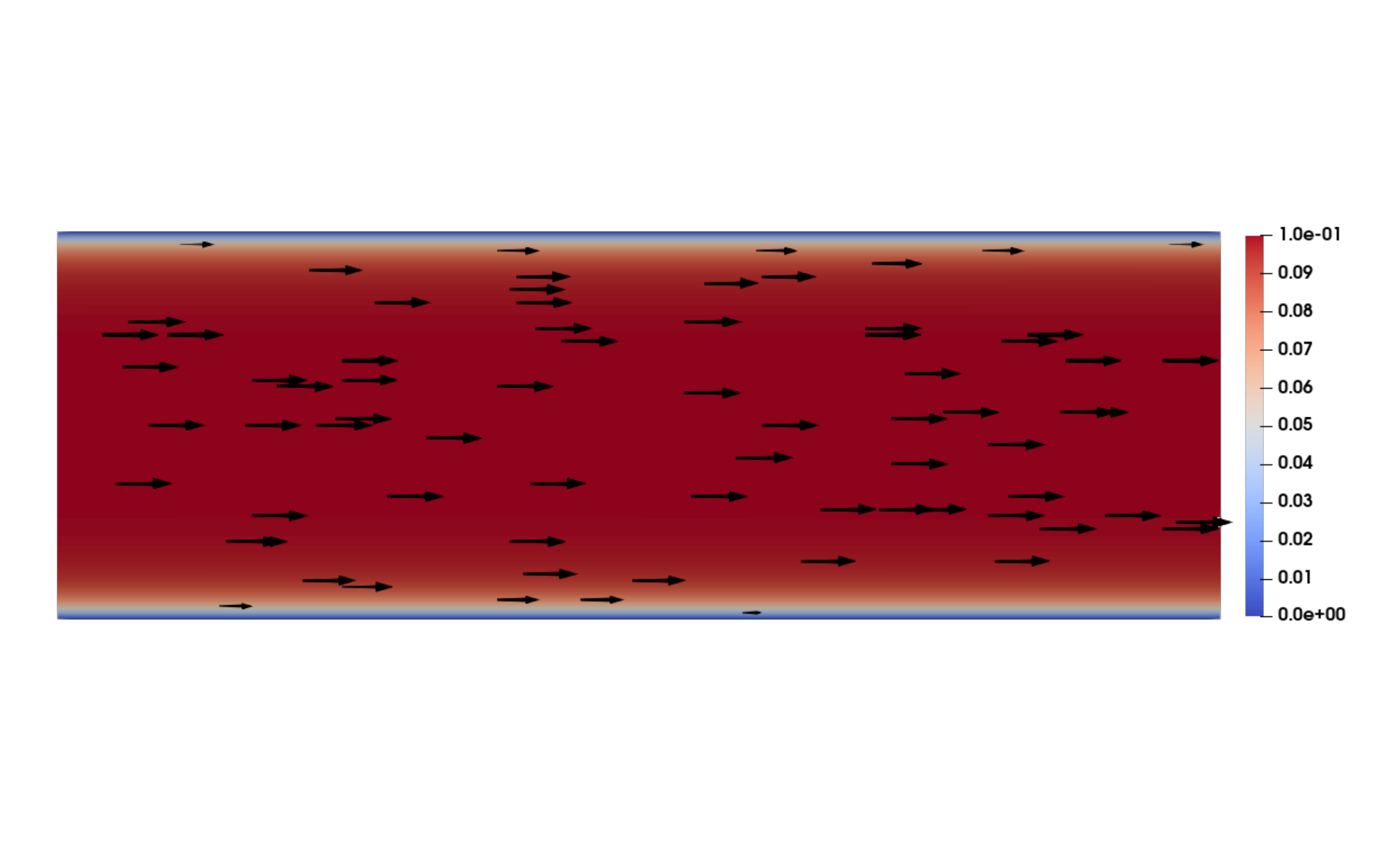}\\[-3.5cm]
\includegraphics[width=0.7\textwidth]{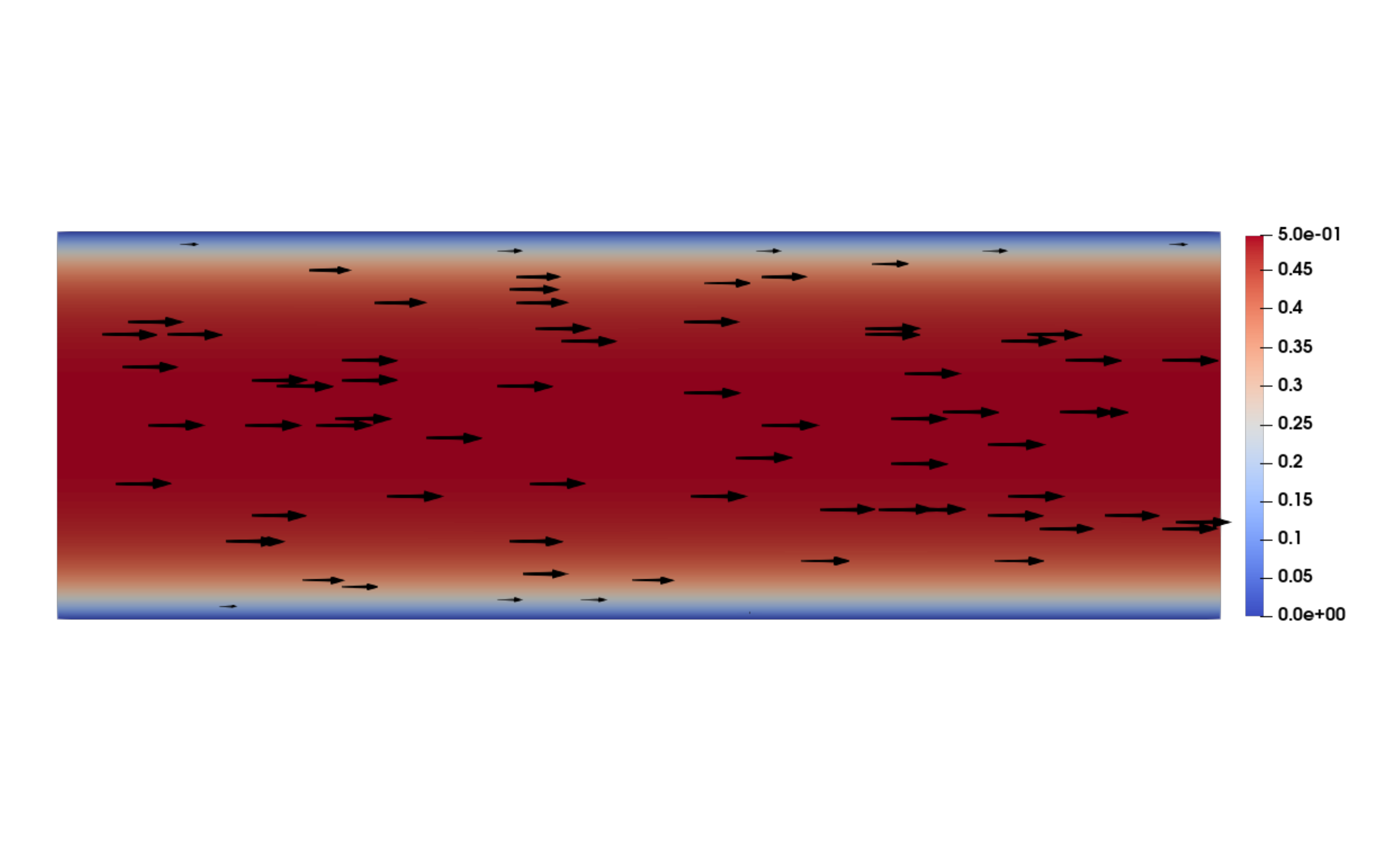}\\[-3.5cm]
\includegraphics[width=0.7\textwidth]{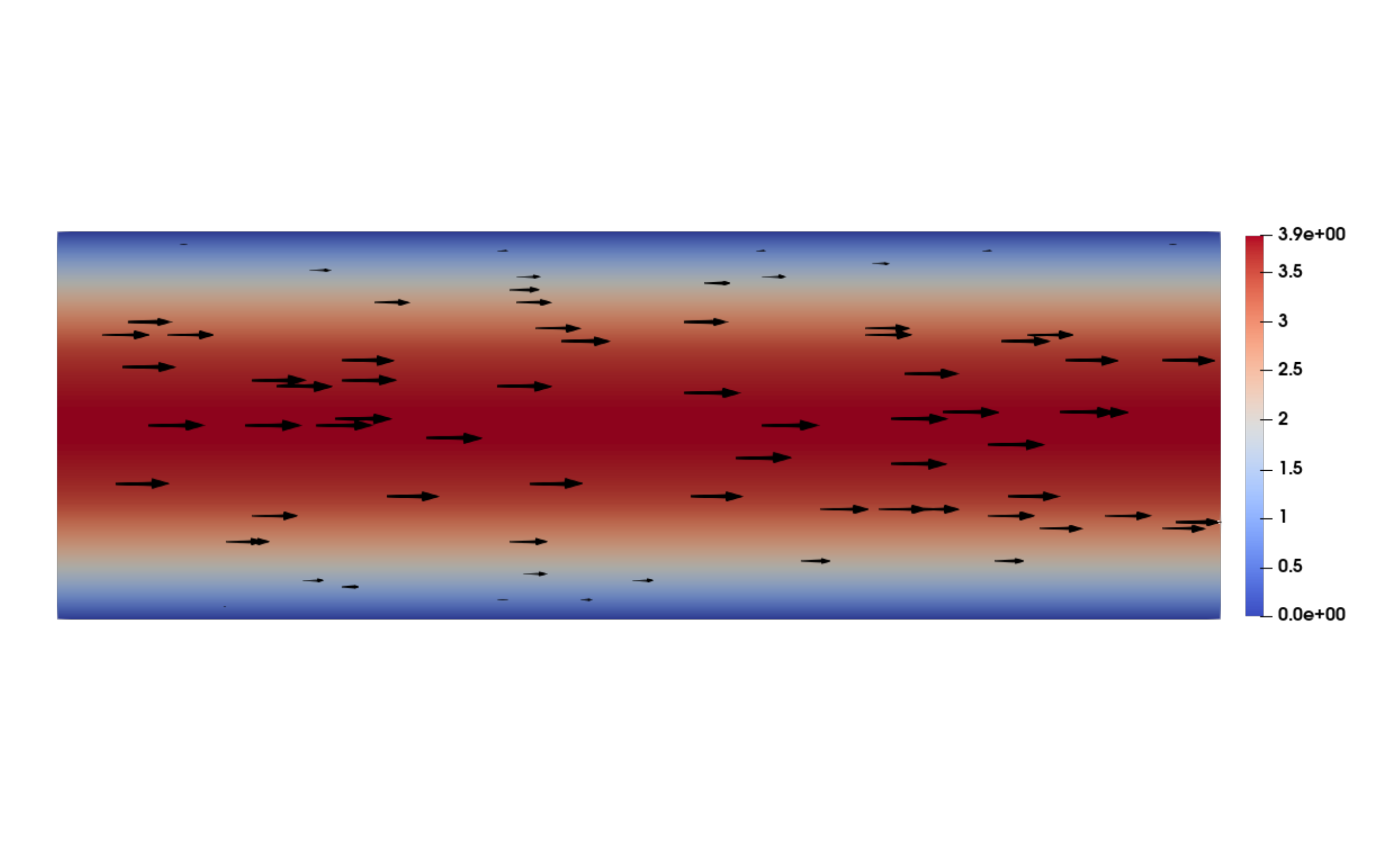}\\[-3.5cm]
\includegraphics[width=0.7\textwidth]{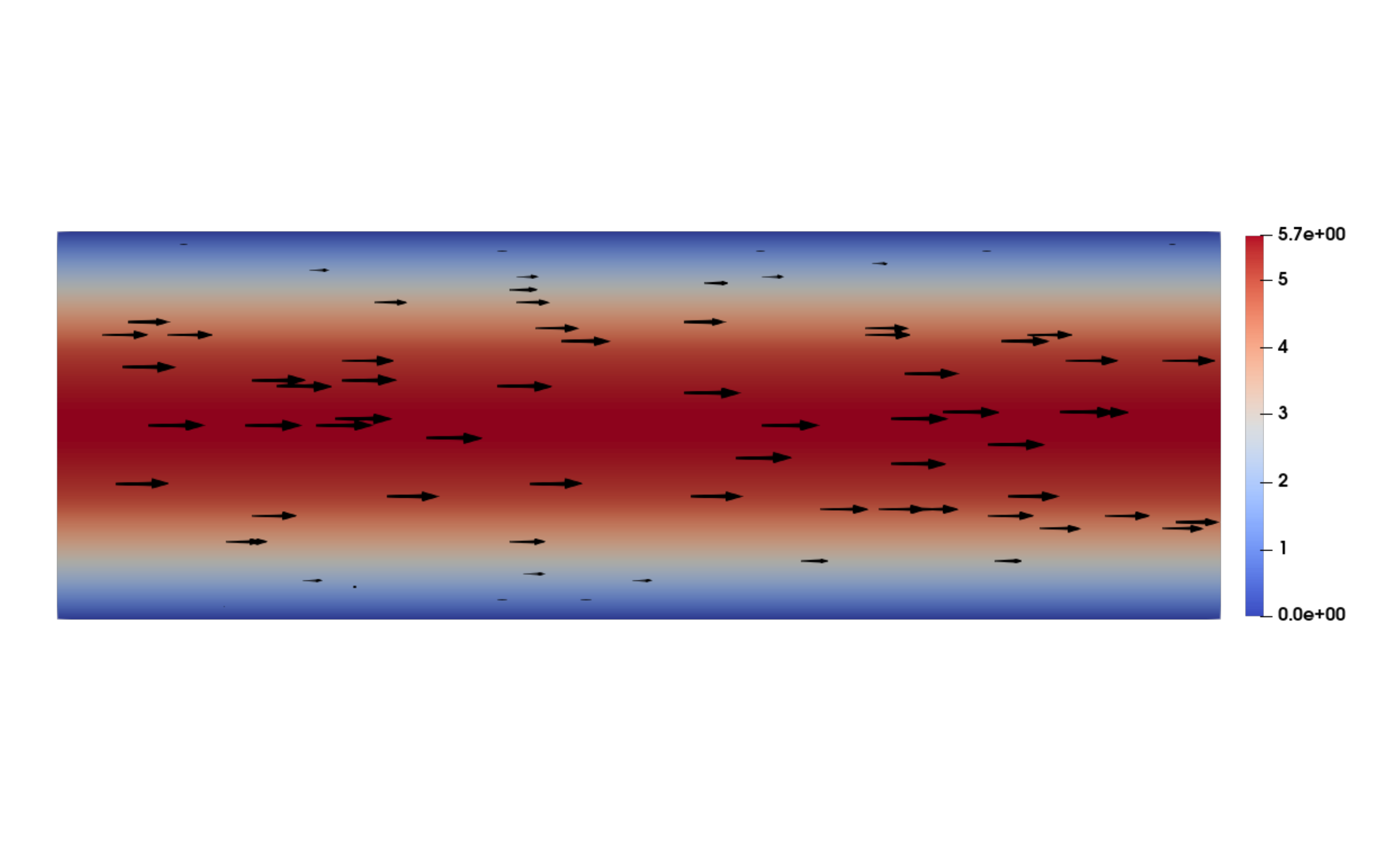}
\end{center}
\vspace{-1.5cm}
\refstepcounter{figure}\label{u_minima0109}
Figure \arabic{figure}: $||\fatu_h(t,\cdot)||_2$ at times $t=
10.00,50.00,500.00,1000.00$. The arrows indicate the flow direction and are scaled relative to $\max_{\,\overline{\Omega}}\,\{||\fatu_h(t,\cdot)||_2\}$ at the corresponding times. \hfill \phantom{.}

\vspace{1cm}

\noindent \textbf{Results for $\chi=\chi_{\mathrm{crit}}+0.001=\frac{2}{15}+0.001$.} In this case, the Flory-Huggins interaction parameter is slightly above the critical value $\chi_{\mathrm{crit}}=2/15$, whence the potential function $f$ has two minima that are close to $0.5$ (cf. the green circles in Figure \ref{errors_chi01343}). Again, we can observe a clear separation behaviour of the components (cf. Figure \ref{phi_chi01343}). However, the level of separation is much lower. We do not show snapshots of the velocity since they are very similar to those in the previous case. Again, the mass error and the mean divergence error are below the threshold for the Newton error (cf. Figure \ref{errors_chi01343}).

\begin{center}
\includegraphics[width=0.40\textwidth]{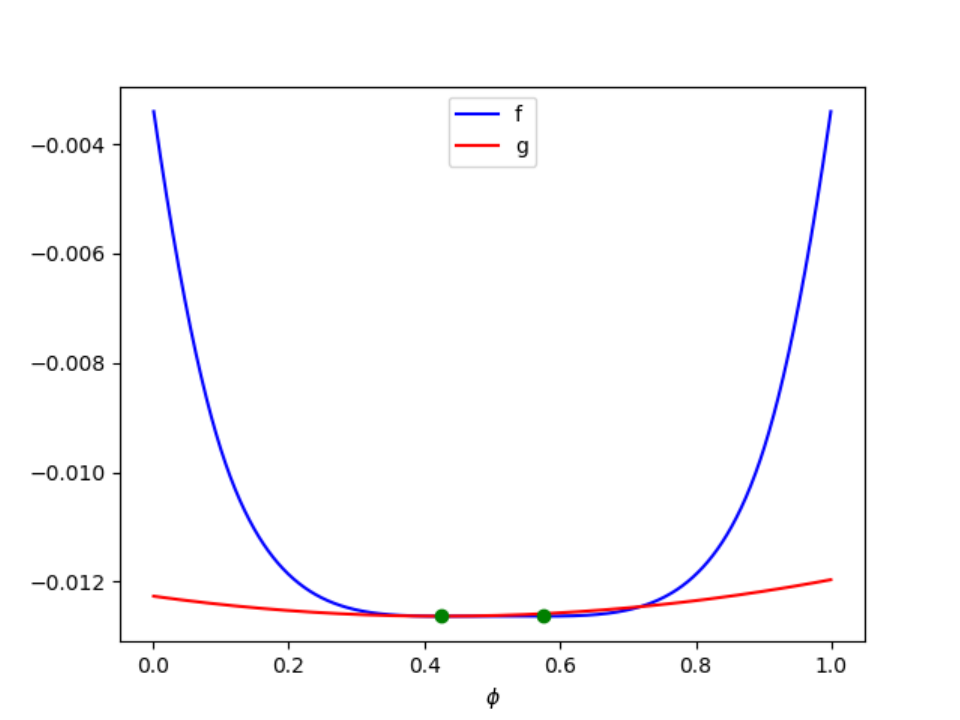} \includegraphics[width=0.40\textwidth]{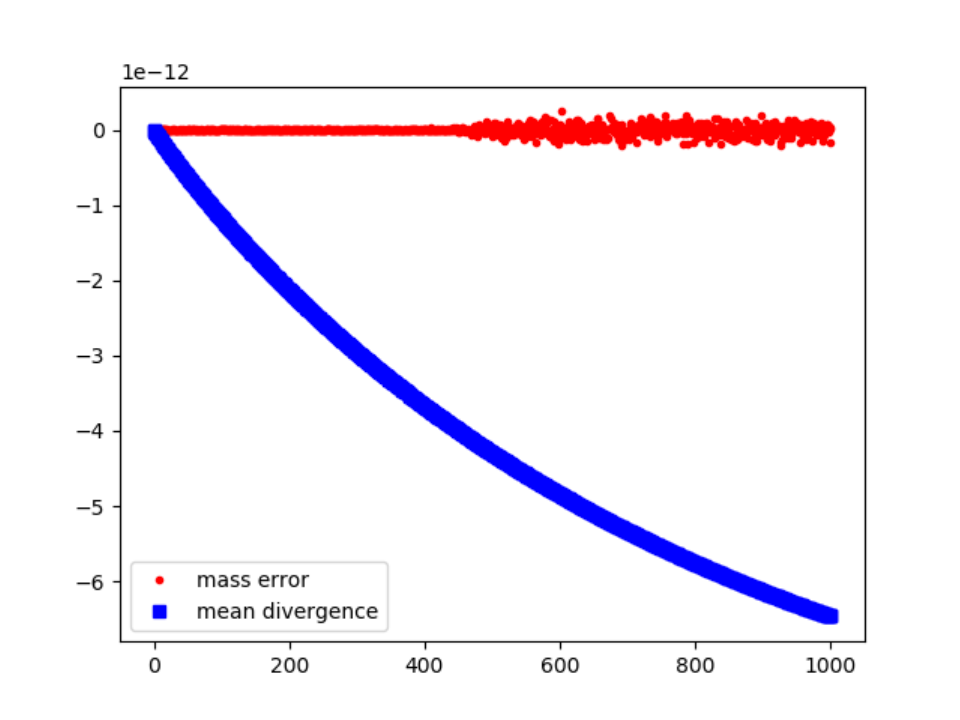} \\
\refstepcounter{figure}\label{errors_chi01343}
Figure \arabic{figure}: The potential functions $f$ and $g$ (left) and the mass/mean divergence error (right).
\end{center}

\begin{center}
\includegraphics[width=0.7\textwidth]{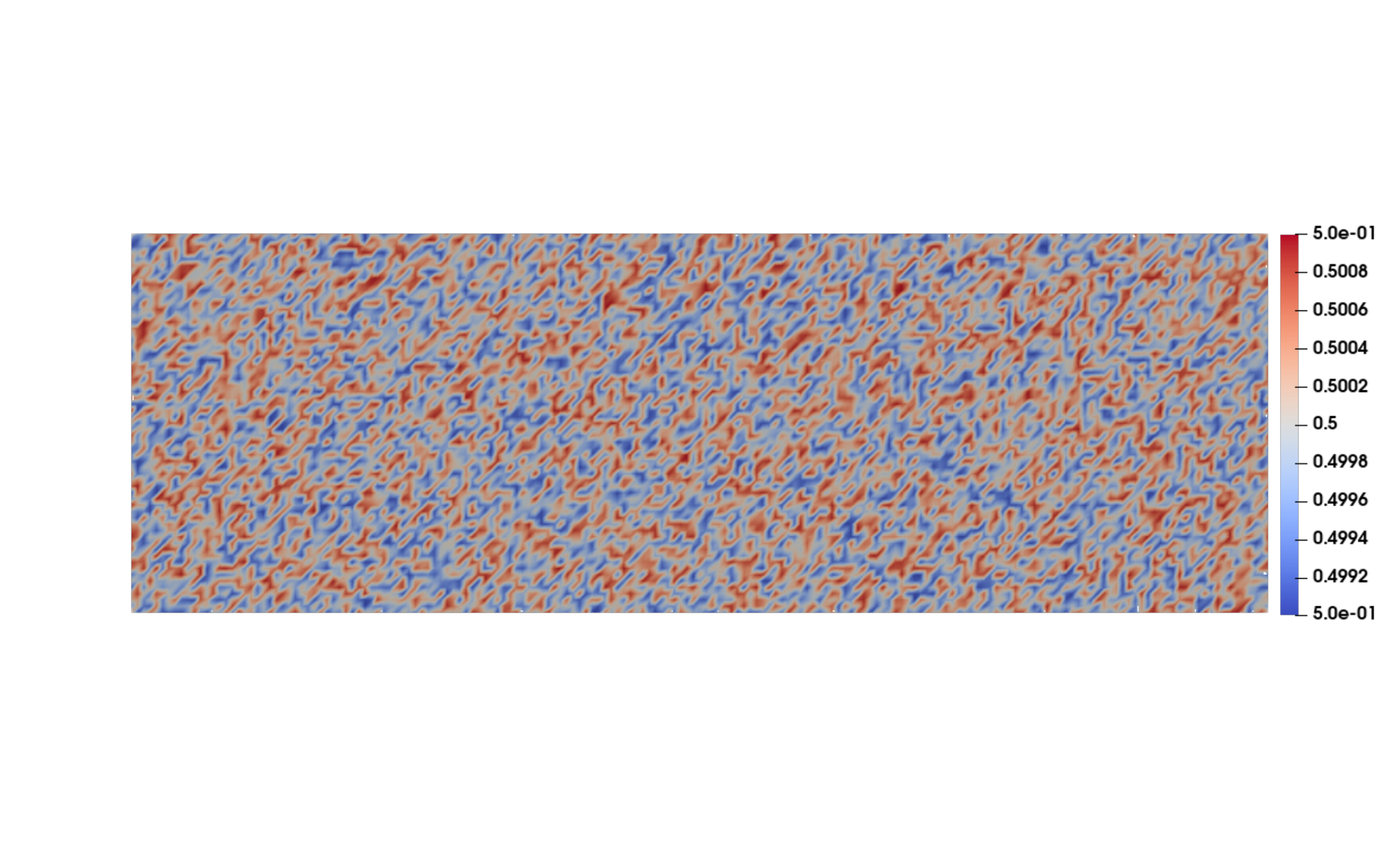}\\[-3.5cm]
\includegraphics[width=0.7\textwidth]{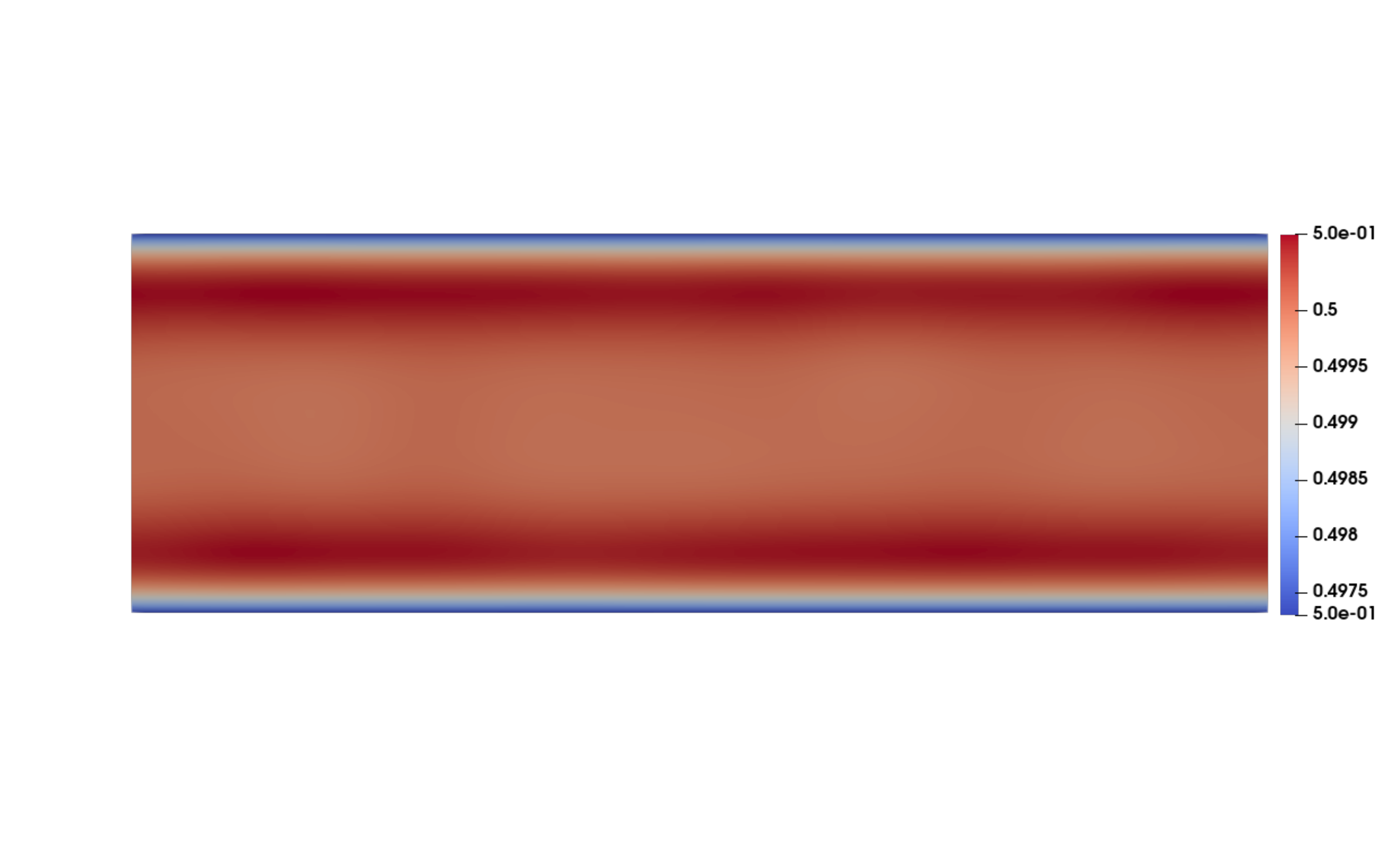}\\[-3.5cm]
\includegraphics[width=0.7\textwidth]{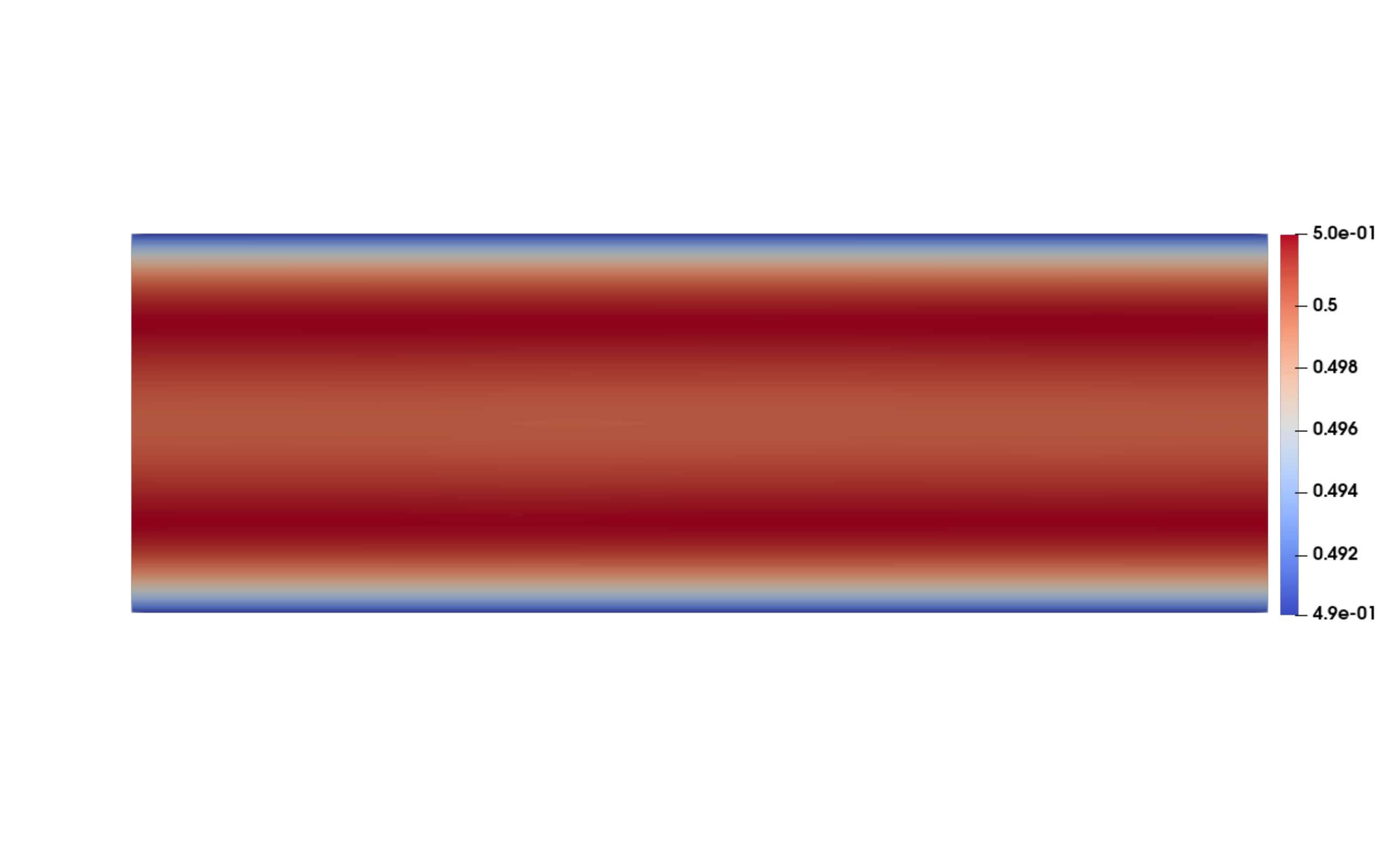}\\[-3.5cm]
\includegraphics[width=0.7\textwidth]{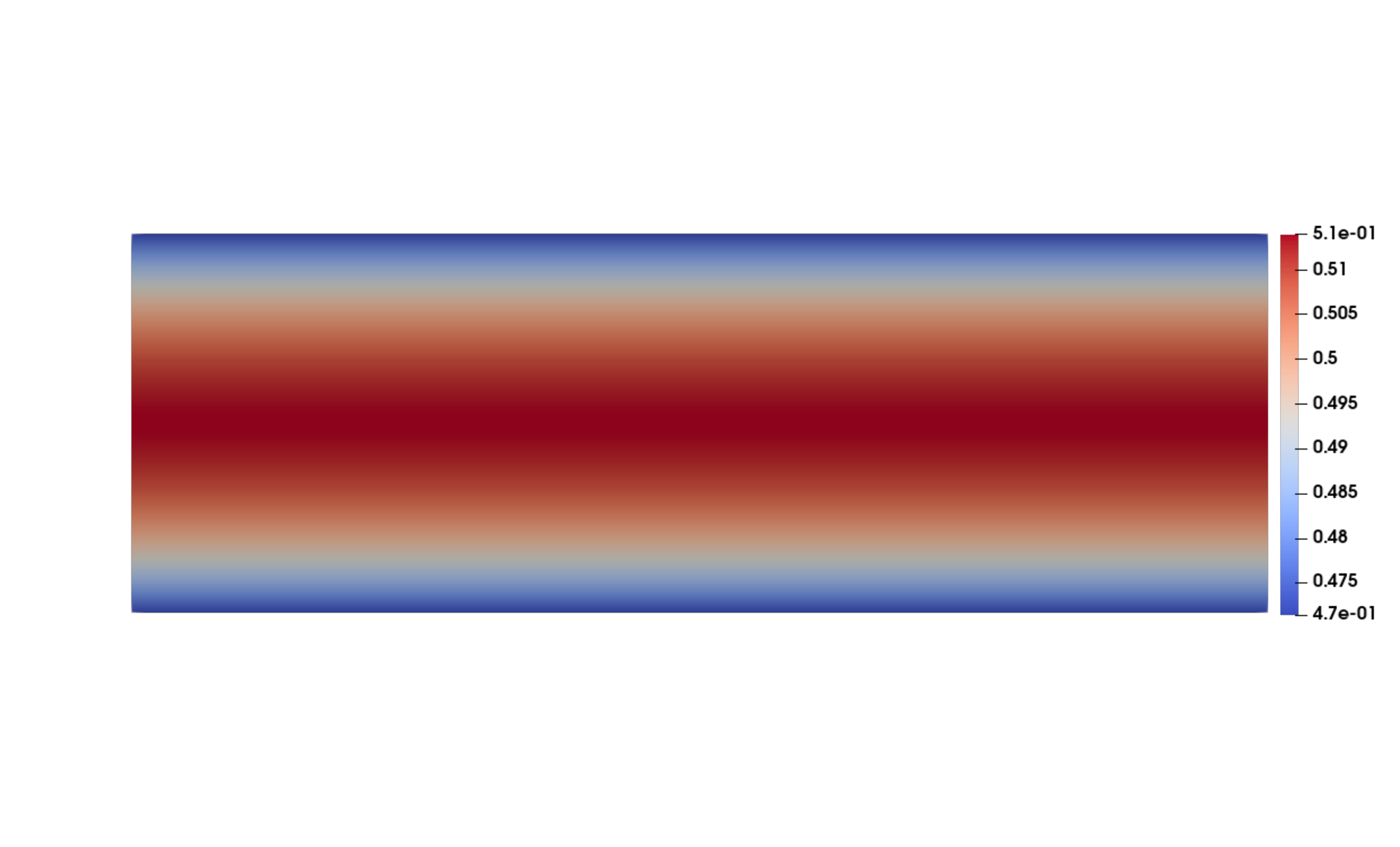}\\[-3.5cm]
\includegraphics[width=0.7\textwidth]{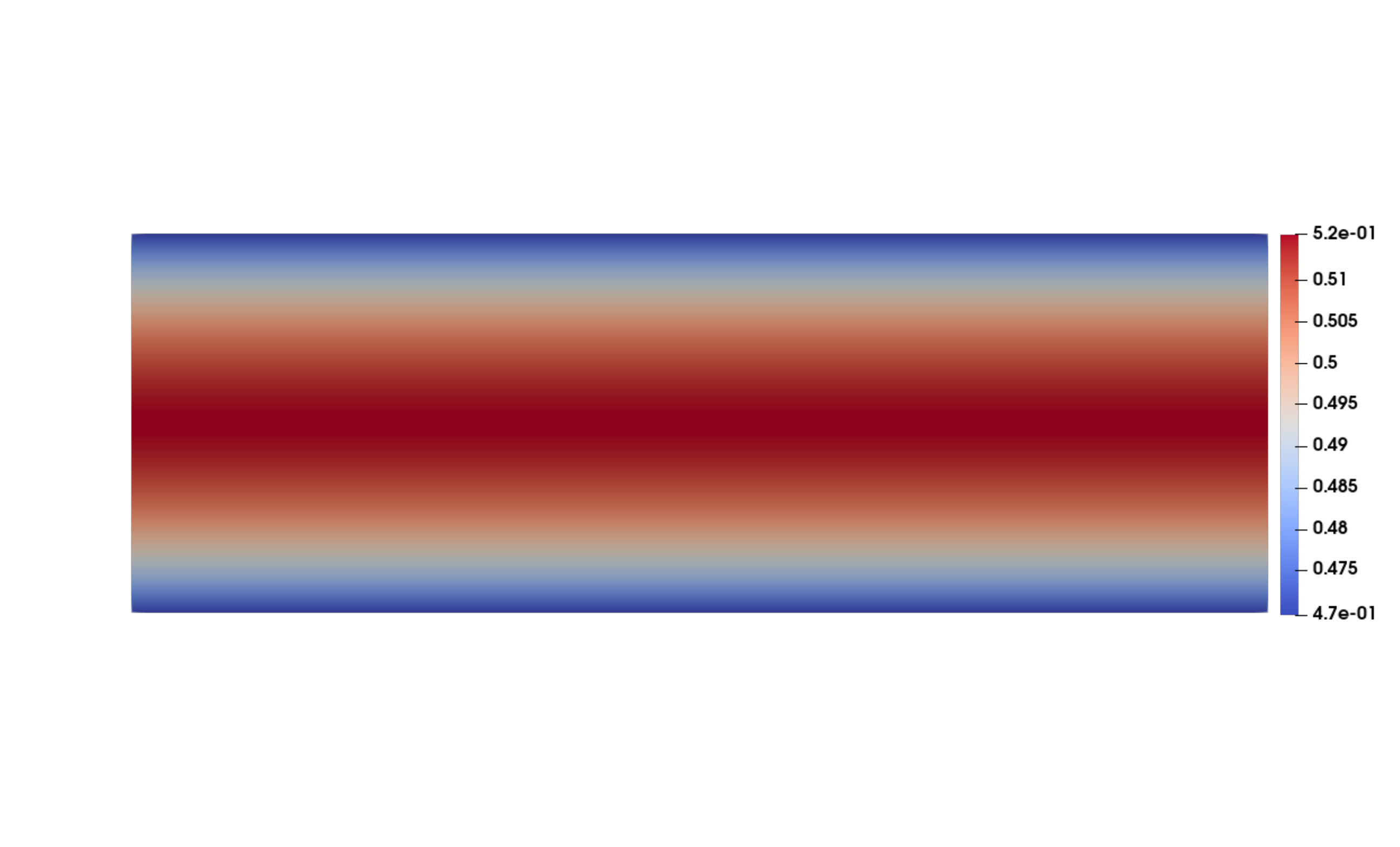} \\[-1cm]
\refstepcounter{figure}\label{phi_chi01343}
Figure \arabic{figure}: $\phi_h(t,\cdot)$ at times $t=0.00,10.00,50.00,500.00,1000.00$
\end{center}


\vspace{1cm}

\noindent \textbf{Results for $\chi=\frac{2}{15}-0.001$.} In this case, the Flory-Huggins interaction parameter is slightly below the critical value $\chi_{\mathrm{crit}}=2/15$, whence the potential function $f$ has only one minimum at $\phi_\star=\phi^\star=0.5$ (cf. the green circle in Figure \ref{errors_chi01323}). As a consequence, the components do not separate (cf. Figure \ref{phi_chi01323}). Again, we do not show snapshots of the velocity since they are very similar to those in the first case. The mass error and the mean divergence error are below the threshold for the Newton error (cf. Figure \ref{errors_chi01323}).

\begin{center}
\includegraphics[width=0.40\textwidth]{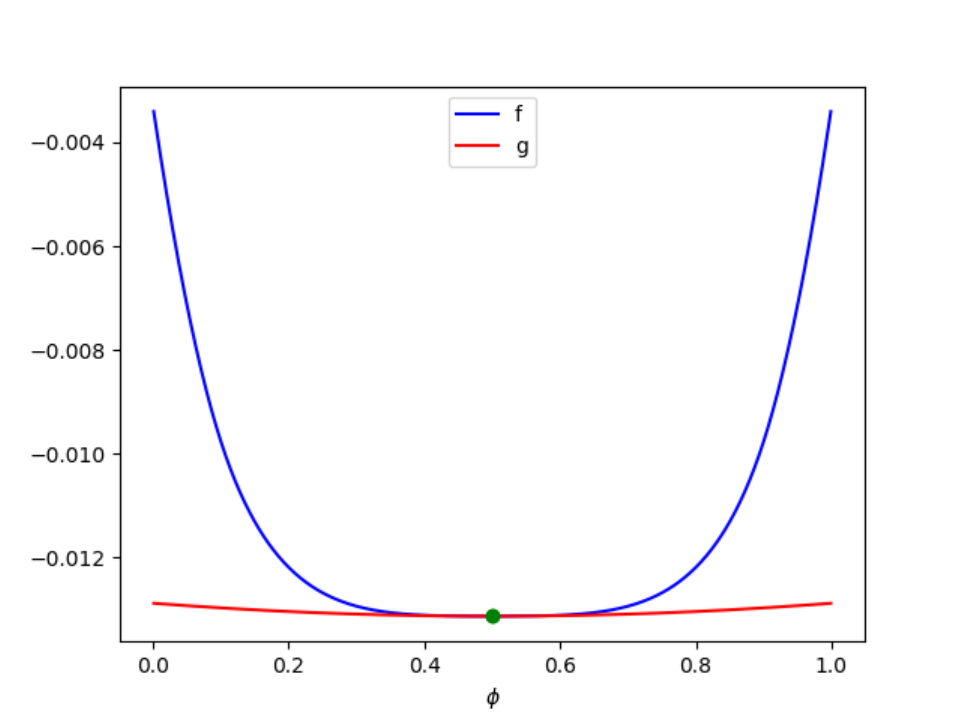} \includegraphics[width=0.40\textwidth]{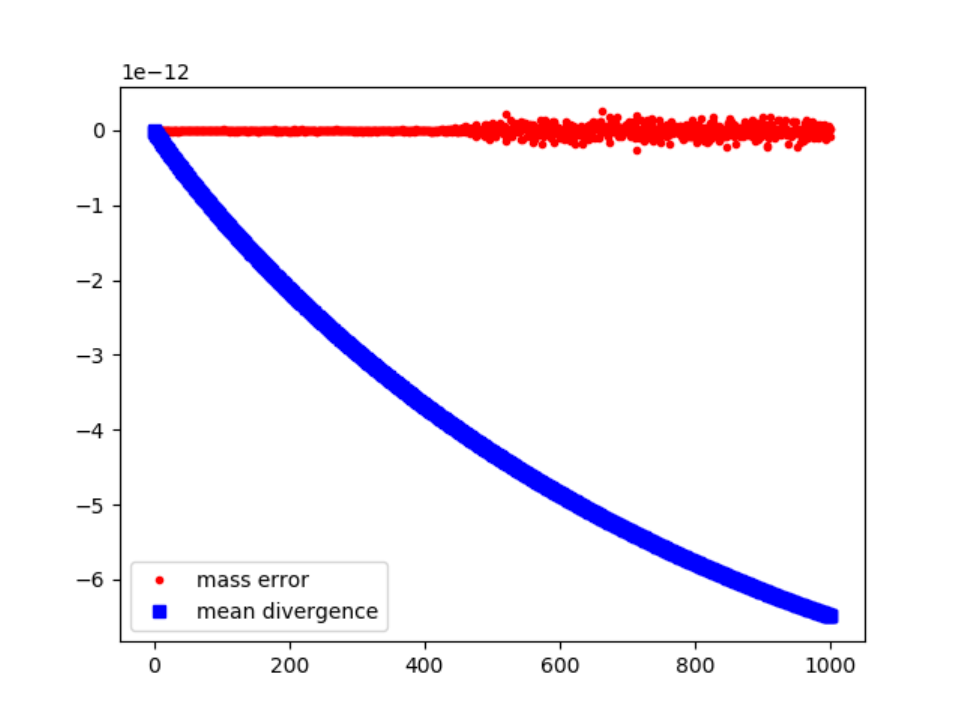} \\
\refstepcounter{figure}\label{errors_chi01323}
Figure \arabic{figure}: The potential functions $f$ and $g$ (left) and the mass/mean divergence error (right).
\end{center}

\begin{center}
\includegraphics[width=0.7\textwidth]{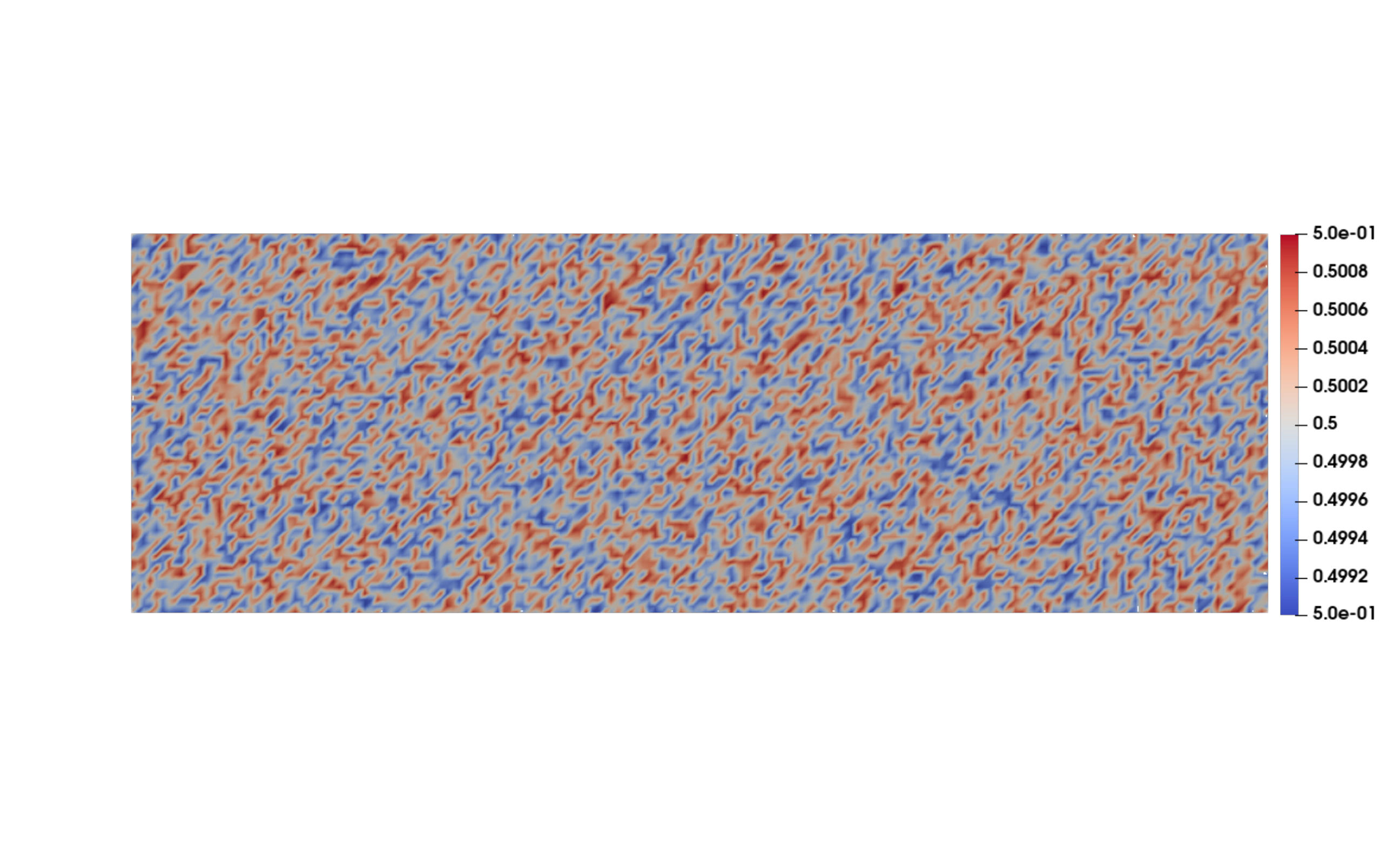}\\[-3.5cm]
\includegraphics[width=0.7\textwidth]{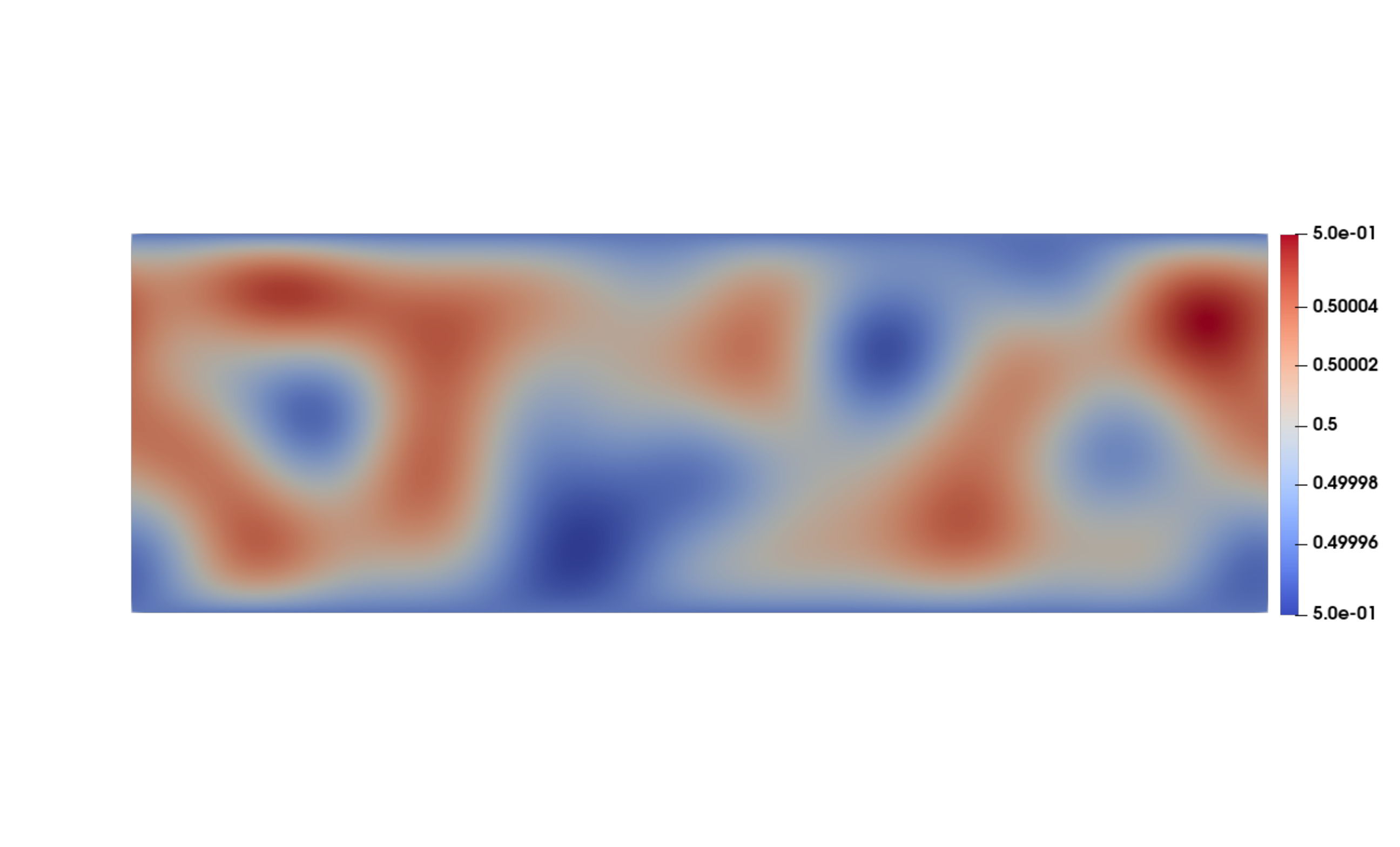}\\[-3.5cm]
\includegraphics[width=0.7\textwidth]{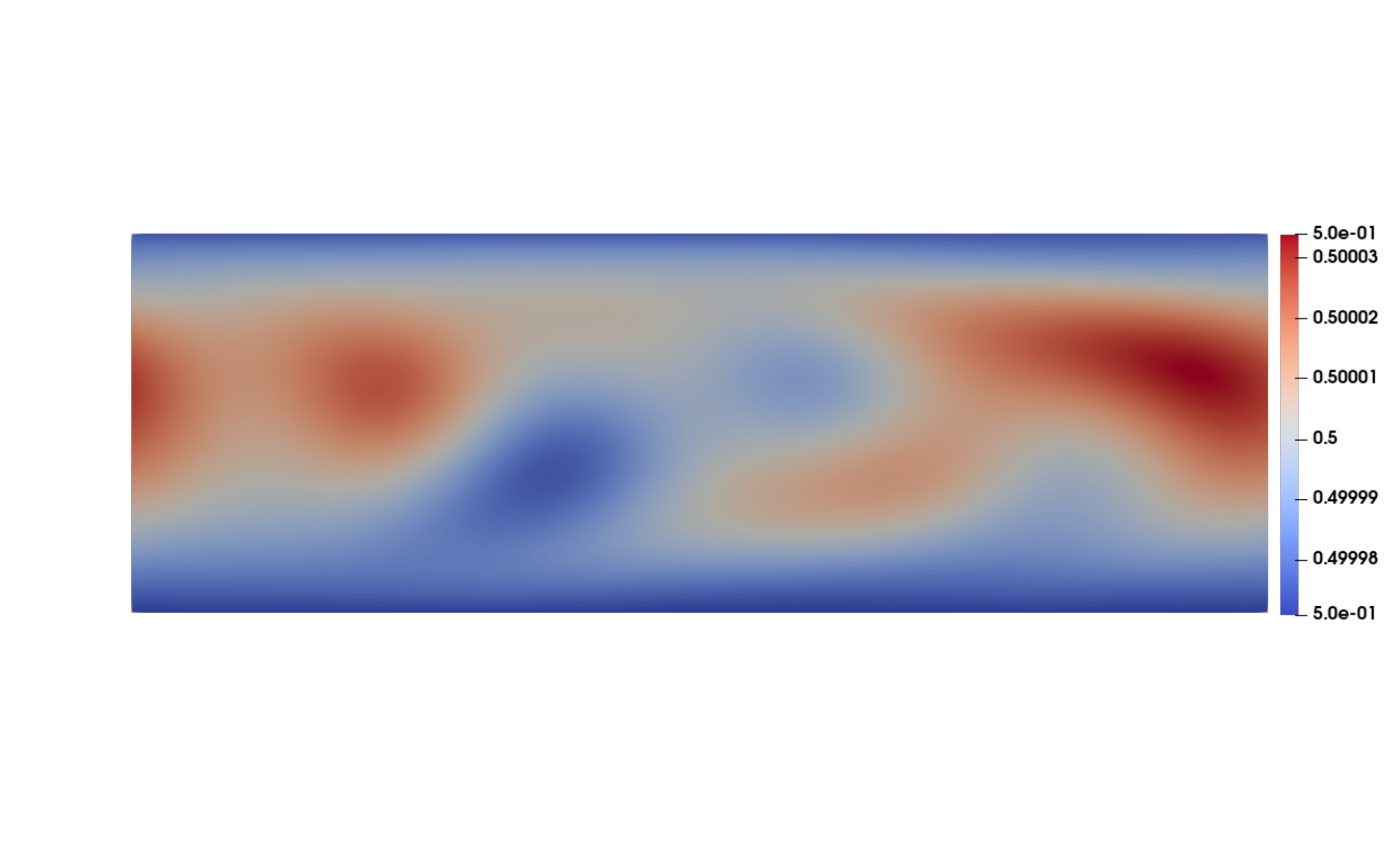}\\[-3.5cm]
\includegraphics[width=0.7\textwidth]{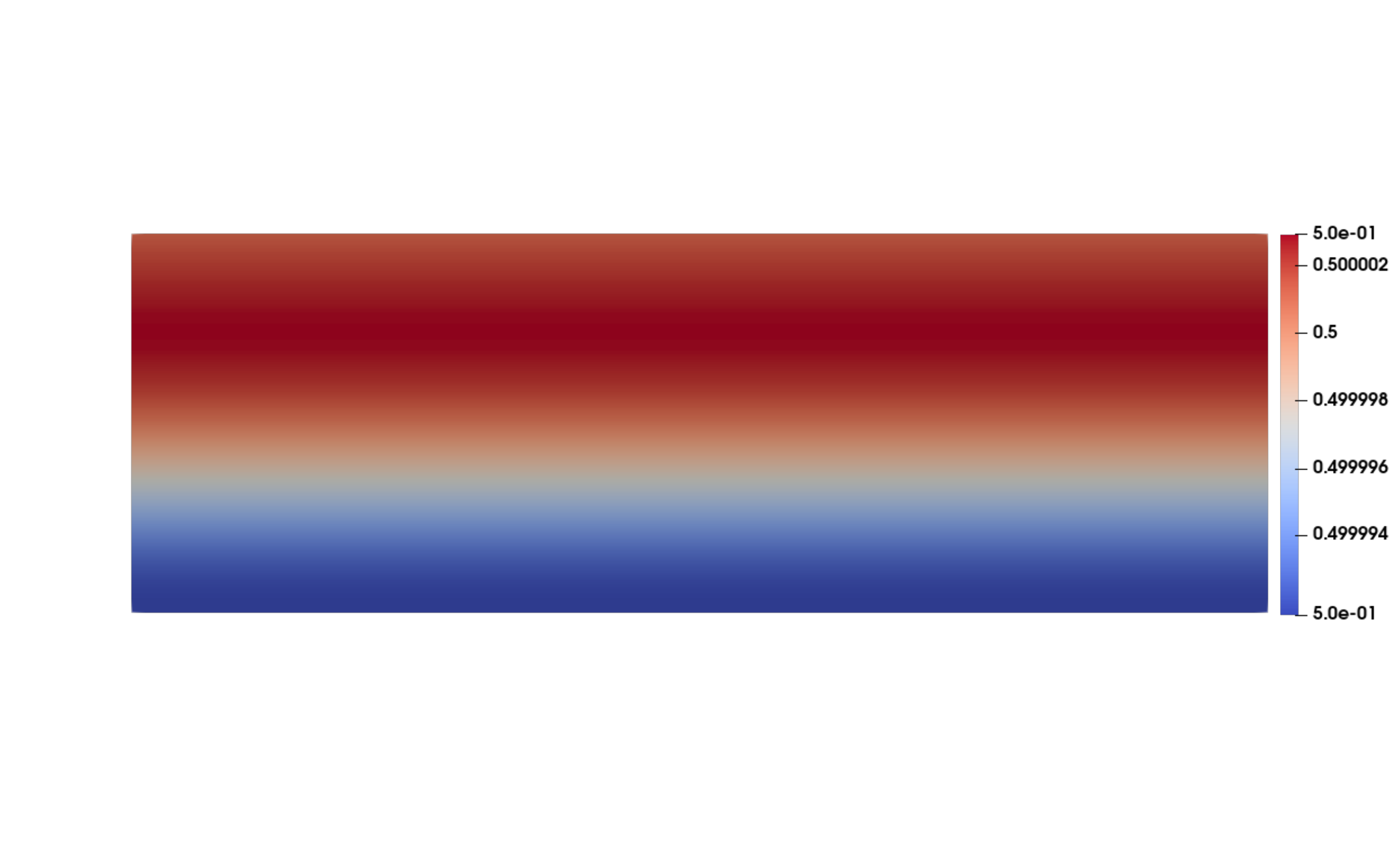}\\[-3.5cm]
\includegraphics[width=0.7\textwidth]{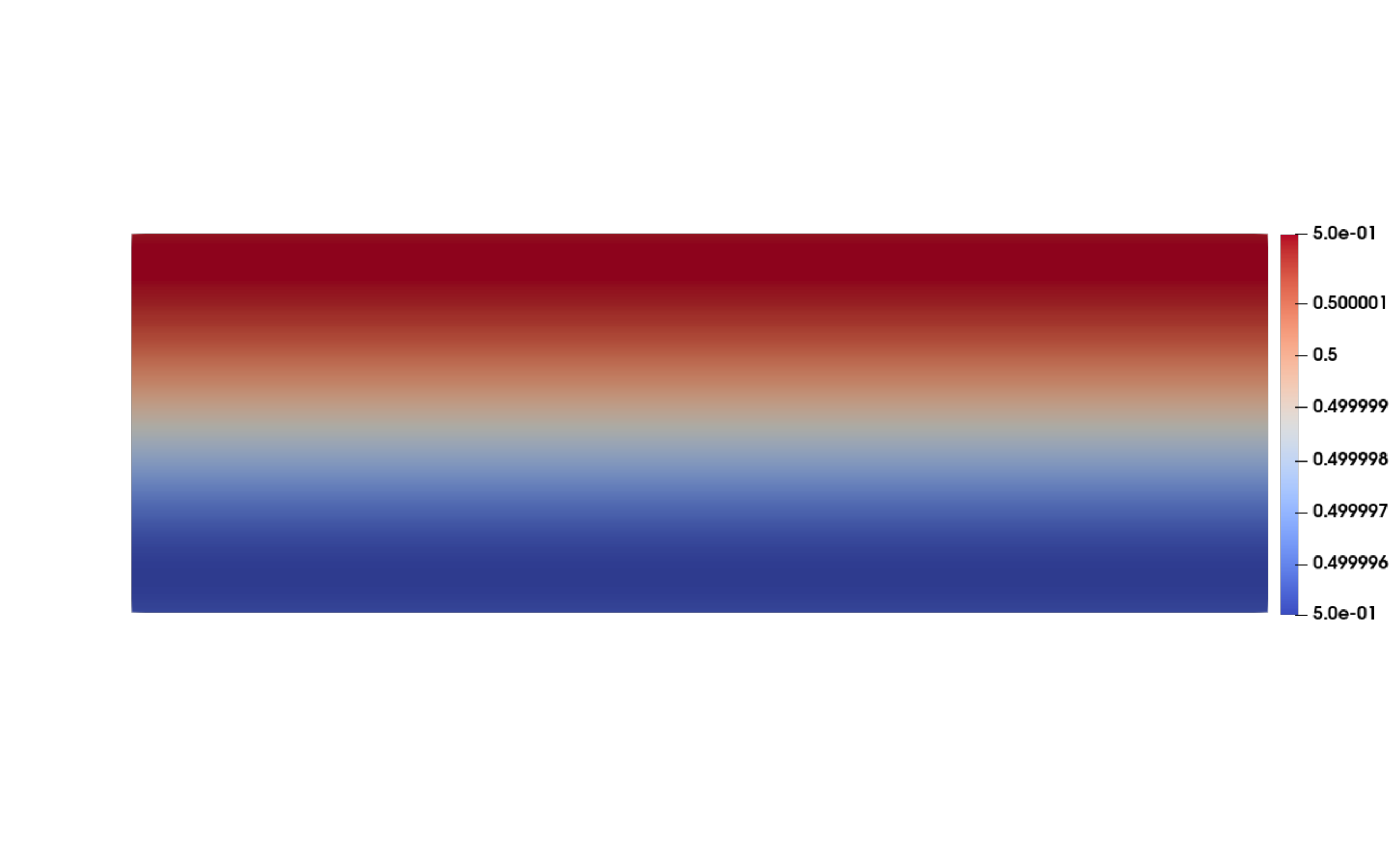}\\[-1cm]
\refstepcounter{figure}\label{phi_chi01323}
Figure \arabic{figure}: $\phi_h(t,\cdot)$ at times $t=0.00,10.00,50.00,500.00,1000.00$
\end{center}


\subsubsection{Convergence study} For the convergence study we choose $L_1=L_2=1$. The final time for the convergence study is $T=2$. The meshes are given by $\mathcal{T}_h^k = \mathcal{T}^{k,1}_h \cup \mathcal{T}^{k,2}_h$, where
\begin{align*}
    \mathcal{T}^{k,1}_h &= \left\{\mathrm{conv}\!\left\{\left.\left(\frac{\ell}{2^{k+3}},\frac{m}{2^{k+3}}\right), \left(\frac{\ell}{2^{k+3}},\frac{m+1}{2^{k+3}}\right), \left(\frac{\ell+1}{2^{k+3}},\frac{m+1}{2^{k+3}}\right)\right\}\right| (\ell,m)\in\{0,\dots,2^{k+3}-1\}^2\right\}, \\[2mm]
    \mathcal{T}^{k,2}_h &= \left\{\mathrm{conv}\!\left\{\left.\left(\frac{\ell}{2^{k+3}},\frac{m}{2^{k+3}}\right), \left(\frac{\ell+1}{2^{k+3}},\frac{m}{2^{k+3}}\right), \left(\frac{\ell+1}{2^{k+3}},\frac{m+1}{2^{k+3}}\right)\right\}\right| (\ell,m)\in\{0,\dots,2^{k+3}-1\}^2\right\},
\end{align*}
$k\in\{0,\dots,5\}$. In particular, $h_k=2^{-(k+3)}\sqrt{2}$, $k\in\{0,\dots,5\}$. The time steps are chosen as $(\deltat)_k 
= h_k/(40\sqrt{2})$, $k\in\{0,\dots,5\}$. $\gamma, s, \fatF, M, f, g, \eta$ are chosen as described in \eqref{parameters_and_functions}--\eqref{g_potential}. The Flory-Huggins interaction parameter is fixed as $\chi=\ln(3)/6$. The discrete initial data $\phi_h^0$ and $\fatu_h^0$ are obtained via interpolation of the continuous initial data
\begin{align*}
    \phi_0(\fatx) = 0.5 + 0.001\cos(6\pi x_1)\cos(2\pi x_2), \qquad \fatu_0(\fatx) = \bm{0}
\end{align*}
(cf. Remark \ref{remark_apriori_estimates}). For the errors 
\begin{align}
    e_{h}^{z}(k) = z_{h_5} - z_{h_k}, \qquad k\in\{0,...,5\}, \qquad z\in\{\phi,\mu,\bm{u},p\}, \notag 
\end{align}
we obtain the following results:

\begin{table}[ht]
\begin{center}
\begin{tabular}{c||c|c||c|c||c|c}
$k$ & $||e_{h}^{\phi}||_{L^{\infty}(H^1)}$ & EOC & $||e_{h}^{\bm{u}}||_{L^{\infty}(L^2)}$ & EOC & $||e_{h}^{\mu}||_{L^2(H^1)}$ & EOC \vphantom{\scalebox{1.5}{A}} \\ \hline 
0 & 8.212e-01 & --- & 8.653e-03 & --- & 2.698e-01 & --- \vphantom{\scalebox{1.5}{A}} \\ 
1 & 3.910e-01 & 1.071 & 4.711e-04 & 4.199 & 1.241e-01 & 1.121 \vphantom{\scalebox{1.5}{A}} \\ 
2 & 1.913e-01 & 1.031 & 5.400e-05 & 3.125 & 6.235e-02 & 0.993 \vphantom{\scalebox{1.5}{A}} \\ 
3 & 9.299e-02 & 1.041 & 7.201e-06 & 2.907 & 3.072e-02 & 1.021 \vphantom{\scalebox{1.5}{A}} \\ 
4 & 4.155e-02 & 1.162 & 9.926e-07 & 2.859 & 1.379e-02 & 1.156 \vphantom{\scalebox{1.5}{A}} 
\end{tabular}
\end{center}
\end{table}

\begin{table}[ht]
\begin{center}
\begin{tabular}{c||c|c||c|c}
$k$ & $||e_{h}^{p}||_{L^2(L^2)}$ & EOC & $||e_{h}^{\bm{u}}||_{L^2(H^1)}$ & EOC \vphantom{\scalebox{1.5}{A}} \\ \hline 
0 & 6.076e-03 & --- & 3.189e-01 & --- \vphantom{\scalebox{1.5}{A}} \\ 
1 & 1.402e-03 & 2.116 & 5.683e-02 & 2.488 \vphantom{\scalebox{1.5}{A}} \\ 
2 & 3.430e-04 & 2.031 & 1.530e-02 & 1.893 \vphantom{\scalebox{1.5}{A}} \\ 
3 & 8.329e-05 & 2.042 & 4.303e-03 & 1.830 \vphantom{\scalebox{1.5}{A}} \\ 
4 & 1.816e-05 & 2.197 & 1.066e-03 & 2.014 \vphantom{\scalebox{1.5}{A}} 
\end{tabular} \\[0.5cm]
\refstepcounter{table}\label{eoc_tables}
Table \arabic{table}: errors and experimental orders of convergence 
\end{center}
\end{table}

Here, the $k$-th EOC for a quantity $z\in\{\phi,\mu,\bm{u},p\}$ with respect to some norm $||\cdot||$ is computed from the errors $e_{h}^{z}(k-1)$, $e_{h}^{z}(k)$ via the formula
\begin{align}
    \mathrm{EOC}(k) = \log_2\!\left(\frac{||e_h^{z}(k-1)||}{||e_h^{z}(k)||}\right), \qquad k\in\{1,\dots,4\}. \notag
\end{align}

\section{Summary and outlook}
In this paper, we investigate the flow behaviour of dense melts of flexible and semiflexible ring polymers near confining walls using a hybrid multiscale modelling approach. At the microscopic level, we conduct MD simulations and utilise the Irving-Kirkwood formalism to compute an averaged stress tensor for input into our macroscopic model. This macroscopic model is based on a Cahn-Hilliard-Navier-Stokes system, incorporating both dynamic and no-slip boundary conditions. We perform numerical simulations of the macroscopic flow using a semi-implicit finite element method and provide rigorous proofs demonstrating the solvability and  energy stability of our numerical scheme. Phase segregation under flow between flexible and semiflexible rings, as observed in microscopic simulations, can be enforced in the macroscopic description by introducing effective attractive forces.
In the future, it will be interesting to investigate the relation between phase segregation and critical values of the Flory-Huggins interaction parameter more deeply. The goal will be to derive a data-driven study for the respective potential functions responsible for attractive forces. Having obtained energy stability estimates of finite element solutions, a further goal is to derive a rigorous error analysis of the numerical scheme using the relative energy technique.

\section*{Acknowledgements}
This project was funded by the Deutsche Forschungsgemeinschaft (DFG, German Research Foundation) in the framework to the collaborative research center ``Multiscale Simulation Methods for Soft-Matter Systems'' (TRR 146) under Project No. 233630050. M.L. gratefully acknowledges support of the Gutenberg Research College of the University of Mainz and the Mainz Institute of Multiscale
Modeling.

\bibliographystyle{plain}
\bibliography{references_phys,references_math}

\appendix
\section{Viscosity fitting}\label{sec_fitting}
In the following, we list the values for the parameters $\{\eta_{i,0}, \eta_{i,\infty}, a_{i,1}, a_{i,2}, a_{i,3}\}_{i\,=\,0,\dots,6}$ resulting from a least-squares fitting of the MD data shown in Figure \ref{fig1}a to the Carreau-Yasuda model \eqref{carreau-yasuda}. 

\begin{center}
\begin{tabular}{c|c|c||c|c|c}
    visc. & param. & value & visc. & param. & value \\ 
    part & & & part & &  \\ \hline
    $\eta_0$ & $\eta_{0,0}$\footnotemark & 2525.691875603924 & $\eta_1$ & $\eta_{1,0}$ & 424.982 \vphantom{$\overline{\overline{A}}$} \\
    & $\eta_{0,\infty}$ & 1.426364395540124 & & $\eta_{1,\infty}$ & 1.1046117170073273 \\
    & $a_{0,1}$ & $-0.21522922980127146$ & & $a_{1,1}$ & $-0.6487544203489711$ \\
    & $a_{0,2}$ & 59790.90467037572 & & $a_{1,2}$ & 11706.366708114712 \\
    & $a_{0,3}$ & 3.6355896735077615 & & $a_{1,3}$ & 1.0702121238758442 
\end{tabular}
\footnotetext{This value is a result of the fitting. Due to time constraints, it was not possible to determine it via MD simulations.}
         
\vspace{0.5cm}

\begin{tabular}{c|c|c||c|c|c}    
    visc. & param. & value & visc. & param. & value \\ 
    part & & & part & &  \\ \hline
    $\eta_3$ & $\eta_{3,0}$ & 40.5981 & $\eta_4$ & $\eta_{4,0}$ & 22.0876 \vphantom{$\overline{\overline{A}}$} \\
    & $\eta_{3,\infty}$ & 0.8298572446649171 & & $\eta_{4,\infty}$ & 0.7213002896042175  \\
    & $a_{3,1}$ & $-0.45749510263507903$ & & $a_{4,1}$ & $-0.34756459397139766$ \\
    & $a_{3,2}$ & 762.3394789734967 & & $a_{4,2}$ & 366.8389770486728 \\
    & $a_{3,3}$ & 1.2528181514824517 & & $a_{4,3}$ & 1.4842731261602664   
\end{tabular}
    
\vspace{0.5cm}

\begin{tabular}{c|c|c||c|c|c}
    visc. & param. & value & visc. & param. & value \\ 
    part & & & part & &  \\ \hline
    $\eta_5$ & $\eta_{5,0}$ & 15.2896 & $\eta_6$ & $\eta_{6,0}$ & 8.96233 \vphantom{$\overline{\overline{A}}$} \\
    & $\eta_{5,\infty}$ & 0.2981165425675458 & & $\eta_{6,\infty}$ & 1.7711189990810314  \\
    & $a_{5,1}$ & $-0.25999457614582533$ & & $a_{6,1}$ & $-1.1046100081035466$ \\
    & $a_{5,2}$ & 279.906851371738 & & $a_{6,2}$ & 23.252692619833642 \\
    & $a_{5,3}$ & 1.6164517304463373 & & $a_{6,3}$ & 0.9266624247139768    
\end{tabular}
    
\vspace{0.5cm}

\begin{tabular}{c|c|c}
    visc. & param. & value \\ 
    part & &  \\ \hline
    $\eta_7$ & $\eta_{7,0}$ & 5.98412 \vphantom{$\overline{\overline{A}}$} \\
    & $\eta_{7,\infty}$ & 0.9641347243968266 \\
    & $a_{7,1}$ & $-0.24083731713558557$ \\
    & $a_{7,2}$ & 41.67183783711502 \\
    & $a_{7,3}$ & 1.7990824755243957 
\end{tabular}

\end{center}

\section{Inequalities}\label{sec_inequalities}
Here, we list the inequalities used in Section \ref{sec_macroscopic}. We start with the inequalities attributed to Young and Hölder. Both inequalities are proven in the proof of \citep[Theorem 2.4]{Adams}.
\begin{theorem}[Young's inequality]
    Let $a,b\geq 0$ and $p,q\in(1,\infty)$ such that $1/p+1/q=1$. Then
    \begin{align}
        ab \leq \frac{a^p}{p} + \frac{b^q}{q}. \label{young}
    \end{align}
    Equality holds if and only if $a^p=b^q$.
\end{theorem}

\begin{theorem}[Hölder's inequality]\label{thm_hölder}
    Let $p\in[1,\infty)$ and $q=p/(p-1)$ if $p>1$ and $q=\infty$ if $p=1$. Then for every pair of functions $(f,g)\in L^p(\Omega)\times L^q(\Omega)$ the product $fg$ is an element of $L^1(\Omega)$ and, in particular,
    \begin{align}
        ||fg||_{L^1(\Omega)} \leq ||f||_{L^p(\Omega)}||g||_{L^q(\Omega)}. \label{hölder}
    \end{align}
\end{theorem}

In the following results, the constants also depend on the geometry of $\Omega$. Since we consider $\Omega$ to be fixed, we suppress this dependence in the notation. All of the results remain valid
if we replace $\Omega$ with an arbitrary Lipschitz domain $U\subset\reals^d$. The corresponding constants will then depend on the geometry of $U$. The first of the aforementioned results is the trace theorem. For its proof, we refer to \citep[Satz 6.15]{Dobrowolski}.
\begin{theorem}[trace theorem]\label{thm_trace}
    Let $q\in[1,4]$ if $d=3$ and $q\in[1,\infty)$ if $d=2$. There exists a unique continuous linear operator $\tr:H^1(\Omega)\to L^q(\partial\Omega)$ satisfying $\tr[u]=u|_{\partial\Omega}$ for all $u\in H^1(\Omega)\cap C(\overline{\Omega})$. In particular, there exists a constant $C_{\tr}=C_{\tr}(q,d)$ such that
    \begin{align}
        ||\tr[u]||_{L^q(\partial\Omega)} \leq C_{\tr}\,||u||_{H^1(\Omega)} \label{trace_inequality}
    \end{align}
    for all $u\in H^1(\Omega)$. The operator $\mathrm{tr}$ is called $\mathrm{trace}$ $\mathrm{operator}$ or simply $\mathrm{trace}$.
\end{theorem}

Next, we turn to the Poincaré inequalities. For the proof of inequalities \eqref{mean_poincare} and \eqref{poincare_bdry} we refer to \citep[Chapter II, Section 1.4]{Temam}. Inequality \eqref{poincare} is an immediate consequence of inequality \eqref{poincare_bdry}.
\begin{theorem}[Poincaré inequalities]~
    \begin{enumerate}
        \item[$\mathrm{(i)}$]{There exists a constant $C_P^{(1)}=C_P^{(1)}(d)$ such that for any $\mathcal{U}\in H^1(\Omega)\cup H^1(\Omega)^d$ 
        \begin{align}
            ||\mathcal{U}||^2_{H^1(\Omega)} \leq C_P^{(1)}\left(||\gradx \mathcal{U}||^2_{L^2(\Omega)} + \left|\int_\Omega \mathcal{U}\;\dx\right|^2\right).\label{mean_poincare}
        \end{align}
        } 
        \item[$\mathrm{(ii)}$]{There exists a constant $C_P^{(2)}=C_P^{(2)}(d)$ such that for any $\mathcal{V}\in H^1_{0,\mathrm{per}}(\Omega)\cup H^1_{0,\mathrm{per}}(\Omega)^d$ 
        \begin{align}
            ||\mathcal{V}||_{H^1(\Omega)} \leq C_P^{(2)}\,||\gradx \mathcal{V}||_{L^2(\Omega)}.\label{poincare}
        \end{align}
        }
        \item[$\mathrm{(iii)}$]{There exists a constant $C_P^{(3)}=C_P^{(3)}(d)$ such that for any $\mathcal{W}\in H^1(\Omega)\cup H^1(\Omega)^d$ 
        \begin{align}
            ||\mathcal{W}||^2_{H^1(\Omega)} \leq C_P^{(3)}\Big(||\gradx \mathcal{W}||^2_{L^2(\Omega)} + ||\mathrm{tr}[\mathcal{W}]||^2_{L^2(\partial\Omega)}\Big).\label{poincare_bdry}
        \end{align}
        }
    \end{enumerate}
\end{theorem}

We further need the following Sobolev inequalities. For inequality \eqref{sobolev} we refer to \citep[Theorem 4.12, Part I]{Adams}. Inequality \eqref{sobolev_bdry} is a consequence of inequality \eqref{sobolev} and the definition of the spaces $L^q(\Gamma)$, $H^1(\Gamma)$.
\begin{theorem}[Sobolev's inequality]
    Let $p,q$ be numbers such that
    \begin{align*}
        p\in\left\{
        \begin{array}{cl}
            [1,\infty) & \text{if $d=2$,} \\[2mm]
            [1, 6] & \text{if $d=3$,}  
        \end{array}\right.
        \qquad q\in[1,\infty).
    \end{align*}
    \begin{enumerate}
        \item[$(\mathrm{i})$]{There exists a constant $C_S^{(1)}=C_S^{(1)}(p,d)$ such that for any $\mathcal{U}\in H^1(\Omega)\cup H^1(\Omega)^d$
        \begin{align}
            ||\mathcal{U}||_{L^p(\Omega)}\leq C_S^{(1)}\,||\mathcal{U}||_{H^1(\Omega)}.\label{sobolev}
        \end{align}
        }
        \item[$(\mathrm{ii})$]{There exists a constant $C_S^{(2)}=C_S^{(2)}(q,d)$ such that for any $\mathcal{V}\in H^1(\Gamma)\cup H^1(\Gamma)^d$
        \begin{align}
            ||\mathcal{V}||_{L^q(\Gamma)}\leq C_S^{(2)}\,||\mathcal{V}||_{H^1(\Gamma)}.\label{sobolev_bdry}
        \end{align}
        }
    \end{enumerate}
    
\end{theorem}

\begin{theorem}[Korn's inequality]
    The following inequality holds for all $\fatu\in H^1_{0,\mathrm{per}}(\Omega)^d$: 
    \begin{align}
        ||\gradx\fatu||^2_{L^2(\Omega)} \leq 2\,||\fatD{\fatu}||^2_{L^2(\Omega)}. \label{korn}
    \end{align}
\end{theorem}

\begin{proof}
    We only sketch the proof. \\[2mm]
    \textbf{Step 1.} We first observe that for any $\fatu\in H^1_{0,\mathrm{per}}(\Omega)^d$ there exists a sequence $(\fatu_m)_{m\,\in\,\naturals}\subset C^\infty(\overline{\Omega})^d\cap H^1(\Omega)^d$ satisfying 
    \begin{align}
        ||\fatu_m-\fatu||_{H^1(\Omega)} \xlongrightarrow{\,m\,\to\,\infty\,} 0 \qquad \text{and} \qquad (\fatu_m,\gradx\fatu_m)(\fatx) = (\fatu_m,\gradx\fatu_m)(\fatx+L_j\fate_j) \notag 
    \end{align}
    for all $\fatx\in \partial\Omega\cap \{x_j=0\}$, $j\in\{1,\dots,d-1\}$, and all $m\in\naturals$. To see this, we define the function $\fatv\in L^2(\reals^d)^d$ via
    \begin{align}
        \fatv(\fatx) = \left\{\begin{array}{cl}
            \fatu\!\left(x_1-\left\lfloor\dfrac{x_1}{L_1}\right\rfloor L_1, \dots, x_d-\left\lfloor\dfrac{x_d}{L_d}\right\rfloor L_d\right) & \text{if $\fatx\in(-\bm{L},2\bm{L})\backslash \partial\Omega$,} \\[4mm]
            \bm{0} & \text{else.} 
        \end{array}
        \right. \notag
    \end{align}
    Clearly, $\fatv\in H^1((z_1L_1, (z_1+1)L_1\times\dots\times(z_dL_d, (z_d+1)L_d))^d$ for all $\fatz\in\{-1,0,1\}^d$. Using Green's formula for Sobolev functions (cf. \citep[Theorem 3.1.1]{Necas}) we deduce that the obvious (square integrable) candidates are indeed weak derivatives of $\fatv$ in $(-\bm{L},2\bm{L})$ whence $\fatv\in H^1((-\bm{L},2\bm{L}))^d$. Since $\fatv$ is an extension of $\fatu$, the desired sequence can be obtained by mollification of $\fatv$ (cf. \citep[Lemma 3.16]{Adams}). \\[2mm]
    \textbf{Step 2.} Using the properties of the sequence $(\fatu_m)_{m\,\in\,\naturals}$ and Green's formula, we deduce that
    \begin{align}
        2\,||\fatD{\fatu_m}||^2_{L^2(\Omega)} &= ||\gradx\fatu_m||^2_{L^2(\Omega)} + \sum_{k,\ell\,=\,1}\int_\Omega \frac{\partial u_{m,k}}{\partial x_\ell}\,\frac{\partial u_{m,\ell}}{\partial x_k}\;\dx \notag \\[2mm]
        &= ||\gradx\fatu_m||^2_{L^2(\Omega)} - \sum_{k,\ell\,=\,1}\int_\Omega \frac{\partial^2 u_{m,k}}{\partial x_k \partial x_\ell}\,u_{m,\ell}\;\dx \notag \\[2mm]
        &= ||\gradx\fatu_m||^2_{L^2(\Omega)} + \sum_{k,\ell\,=\,1}\int_\Omega \frac{\partial u_{m,k}}{\partial x_k}\,\frac{\partial u_{m,\ell}}{\partial x_\ell}\;\dx \geq ||\gradx\fatu_m||^2_{L^2(\Omega)} \notag 
    \end{align}
    for all $m\in\naturals$. Sending $m\to\infty$, we obtain \eqref{korn}.
\end{proof}

\section{Triangulations and inf-sup stability of \texorpdfstring{$P_2-P_1$}{P2-P1} finite elements}
\label{sec_inf-sup}
We define admissible triangulations of $\Omega$.
\begin{definition}\label{def_triangulation}~
    \begin{enumerate}
        \item[(a)]{A finite set $\mathcal{T}$ of closed $d$-simplices is called an \textit{admissible triangulation of $\Omega$} if it has the following properties:
        \begin{itemize}
            \item[(1)]{$\overline{\Omega} = \bigcup_{K\,\in\,\mathcal{T}} K$.
            } 
            \item[(2)]{If $K,L\in\mathcal{T}$, then the intersections $K\cap L$ and $\{(L_j\fate_j+K)\cap L\}_{j\,\in\,\{1,\dots,d-1\}}$ are either empty, a common vertex, a common edge or a common face.
            }
            \item[(3)]{If $d=2$, every $K\in\mathcal{T}$ has at most one edge on $\Gamma$. If $d=3$, every $K\in\mathcal{T}$ has at least three edges that are not contained in $\Gamma$.
            }
        \end{itemize}
        The number $h=\max\{h_K\,|\,K\in\mathcal{T}\}$ characterizing the mesh size is called \textit{mesh parameter}. 
        }
        \item[(b)]{A family $\{\mathcal{T}_i\}_{i\,\in\,I}$ of admissible triangulations of $\Omega$ is called \textit{uniform} if there exists a number $\varepsilon>0$ such that $\varepsilon h_i \leq r_K\leq h_K \leq h_i$, whenever $K\in\mathcal{T}_i$ for some $i\in I$. Here, $r_K$ is the radius of the largest ball contained in $K$, $h_K = \mathrm{diam}(K)$ and $h_i=\max\{h_K\,|\,K\in\mathcal{T}_i\}$. 
        }
    \end{enumerate}
\end{definition}
\begin{remark}
    Often, families $\{\mathcal{T}_i\}_{i\,\in\,I}$ of triangulations are chosen in such a way that their members are uniquely characterized by their mesh parameter. Such families are typically indexed by their mesh parameters. 
\end{remark}
The inf-sup stability of $P_2-P_1$ finite elements is proven in \citep[Chapter VI.6]{Brezzi}.
\begin{theorem}[Inf-sup stability of $P_2-P_1$ finite elements]\label{thm_inf-sup}
    Let $\{\mathcal{T}_h\}_{h\,\in\,(0,H]}$ be a uniform family of admissible triangulations of $\Omega$. Then there exists a constant $\beta>0$ such that 
    \begin{align}
        \sup_{\fatv\,\in\,\mathcal{V}_h^d\backslash\{\bm{0}\}}\left\{\int_\Omega\frac{\divx(\fatv)\,q}{||\fatv||_{H^1(\Omega)}}\;\dx \right\}\geq \beta\left|\left|q-\frac{1}{|\Omega|}\int_\Omega q\;\dx\right|\right|_{L^2(\Omega)}, \label{inf-sup}
    \end{align}
    whenever $q\in \mathcal{Q}_h$ for some $h\in(0,H]$.
\end{theorem}

\section{The properties of the map \texorpdfstring{$\mathcal{S}_h$}{Sh}}\label{sec_sh}
Before we turn to the proof of the properties of the map $\mathcal{S}_h$ appearing in Section \ref{sec_proof}, we recall two classical results. The first of them is the equivalence of norms in finite-dimensional vector spaces; see \citep[Satz 2.5]{Dobrowolski}.
\begin{defandthm}\label{thm_equivalence_of_norms}
    Let $X$ be a vector space. Two norms $||\cdot||_1, ||\cdot||_2$ on $X$ are called $\mathrm{equivalent}$ if there exist constants $0<m_*\leq m^*$ such that
    \begin{align}
        m_*||x||_1 \leq ||x||_2 \leq m^*||x||_1 \qquad \text{for all $x\in X$.} \notag
    \end{align}
    If $X$ is finite-dimensional, then all norms on $X$ are equivalent.
\end{defandthm}

The second result is the Riesz representation theorem; see \citep[Satz 2.25]{Dobrowolski}.
\begin{theorem}[Riesz representation theorem]\label{riesz}
    Let $(H,\langle\cdot,\cdot\rangle_H)$ be a (real) Hilbert space. For any continuous linear map $f:H\to\reals$ there exists a unique element $x\in H$ such that
    \begin{align*}
        f(y) = \langle x, y\rangle_H \qquad \text{for all $y\in H$.}
    \end{align*}
    In particular, 
    \begin{align*}
        ||x||_H = \sup_{y\,\in\,H}\left\{\frac{|f(y)|}{||y||_H}\right\}.
    \end{align*}
\end{theorem}

\noindent \textbf{Proof of the properties of $\mathcal{S}_h$.} We choose an arbitrary
$(\mu_h,\Delta\phi_h,\fatu_h,p_h,r_h)\in X_h$ and consider the linear functionals 
\begin{align}
    f_6:X_h\to\reals, \qquad (\psi_h,w_h,\fatv_h,q_h,s_h) \mapsto f_6(\psi_h,w_h,\fatv_h,q_h,s_h) &= -\big\langle f'_{\mathrm{vex}}\!\left(\phi_h^n+\deltat\Delta\phi_h\right), w_h\big\rangle_{\Omega}, \notag \\[2mm]
    f_{13}:X_h\to\reals, \qquad (\psi_h,w_h,\fatv_h,q_h,s_h) \mapsto f_{13}(\psi_h,w_h,\fatv_h,q_h,s_h) &= \frac{1}{4}\,\big\langle[(\fatu_h + \fatu_h^n)\cdot\gradx]\fatv_h, \fatu_h\big\rangle_\Omega \notag
\end{align}
as examples. Using Hölder's inequality \eqref{hölder}, Korn's inequality \eqref{korn} and Assumption (\hyperlink{A5}{A5}), we deduce that
\begin{align}
    |f_6(\psi_h,w_h,\fatv_h,q_h,s_h)| &\leq ||f'_{\mathrm{vex}}(\phi_h^n+\deltat\Delta\phi_h)||_{L^2(\Omega)}||w_h||_{L^2(\Omega)} \notag \\[2mm]
    &\leq ||f'_{\mathrm{vex}}(\phi_h^n+\deltat\Delta\phi_h)||_{L^2(\Omega)}||(\psi_h,w_h,\fatv_h,q_h,s_h)||_{X_h}, \notag \\[2mm]
    |f_{13}(\psi_h,w_h,\fatv_h,q_h,s_h)| &\leq ||\fatu_h+\fatu_h^n||_{L^4(\Omega)}||\fatu_h||_{L^4(\Omega)}||\gradx\fatv_h||_{L^2(\Omega)} \leq \sqrt{2}\,||\fatu_h+\fatu_h^n||_{L^4(\Omega)}||\fatu_h||_{L^4(\Omega)}||\fatD{\fatv_h}||_{L^2(\Omega)} \notag \\[2mm]
    &\leq (2/\underline{\eta})^{1/2}\,||\fatu_h+\fatu_h^n||_{L^4(\Omega)}||\fatu_h||_{L^4(\Omega)}||(\psi_h,w_h,\fatv_h,q_h,s_h)||_{X_h}. \notag
\end{align}
Consequently, $f_6$ and $f_{13}$ are bounded and hence continuous and the Riesz representation theorem yields the existence and uniqueness of $\mathcal{S}_h^{(6)}(\mu_h,\Delta\phi_h,\fatu_h,p_h,r_h),\mathcal{S}_h^{(13)}(\mu_h,\Delta\phi_h,\fatu_h,p_h,r_h)\in X_h$ such that 
\begin{align}
    f_6(\psi_h,w_h,\fatv_h,q_h,s_h) &= -\big\langle f'_{\mathrm{vex}}\!\left(\phi_h^n+\deltat\Delta\phi_h\right), w_h\big\rangle_{\Omega} = \left\langle\mathcal{S}_h^{(6)}(\mu_h,\Delta\phi_h,\fatu_h,p_h,r_h), (\psi_h,w_h,\fatv_h,q_h,s_h)\right\rangle_{X_h}, \notag \\[2mm]
    f_{13}(\psi_h,w_h,\fatv_h,q_h,s_h) &= \frac{1}{4}\,\big\langle[(\fatu_h + \fatu_h^n)\cdot\gradx]\fatv_h, \fatu_h\big\rangle_\Omega = \left\langle\mathcal{S}_h^{(13)}(\mu_h,\Delta\phi_h,\fatu_h,p_h,r_h), (\psi_h,w_h,\fatv_h,q_h,s_h)\right\rangle_{X_h}\notag 
\end{align}
for all $(\psi_h,w_h,\fatv_h,q_h,s_h)\in X_h$. To prove the continuous dependence of $\mathcal{S}_h^{(6)}, \mathcal{S}_h^{(13)}$ on $\fatx_h = (\mu_h,\Delta\phi_h,\fatu_h,p_h,r_h)$, we choose an arbitrary sequence $\{\fatx_{h,k}\}_{k\,\in\,\naturals} = \{(\mu_{h,k},\Delta\phi_{h,k},\fatu_{h,k},p_{h,k},r_{h,k})\}_{k\,\in\,\naturals}\subset X_h$ with limit $\fatx_h$. Focusing on $\mathcal{S}_h^{(6)}$, we set $z_{h,k}=\phi_h^n + \deltat\Delta\phi_{h,k}$ and $z_{h}=\phi_h^n + \deltat\Delta\phi_{h}$. Using Hölder's inequality \eqref{hölder}, we deduce that
\begin{align}
    \left|\left\langle\mathcal{S}_h^{(6)}(\fatx_{h,k}) - \mathcal{S}_h^{(6)}(\fatx_h), \faty_h\right\rangle_{X_h}\right| 
    &\leq |\Omega|^{1/2}||f'_{\mathrm{vex}}(z_{h,k})-f'_{\mathrm{vex}}(z_h)||_{L^\infty(\Omega)}||w_h||_{L^2(\Omega)} \notag \\[2mm]
    &\leq |\Omega|^{1/2}||f'_{\mathrm{vex}}(z_{h,k})-f'_{\mathrm{vex}}(z_h)||_{C(\overline{\Omega})}||\faty_h||_{X_h} \label{est_Sh6}
\end{align}
for all $\faty_h = (\psi_h,w_h,\fatv_h,q_h,s_h)\in X_h$. Due to the equivalence of all norms in $X_h$ (see Definition and Theorem \ref{thm_equivalence_of_norms}) we have
\begin{align}
    ||z_{h,k}-z_h||_{C(\overline{\Omega})} = 
    \deltat\,||\Delta\phi_{h,k}-\Delta\phi_h||_{C(\overline{\Omega})} \to 0 \label{delta_phi_convergence}
\end{align}
as $k\to\infty$. In particular, there exists a compact set $K_h\subset\reals$ such that
\begin{align}
    z_h(\overline{\Omega})\cup\bigcup_{k\,\in\,\naturals}z_{h,k}(\overline{\Omega})\subset K_h. \notag
\end{align}
Moreover, since $f'_{\mathrm{vex}}$ is continuous on $\reals$, it is uniformly continuous on $K_h$ and thus it follows from \eqref{delta_phi_convergence} that
\begin{align}
    ||f'_{\mathrm{vex}}(z_{h,k})-f'_{\mathrm{vex}}(z_h)||_{C(\overline{\Omega})} \to 0 \notag
\end{align}
as $k\to\infty$. Combining this fact with the Riesz representation theorem and \eqref{est_Sh6}, we deduce that
\begin{align}
    \left|\left|\mathcal{S}_h^{(6)}(\fatx_{h,k}) - \mathcal{S}_h^{(6)}(\fatx_h)\right|\right|_{X_h} &= \sup_{\faty_h\,\in\,X_h\backslash\{\bm{0}\}}\left\{\frac{1}{||\faty_h||_{X_h}}\left|\left\langle\mathcal{S}_h^{(6)}(\fatx_{h,k}) - \mathcal{S}_h^{(6)}(\fatx_h), \faty_h\right\rangle_{X_h}\right|\right\} \notag \\[2mm]
    &\leq |\Omega|^{1/2}||f'_{\mathrm{vex}}(z_{h,k})-f'_{\mathrm{vex}}(z_h)||_{C(\overline{\Omega})} \to 0 \notag
\end{align}
as $k\to\infty$. This proves the continuity of $\mathcal{S}_h^{(6)}$. 
To prove the continuity of $\mathcal{S}_h^{(13)}$, we use Hölder's inequality \eqref{hölder}, Sobolev's inequality \eqref{sobolev}, Poincaré's inequality \eqref{poincare}, Korn's inequality \eqref{korn} and Assumption (\hyperlink{A5}{A5}) to deduce that
\begin{align}
    &\left|\left\langle\mathcal{S}_h^{(13)}(\fatx_{h,k}) - \mathcal{S}_h^{(13)}(\fatx_h), \faty_h\right\rangle_{X_h}\right| \notag \\[2mm]
    &\quad = \frac{1}{4}\,\Big|\big\langle[(\fatu_{h,k}+\fatu_h^n)\cdot\gradx]\fatv_{h}, \fatu_{h,k}\big\rangle_\Omega - \big\langle[(\fatu_{h}+\fatu_h^n)\cdot\gradx]\fatv_{h}, \fatu_{h}\big\rangle_\Omega\Big| \notag \\[2mm]
    &\quad = \frac{1}{4}\,\Big|\big\langle[(\fatu_{h,k}-\fatu_h)\cdot\gradx]\fatv_{h}, \fatu_{h,k}\big\rangle_\Omega + \big\langle(\fatu_{h}^n\cdot\gradx)\fatv_h, \fatu_{h,k}-\fatu_h\big\rangle_\Omega + \big\langle(\fatu_{h,k}\cdot\gradx)\fatv_h, \fatu_{h,k}-\fatu_h\big\rangle_\Omega\Big| \notag \\[2mm]
    &\quad \leq \left(||\fatu_{h,k}||_{L^4(\Omega)} + ||\fatu_{h}^n||_{L^4(\Omega)}\right)||\fatu_{h,k}-\fatu_h||_{L^4(\Omega)}||\gradx\fatv_h||_{L^2(\Omega)} \notag \\[2mm]
    &\quad \leq \left(||\fatu_{h,k}-\fatu_h||_{L^4(\Omega)} + ||\fatu_{h}||_{L^4(\Omega)} + ||\fatu_{h}^n||_{L^4(\Omega)}\right)||\fatu_{h,k}-\fatu_h||_{L^4(\Omega)}||\gradx\fatv_h||_{L^2(\Omega)} \notag \\[2mm]
    &\quad \leq (2/\underline{\eta})C_S^{(1)}C_P^{(2)}\left((2/\underline{\eta})^{1/2}C_S^{(1)}C_P^{(2)}||\fatx_{h,k}-\fatx_h||_{X_h} + ||\fatu_{h}||_{L^4(\Omega)} + ||\fatu_{h}^n||_{L^4(\Omega)}\right)||\fatx_{h,k}-\fatx_h||_{X_h}||\faty_h||_{X_h} \notag 
\end{align}
for all $\faty_h=(\psi_h,w_h,\fatv_h,q_h,s_h)\in X_h$. It follows that
\begin{align}
    \left|\left|\mathcal{S}_h^{(13)}(\fatx_{h,k}) - \mathcal{S}_h^{(13)}(\fatx_h)\right|\right|_{X_h} &= \sup_{\faty_h\,\in\,X_h\backslash\{\bm{0}\}}\left\{\frac{1}{||\faty_h||_{X_h}}\left|\left\langle\mathcal{S}_h^{(13)}(\fatx_{h,k}) - \mathcal{S}_h^{(13)}(\fatx_h), \faty_h\right\rangle_{X_h}\right|\right\} \to 0 \notag
\end{align}
as $k\to\infty$. The remaining terms on the right-hand side of \eqref{reformulated_problem} can be treated similarly. \hfill $\Box$

\section{Schaefer's fixed point theorem}\label{sec_schaefer}
We recall Schaefer's fixed point theorem for finite-dimensional vector spaces. It is an immediate consequence of the general case proven in \citep{Schaefer}. 
\begin{theorem}[Schaefer's fixed point theorem]\label{schaefer}
    Let $X$ be a finite-dimensional vector space and $\mathcal{S}:X\to X$ a continuous mapping. If the set 
    \[\big\{\fatx\in X\,\big|\,\fatx = \lambda \mathcal{S}(\fatx) \;\text{for some $\lambda\in[0,1]$}\big\}\] 
    is bounded, then $\mathcal{S}$ has a fixed point.
\end{theorem}

\end{document}